\documentclass[11pt,reqno]{amsart}

\usepackage[T1]{fontenc}

\usepackage{amsmath}								%math
\usepackage{amssymb}
\usepackage{amsthm}
\usepackage{amscd}
\usepackage{amsfonts}
\usepackage{stmaryrd}
\usepackage[all]{xy}

\usepackage{euler}

\usepackage{extarrows}

\usepackage[colorlinks, citecolor = blue, linkcolor = blue]{hyperref}
\usepackage{color}
\usepackage{tikz}									
\usetikzlibrary{matrix}
\usetikzlibrary{patterns}
\usetikzlibrary{matrix}
\usetikzlibrary{positioning}
\usetikzlibrary{decorations.pathmorphing}
\usetikzlibrary{cd}
\usetikzlibrary{intersections,calc}

\usepackage{fullpage}
\usepackage[shortlabels]{enumitem}

\newtheorem{maintheorem}{Theorem}								\newtheorem{maincorollary}{Corollary}									
	
\newtheorem{theorem}{Theorem}[section]
\newtheorem{lemma}[theorem]{Lemma}
\newtheorem{proposition}[theorem]{Proposition}
\newtheorem{corollary}[theorem]{Corollary}

\theoremstyle{definition}
\newtheorem{definition}[theorem]{Definition}
\newtheorem{example}[theorem]{Example}

\newtheorem{remark}[theorem]{Remark}

\newcommand{\comment}[1]{}

\newcommand{\R}{\mathbb{R}}
\newcommand{\Rbar}{\overline{\mathbb{R}}}
\newcommand{\Z}{\mathbb{Z}}
\newcommand{\ZZ}{\mathbb{Z}}
\newcommand{\Q}{\mathbb{Q}}

\newcommand{\N}{\mathbb{N}}
\newcommand{\PP}{\mathbb{P}}

\newcommand{\QQ}{\mathbb{Q}}
\newcommand{\RR}{\mathbb{R}}

\newcommand{\G}{\mathbb{G}}
\newcommand{\GG}{\mathbb{G}}

\renewcommand{\SS}{\mathbb{S}}

\newcommand{\mcF}{\mathcal{F}}
\newcommand{\mcC}{\mathcal{C}}

\newcommand{\Mbar}{\overline{M}}

\newcommand{\calA}{\mathcal{A}}

\newcommand{\calC}{\mathcal{C}}
\newcommand{\calD}{\mathcal{D}}

\newcommand{\calF}{\mathcal{F}}

\newcommand{\calM}{\mathcal{M}}

\newcommand{\calW}{\mathcal{W}}
\newcommand{\calX}{\mathcal{X}}
\newcommand{\calY}{\mathcal{Y}}
\newcommand{\calMbar}{\overline{\mathcal{M}}}

\newcommand{\bfC}{\mathbf{C}}
\newcommand{\bfD}{\mathbf{D}}

\def\RPC{\mathbf{RPC}}
\def\RPCC{\mathbf{RPCC}}
\def\RRPC{\mathbf{PC}}
\def\Top{\mathbf{Top}}
\def\ShpMon{\mathbf{ShpMon}}
\def\ComMon{\mathbf{ComMon}}

\DeclareMathOperator{\Pic}{Pic}
\DeclareMathOperator{\Spec}{Spec}
\DeclareMathOperator{\Hom}{Hom}

\DeclareMathOperator{\Aut}{Aut}

\def\trop{\mathrm{trop}}
\def\an{\mathrm{an}}
\DeclareMathOperator{\val}{val}

\DeclareMathOperator{\LOG}{LOG}

\DeclareMathOperator{\pr}{pr}

\DeclareMathOperator{\Span}{Span}
\DeclareMathOperator{\Cone}{Cone}

\DeclareMathOperator{\ACP}{ACP}

\newcommand{\et}{\mathrm{\acute{e}t}}
\newcommand{\Et}{\mathrm{\acute{E}t}}

\title[A moduli stack of tropical curves]{A moduli stack of tropical curves}

\author{Renzo Cavalieri}
\address{Department of Mathematics, Colorado State University, Fort Collins, Colorado 80523-1874}
\email{\href{mailto:renzo@math.colostate.edu}{renzo@math.colostate.edu}}
\thanks{} 

\author{Melody Chan}
\address{Department of Mathematics, Brown University, Providence, Rhode Island 02912}
\email{\href{mailto:melody_chan@brown.edu}{melody\_chan@brown.edu}}
\thanks{} 

\author{Martin Ulirsch}
\address{Institut f\"ur Mathematik, Goethe-Universit\"at Frankfurt, 60325 Frankfurt am Main, Germany}
\email{\href{mailto:ulirsch@math.uni-frankfurt.de}{ulirsch@math.uni-frankfurt.de}}
\thanks{} 

\author{Jonathan Wise}
\address{Department of Mathematics, University of Colorado, Boulder, Boulder, Colorado 80309-0395}
\email{\href{mailto:jonathan.wise@math.colorado.edu}{jonathan.wise@math.colorado.edu}}
\thanks{} 


\begin{document}

\begin{abstract} 

We contribute to the foundations of tropical geometry with a view towards formulating tropical moduli problems, and with the moduli space of curves as our main
example.  We
propose a moduli functor for the moduli space of curves and show that it is representable by a geometric stack over the category of rational polyhedral cones. In this framework the natural forgetful morphisms between moduli spaces of curves with marked points function as universal curves.

Our approach to tropical geometry permits tropical moduli problems---moduli of curves or otherwise---to be extended to logarithmic schemes.  We use this to construct a smooth tropicalization morphism from the moduli space of algebraic curves to the moduli space of tropical curves, and we show that this morphism commutes with all of the tautological morphisms.

\end{abstract}

\maketitle

\setcounter{tocdepth}{1}
\tableofcontents

%%%%%%%%%%%%%%%%%%%%%%%%%%%%%%%%%%%%%%%%%%%%%%%%%%%%%%
\part*{Introduction}

Let $g$ and $n$ be non-negative integers such that $2g-2+n>0$. The moduli spaces $M_{g,n}^{\trop}$ of stable tropical curves of genus $g$ with $n$ marked points, and variants of it, are some of the most-studied objects in tropical geometry (see e.g.  \cite{Caporaso_tropicalmoduli, ChanMeloViviani_tropicalmoduli, Caporaso_tropicalmoduliII, Chan_lectures} for survey papers). In particular, we refer the reader to \cite{Mikhalkin_ICM}, \cite{Mikhalkin_Gokova}, \cite{GathmannKerberMarkwig_tropicalfans}, and \cite{GathmannMarkwig_tropicalKontsevich} for the case of genus $g=0$, as well as to \cite{CaporasoViviani_tropicalTorelli}, \cite{BrannettiMeloViviani_tropicalTorelli}, \cite{Chan_tropicalTorelli}, and \cite{Viviani_tropvscompTorelli} for the higher genus case and the connections to the  tropical Torelli map. A new chapter in these developments was begun by \cite{AbramovichCaporasoPayne_tropicalmoduli}, where the authors recast the foundations of the subject by embedding $M_{g,n}^{\trop}$ as a non-Archimedean analytic skeleton of the Berkovich analytic stack $\calM_{g,n}^{\an}$ using Thuillier's non-Archimedean geometry of toroidal embeddings \cite{Thuillier_toroidal}. From there, the focus of the subject has moved in two directions: towards a deeper understanding of the geometry of $M_{g,n}^{\trop}$ with a view towards computations of the top-weight cohomology of $\mathcal{M}_{g,n}$ (see e.g. \cite{ChanGalatiusPayne_tropicalmoduliII, Chan_topologyM02, ChanGalatiusPayne_Mgn,ChanFaberGalatiusPayne}), and towards  generalizations to other moduli spaces, often with enumerative applications in mind (see  e.g. \cite{CavalieriMarkwigRanganathan_tropicalHurwitz, CavalieriHampeMarkwigRanganathan_tropicalHassett, Ulirsch_tropicalHassett, Yu_tropstablemaps, Ranganathan_ratcurtorvar&nonArch}). 

All of the aforementioned perspectives on $M_{g,n}^{\trop}$ are essentially $1$-categorical; that is, they regard $M_{g,n}^{\trop}$ as a set with additional structure.  %regard it as a set, albeit sometimes with additional structure.
Without accounting for automorphisms from a $2$-categorical point of view, these moduli spaces do not admit universal curves.  In particular, the tropical forgetful map $\pi_{g,n+1}\mathrel{\mathop:}M_{g,n+1}^{\trop}\rightarrow M_{g,n}^{\trop}$, as constructed in \cite[Section 8]{AbramovichCaporasoPayne_tropicalmoduli}, cannot be the universal curve. The problem lies in the fact that the topological preimage $\pi_{g,n+1}^{-1}\big([\Gamma]\big)$ of a point $[\Gamma]$ in $M_{g,n}^{\trop}$ is not the tropical curve $\Gamma$ itself, but rather  the quotient $\Gamma/\!\Aut(\Gamma)$ (see Figure \ref{figure_M12}). In the case $g=0$, when no stack structure is necessary due to the absence of nontrivial automorphisms, this is not a problem (see \cite{FrancoisHampe_universalfamilies}).

\begin{figure}[h]\begin{tikzpicture}

\fill[pattern color=gray, pattern=north west lines] (0,0) -- (2.8,0) -- (2.8,2.8) -- (0,0);
\fill[pattern color=gray, pattern=north east lines] (0,0) -- (2.8,0) -- (2.8,-2.8) -- (0,-2.8) -- (0,0);
\draw [->](0,0) -- (3,0);
\draw[dashed, ->] (0,0) -- (3,3);
\draw [->](0,0) -- (0,-3);
\fill (0,0) circle (0.05);

\begin{scope}[scale=0.8, shift={(0.5,0)}]
\draw[thick, gray] (4,0) circle (0.4);
\draw[gray] (4.4,0) -- (5,0.5);
\draw[gray] (4.4,0) -- (5,-0.5);
\fill[gray] (4.4,0) circle (0.05);
\node[gray, draw=none, label={[gray]0:$\scriptstyle 1$}] (mark1) at (4.7,.5) {};
\node[gray, draw=none, label={[gray]0:$\scriptstyle 2$}] (mark2) at (4.7,-.5) {};
\end{scope}

% half quadrant
\begin{scope}[scale=1, shift={(0.6,0)}]
\draw[gray] (2.8,1.7) -- (3.3,1.7);
\draw[gray] (3.7,1.7) circle (0.4);
\draw[gray] (4.1,1.7) -- (4.6,1.7);
\node[gray, draw=none, label={[gray]180:$\scriptstyle 1$}] (mark1) at (3,1.7) {};
\node[gray, draw=none, label={[gray]0:$\scriptstyle 2$}] (mark2) at (4.4,1.7) {};
\fill[gray] (3.3,1.7) circle (0.05);
\fill[gray] (4.1,1.7) circle (0.05);
\end{scope}

% quadrant
\begin{scope}[scale=0.9, shift={(1.5,-1.5)}]
\draw[gray] (2,-2) circle (0.4);
\draw[gray] (2.4,-2) -- (3,-2);
\draw[gray] (3,-2) -- (3.6,-1.5);
\draw[gray] (3,-2) -- (3.6,-2.5);
\node[gray, draw=none, label={[gray]0:$\scriptstyle 1$}] (mark1) at (3.5,-1.5) {};
\node[gray, draw=none, label={[gray]0:$\scriptstyle 2$}] (mark2) at (3.5,-2.5) {};
\fill[gray] (2.4,-2) circle (0.05);
\fill[gray] (3,-2) circle (0.05);
\end{scope}

% down ray
\begin{scope}[scale=0.8, shift={(-1.1,-.7)}]
\draw[gray] (-0.3,-4) -- (0.3,-4);
\draw[gray] (0.3,-4) -- (0.9,-3.5);
\draw[gray] (0.3,-4) -- (0.9,-4.5);
\node[gray] at (-0.5,-3.8) {$1$};
\node[gray, draw=none, label={[gray]0:$\scriptstyle 1$}] (mark1) at (0.8,-3.5) {};
\node[gray, draw=none, label={[gray]0:$\scriptstyle 2$}] (mark2) at (0.8,-4.5) {};
\fill[gray] (-0.3,-4) circle (0.05);
\fill[gray] (0.3,-4) circle (0.05);
\end{scope}

% cone point
\begin{scope}[scale=0.8, shift={(-.4,0)}]
\draw[gray] (-1,0) -- (-0.4,0.5);
\draw[gray] (-1,0) -- (-0.4,-0.5);
\node[gray, draw=none, label={[gray]0:$\scriptstyle 1$}] (mark1) at (-0.5,.5) {};
\node[gray, draw=none, label={[gray]0:$\scriptstyle 2$}] (mark2) at (-0.5,-.5) {};
\fill[gray] (-1,0) circle (0.05);
\node[gray] at (-1.2, 0.3) {$1$};
\end{scope}

% M11 ray
\begin{scope}[scale=1, shift={(0.0,0)}]
\draw[->] (0,-6) -- (3,-6);
\fill (0,-6) circle (0.05);
\end{scope}

\begin{scope}[scale=1, shift={(0,0)}]
\fill[gray] (-1,-6) circle (0.05);
\node[gray, draw=none, label={[gray]0:$\scriptstyle 1$}] (mark1) at (-0.6,-6) {};
\draw[gray] (-1,-6) -- (-0.4,-6);
\node[gray] at (-1.2, -5.7) {$1$};
\end{scope}

% M11 type
\begin{scope}[scale=1, shift={(0.0,0)}]
\draw[gray] (4,-6) circle (0.4);
\draw[gray] (4.4,-6) -- (5,-6);
\node[gray, draw=none, label={[gray]0:$\scriptstyle 1$}] (mark1) at (4.8,-6) {};
\fill[gray] (4.4,-6) circle (0.05);
\end{scope}

\draw[->] (1.5,-4) -- (1.5,-5);
\node at (2,-4.5) {$\pi_{1,2}^{\trop}$};
\node at (6,-1) {$M_{1,2}^{\trop}$};
\node at (6,-6) {$M_{1,1}^{\trop}$};

% The fiber (MC)
\fill (1.5,-6) circle (0.07);
\node[draw=none, label={90:$[\Gamma]$}] at (1.5,-7) {};

\draw[very thick, ->] (1.5,0)-- (1.5,-2.8);
\draw[very thick] (1.5,0)--(0.75,0.75);

%\fill (-7,-6) circle (0.05);

%\draw[->] (-7,-3) -- (-7,-5);
%\draw[->] (-6,-1) -- (-2,-1);
%\draw[->] (-6,-6) -- (-2,-6);

%\node at (-8.5, -6) {$[\Gamma]$};
%\node at (-8.5,-1) {$\pi_{1,1}^{-1}\big([\Gamma]\big)$};
%\node at (-8.5,-1.7) {$=\Gamma/\Aut(\Gamma)$};

%\draw (-7,0) -- (-7,-2);
%\draw (-7,0) -- (-7.5,0.5);

\end{tikzpicture}\caption{The preimage $(\pi_{1,2}^{\trop})^{-1}\big([\Gamma]\big)$ in the coarse moduli space $M_{1,2}^{\trop}$ of a point $[\Gamma]\in M_{1,1}^{\trop}$ under the forgetful morphism is not $\Gamma$ itself, but rather $\Gamma/\!\Aut(\Gamma)$. Here the nontrivial automorphism of $\Gamma$ flips the loop edge of $\Gamma$. 
Compare this with Figures~\ref{figure_M12stack} and~\ref{figure_M11stack}.} \label{figure_M12}\end{figure}
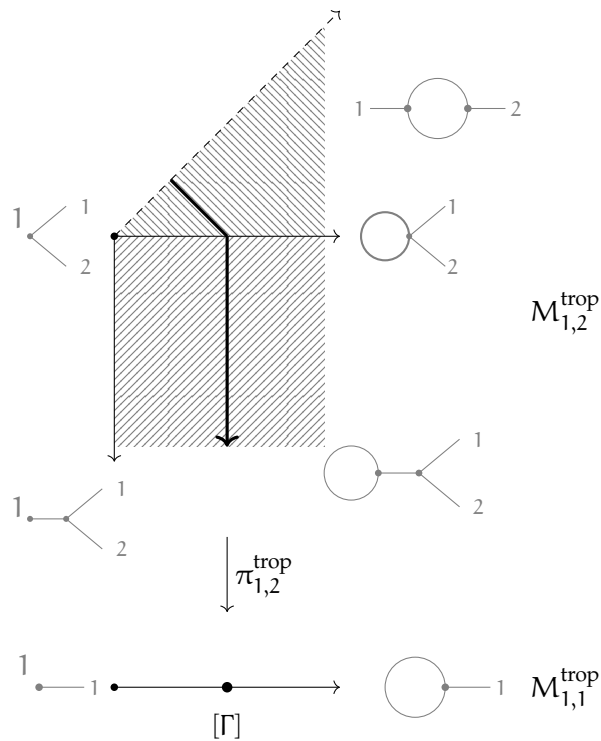

An analogous problem also arises in the algebraic world: Given a point $[C]$ in the coarse algebraic moduli space $\Mbar_{g,n}$,  its preimage under the forgetful morphism (on the level of coarse moduli spaces) is not $C$ itself but rather the quotient $C/\!\Aut(C)$. {The solution to this problem}, by Deligne--Mumford \cite{DeligneMumford_moduliofcurves} and Knudsen \cite{Knudsen_projectivityII}, was to replace $\Mbar_{g,n}$ with the fine moduli stack $\calMbar_{g,n}$. 
%\jonathancolor{as the fundamental objects}. 
Using this language, the preimage via $\pi_{g,n+1}$ of $[C]$ in $\calMbar_{g,n+1}$ (in the sense of taking $2$-categorical fibered products) is in fact $C$ itself. 

\subsection*{A tropical moduli functor} In this article we develop a stack-theoretic approach to the moduli spaces of tropical curves that naturally resolves the same issue in tropical geometry. We propose a \emph{tropical moduli functor}
\begin{equation*}
\calM_{g,n}^{\trop}\mathrel{\mathop:}\RPC^\mathrm{op}\longrightarrow \mathbf{Groupoids}
\end{equation*}
over the category of rational polyhedral cones.  It associates, to a rational polyhedral cone $\sigma$, the groupoid of genus $g$, $n$-marked stable tropical curves $\Gamma$ with edge lengths taking values in the dual monoid $S_\sigma$ to $\sigma$. {This is the notion of} a \emph{tropical curve over $\sigma$} (see Section \ref{section_modulifunctor} for details). The idea is that, while integer, or even real, edge lengths on a tropical curve encode the deformation parameters of a one-parameter degeneration of an algebraic curve, e.g., over the spectrum of a valuation ring with value group $\mathbb{Z}$ or $\mathbb{R}$, allowing the edges to take lengths in $S_\sigma$ keeps track of higher-dimensional deformation parameters of multi-parameter degenerations. We will see in Section \ref{sec:real} how allowing edge lengths in $S_\sigma$ may be generalized to edge lengths in more general monoids including $\mathbb{R}_{\geq 0}^n$. We prove the following representability theorem in Section~\ref{sec:tropcurves}.

\begin{maintheorem}\label{thm_geometric}
The moduli functor $\calM_{g,n}^{\trop}$ is representable by a cone stack, i.e. by a geometric stack over the category of rational polyhedral cone complexes. 
\end{maintheorem}

We refer the reader to Section \ref{section_conestacks}  for the precise definition of a cone stack. Here, we emphasize the following analogy with algebraic geometry instead:

\vspace{1em}
\begin{center}
\begin{tabular}{c | c}
algebraic geometry & tropical geometry \\ \hline
rings & monoids \\
affine schemes  & cones\\
schemes & cone complexes \\
algebraic spaces & cone spaces \\
Deligne--Mumford stacks & cone stacks \\
\end{tabular}
\end{center}
\vspace{1em}

\subsection*{The universal curve} Knudsen \cite{Knudsen_projectivityII} constructs a natural forgetful morphism \begin{equation*}
\pi_{g,n+1}\colon\overline{\calM}_{g,n+1}\longrightarrow\overline{\calM}_{g,n}
\end{equation*}
which removes the $(n+1)$-st marked point of a curve in $\overline{\calM}_{g,n+1}$ and then stabilizes the resulting $n$-marked curve. Via this morphism, $\overline{\calM}_{g,n+1}$ functions as the universal curve over $\overline{\calM}_{g,n}$ in the sense that every family of stable curves in $\overline{\calM}_{g,n}$ arises as a pullback of this family. 

Similarly, expanding on \cite[Section 8]{AbramovichCaporasoPayne_tropicalmoduli}, we construct in Section \ref{section_families}  a natural \emph{forgetful morphism} 
\begin{equation*}
\pi_{g,n+1}^{\trop}\mathrel{\mathop:}\calM_{g,n+1}^{\trop}\longrightarrow \calM_{g,n}^{\trop}
\end{equation*}
given by forgetting the last marked leg of a tropical curve over a cone $\sigma$ and stabilizing all non-stable vertices.

\begin{maintheorem}\label{thm_universalcurve}
The tropical forgetful morphism functions as a universal curve.
\end{maintheorem}
 
More specifically, given a tropical curve $\Gamma$ over $\sigma$, %which, as an element of $\calM_{g,n}^{\trop}(\sigma)$, gives rise to 
i.e., a morphism $u_{[\Gamma]}\mathrel{\mathop:}\sigma\rightarrow \calM_{g,n}^{\trop}$, we construct in Section \ref{section_families}  a cone space $\Cone(\Gamma)\rightarrow \sigma$, the \emph{cone over $\Gamma$}. Using this construction, one way to rephrase Theorem \ref{thm_universalcurve}  is that, for a tropical curve $\Gamma$ over $\sigma$, there is a $2$-cartesian diagram
\begin{equation*}\xymatrix{
\Cone(\Gamma) \ar[r] \ar[d]_c & \calM_{g,n+1}^{\trop} \ar[d]^{\pi_{g,n+1}^{\trop}} \\
\sigma  \ar[r]^{u_{[\Gamma]}} & \calM_{g,n}^{\trop}
}\end{equation*}

When $\sigma=\R_{\geq 0}$, a tropical curve $\Gamma$ over $\sigma$ is nothing but a classical tropical curve (with integer edge lengths) and the cone $\Cone(\Gamma)$ is the topological cone over $\Gamma$. In this case the preimage $c^{-1}(1)$ of $1\in\R_{\geq 0}$, with  metric coming from the lattice length on each component of $\Cone(\Gamma)$, is isometric to $\Gamma$ (see Proposition \ref{prop_preimageisometry}). In the genus zero case, where no stack structure is necessary, this recovers the main result of \cite{FrancoisHampe_universalfamilies}. 

The fact that $\R_{\geq 0}$-valued tropical curves must have integer edge lengths may seem at first like a restriction in our theory, but it is not. We explain in Section~\ref{sec:real} how  our construction of $\calM_{g,n}^{\trop}$ may be canonically extended to admit tropical curves with edge lengths in $\R_{>0}$, or indeed edge lengths valued in any sharp monoid. Moreover, we also explain how to interpret $\calM_{g,n}^{\trop}$ as a topological stack {and how to include the  notion of \emph{extended tropical curves} (as in \cite{Caporaso_tropicalmoduli}) into our framework.} We obtain these various generalizations canonically, via pullback of the original construction of $\calM_{g,n}^{\trop}$ relative to morphisms of topoi over the appropriate sites.

\subsection*{Non-Archimedean tropicalization}
In \cite{AbramovichCaporasoPayne_tropicalmoduli}, Abramovich, Caporaso, and Payne 
describe a natural continuous tropicalization map $\trop_{g,n}^{\ACP}\colon M_{g,n}^{\an}\rightarrow M_{g,n}^{\trop}$ from the Berkovich analytification $M_{g,n}^{\an}$ to the set-theoretic tropical moduli space $M_{g,n}^{trop}$ %\cite{AbramovichCaporasoPayne_tropicalmoduli} 
(see also \cite{Viviani_tropvscompTorelli}).  Their map is given by associating a tropical curve with edge lengths in $\R_{>0}$ to a stable one-parameter degeneration of a smooth algebraic curve over a valuation ring extending a trivially valued field $k$. The edge lengths correspond to the deformation parameters of the degeneration. One of the goals of Part 2 of this paper is to  reinterpret and refine this construction in the language of logarithmic geometry.

{\em Logarithmic geometry} in the sense of Kato-Fontaine-Illusie \cite{Kato_logstr} is a hybrid between algebraic and tropical geometry % in which many of our dreams come true. In particular, 
that allows us to describe multi-parameter degenerations of algebraic curves (see \cite{Kato_logsmoothcurves} and \cite{Abramovichetal_log&moduli}). Our stack-theoretic framework in tropical geometry can be thought of as a ``combinatorial shadow'' of the underlying logarithmic theory. 

To give an idea of what follows, we describe the process of \emph{tropicalization} for logarithmic curves \cite{Kato_logsmoothcurves} over {\em logarithmic points}. In Section~\ref{sec:log-curves}, we  expand this description to families and produce a smooth morphism from $\mathcal M_{g,n}^{\log}$ to $\mathcal M_{g,n}^{\trop}$.

\subsection*{Tropicalization of logarithmic curves over a point} Let $k$ be an algebraically closed field, let $P$ be a sharp monoid, and  denote by $S$ the standard logarithmic point $S=(\Spec k, k^\ast \oplus P)$.
A {\em logarithmically smooth curve} $X\rightarrow S$ may be thought of as a geometrically connected, proper, nodal curve $X$, where each geometric point $x$ of $X$ is endowed with a monoid $\Mbar_{X,x}$. By \cite[Theorem 1.3]{Kato_logsmoothcurves}, the structure of these monoids is highly constrained:
\begin{enumerate}[(i)]
\item At all but a finite number of smooth points $x$ of $X$ we have $\Mbar_{X,x}\simeq P$; 
\item At the remaining smooth points $x_i$ on $X$ we have $\Mbar_{X,x_i}\simeq \N\oplus P$ (these are the marked points);
\item At each node $x_e$ of $X$, the monoid $\overline M_{X,x_e}$ is obtained from $P$ by adjoining two formal generators $a$ and $b$, modulo a relation $a + b = \delta_e$, for some $\delta_e\in P$.

%\item At a node $x_e$ of $X$ we have $\Mbar_{X,x}\simeq \N^2\oplus_{\Delta,\N,p_e} P$, where the amalgamated sum is taken with respect to the diagonal homomorphism $\Delta\mathrel{\mathop:}\N\rightarrow \N^2$ and the homomorphism $\N\rightarrow P$ determined by $1\mapsto p_e\in P$. 
\end{enumerate}

We define \emph{the dual tropical curve of $X/S$} as the tropical curve $\Gamma_{X}$ over $\sigma_S=\Hom(P,\R_{\geq 0})$ given by endowing the dual graph $G_X$ of the pointed curve $(X,x_1, \ldots, x_n)$ with the generalized edge length that associates, to the edge e of $G_X$ corresponding to the node $x_e$ of $X$, the smoothing parameter $\delta_e\in P$. See Figure \ref{figure_dualgraph} for an illustration.

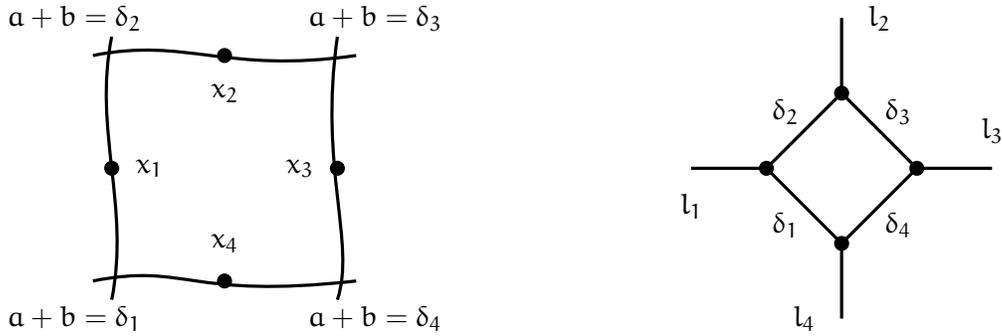
\begin{figure}
\begin{minipage}{0.49\textwidth}\centering
\begin{tikzpicture}
\draw[very thick] (-0.25,0) .. controls (1,0.25) and (1,-0.25) .. (3.25,0);
\draw[very thick] (0,-0.25) .. controls (0.25,1) and (-0.25,2) .. (0,3.25);
\draw[very thick] (-0.25,3) .. controls (1,3.25) and (1.5,2.75) .. (3.25,3);
\draw[very thick] (3,-0.25) .. controls (3.25,0.5) and (2.75,1.5) .. (3,3.25);

\fill (0,1.5) circle (0.10);
\fill (1.5,0) circle (0.10);
\fill (3,1.5) circle (0.10);
\fill (1.5,3) circle (0.10);

\node at (0.5,1.5) {$x_1$};
\node at (1.5,2.5) {$x_2$};
\node at (2.5,1.5) {$x_3$};
\node at (1.5,0.5) {$x_4$};

%\node at (-0.5,2) {$P$};
%\node at (3.5,2) {$P$};
%\node at (2,3.5) {$P$};
%\node at (2,-0.5) {$P$};

\node at (-0.5,3.5) {$a+b=\delta_{2}$};
\node at (-0.5,-0.5) {$a+b=\delta_{1}$};
\node at (3.5,-0.5) {$a+b=\delta_{4}$};
\node at (3.5,3.5) {$a+b=\delta_{3}$};

\end{tikzpicture}
\end{minipage}
\begin{minipage}{0.49\textwidth}\centering
\begin{tikzpicture}
\fill (0,1) circle (0.10);
\fill (2,1) circle (0.10);
\fill (1,0) circle (0.10);
\fill (1,2) circle (0.10);

\draw[very thick] (1,0) -- (2,1);
\draw[very thick] (2,1) -- (1,2);
\draw[very thick] (1,2) -- (0,1);
\draw[very thick] (0,1) -- (1,0);

\draw[very thick] (1,0) -- (1,-1);
\draw[very thick] (2,1) -- (3,1);
\draw[very thick] (1,2) -- (1,3);
\draw[very thick] (0,1) -- (-1,1);

\node at (-1,0.5) {$l_1$};
\node at (3,1.5) {$l_3$};
\node at (0.5,-1) {$l_4$};
\node at (1.5,3) {$l_2$};

\node at (0.25,0.25) {$\delta_{1}$};
\node at (0.25,1.75) {$\delta_{2}$};
\node at (1.75,0.25) {$\delta_{4}$};
\node at (1.75,1.75) {$\delta_{3}$};
\end{tikzpicture}
\end{minipage}
\caption{A logarithmic curve $X$ over a standard logarithmic point $S=(\Spec k, k^\ast \oplus P)$ (on the left) and its associated tropical curve $\Gamma_X$ (on the right). The dual graph has a vertex for each component, a leg for each marked point, and an edge for each node of $X$. The edge lengths of $\Gamma_X$ are  the logarithmic deformation parameters $\delta_{1},\ldots, \delta_{4}$ at the nodes of $X$.}\label{figure_dualgraph}
\end{figure}

The association $X/S\mapsto \Gamma_X/\sigma_S$ defines a natural \emph{tropicalization functor}
\begin{equation*}
\trop_{g,n,S}\colon \calM_{g,n}^{\log}(S)\longrightarrow \calM_{g,n}^{\trop}(\sigma_S)
\end{equation*}
that specializes to the well-known tropicalization map of one-parameter degenerations  in the case $P=\N$ (as in \cite{AbramovichCaporasoPayne_tropicalmoduli}).

Intuitively, one may think of $X/S$ as encoding data coming from a multi-parameter smoothing of a stable, marked curve: the $x_i$'s are the marked points on the central fiber, and each $\delta_e\in P$ is the deformation parameter at the corresponding node. The point is that the structure of a logarithmic curve contains this data even in the absence of an actual multi-parameter degeneration. Tropicalization  then extracts just the combinatorial information: the topological type of the central fiber and {\it how fast} the nodes smooth. %This is because the points  in the boundary of a logarithmic scheme carry more information than just their location in the underlying scheme; they also record a ghost of a smoothing direction towards the interior.  
This  information  determines a well-defined image in the interior of the tropicalization. 

A reasonable question is {then} how to bring into this picture families of curves where the general fiber is not smooth, or equivalently, {where} nodes of the central fiber persist throughout the deformation. On the combinatorial side this leads to {\it extended tropical curves}, as introduced in \cite{Caporaso_tropicalmoduli}: edges corresponding to nodes that do not smooth have infinite length. Their  parameter space is the {generalized} extended cone complex $\overline{M}_{g,n}^{\trop}$ {defined in \cite{Caporaso_tropicalmoduli, AbramovichCaporasoPayne_tropicalmoduli}. From the logarithmic point of view, this means working with {\it pointed monoids}, i.e. monoids that must contain an {\it absorbing element} (playing the role of  $\infty$). We refer the reader to \cite{HuszarMarcusUlirsch_troplogclutch&glue} for a careful treatment of this generalization. %Hence a tropicalization functor may be established so as to connect the more general tropical theory of extended tropical curves to a more restrictive theory on the logarithmic side. 
%In this work, the point of view which is more natural for our purposes is dictated by the following fact: the \emph{moduli functor} of the moduli space of curves is representable by a cone complex, rather than an extended cone complex.
%the broader generality \martin{I am not sure I understand why working with pointed logarithmic structures is more restrictive on the logarithmic side. Isn't it really not restrictive enough?}\renzo{Here I was just referring to the fact that having an absorbing element is a further requirement that we ask of the monoids we wish to consider. Therefore the category of pointed monoids is a subcategory of the category of monoids, and we are restricting to this subcategory.} on \martincolor{the} logarithmic side, which gives us statements about the open tropical moduli space. }
%In Section \ref{sec:com-mon} we give a precise treatment of this discussion, and show that one may canonically recover the category  $\overline{\calM}_{g,n}^{\trop}$ from $\calM_{g,n}^{\trop}$ by pull back.}
In the framework of this article, it is still possible to recover the extended cone complex by evaluating the functor of points of the tropicalization on the commutative monoid  $\R_{\geq 0} \cup \{ \infty \}$ (see Section~\ref{sec:com-mon} for details).   }
 
\subsection*{Tropicalization via Artin fans} The theory of \emph{Artin fans} was developed in  \cite{AbramovichWise_invariance} and \cite{AbramovichChenMarcusWise_boundedness} (see also \cite{Abramovichetal_logsurvey} and \cite{Ulirsch_nonArchArtin}) as an incarnation of cone stacks in the category of logarithmic algebraic stacks. In fact, as we will show in Theorem~\ref{thm:artin-fans}, every cone stack can be lifted to a logarithmic algebraic stack: 

\begin{maintheorem}\label{thm_artinfans}
The $2$-category of cone stacks is equivalent to the $2$-category of Artin fans.
\end{maintheorem}

In our case, $\calM_{g,n}^{\trop}$ lifts to a logarithmic algebraic stack that we denote $a^\ast\calM_{g,n}^{\trop}$.  In Section \ref{sec:log-curves}  we interpret $a^\ast\calM_{g,n}^{\trop}$ as a moduli functor
\begin{equation*}
a^\ast\calM_{g,n}^{\trop}\colon\mathbf{LogSch}\longrightarrow \mathbf{Groupoids}
\end{equation*}
that associates to a logarithmic scheme $S$ the groupoid of \emph{families of tropical curves over $S$}, i.e., of collections $(\Gamma_s)$ of tropical curves, indexed by the geometric points $s$ of $S$ with a generalized edge length with values in the characteristic monoid $\Mbar_{S,s}$ that are compatible under specialization (see Definition \ref{def:tropical-curve-log-base} for details). The following theorem collects the
statements of Corollaries \ref{cor:rep}, \ref{cor:ext}, and Theorem \ref{thm:strictsmooth}.

\begin{maintheorem}\label{thm_tropicalization}
The moduli functor $a^\ast\calM_{g,n}^{\trop}$ is representable by a logarithmic algebraic stack that is logarithmically \'etale over $k$, and the natural \emph{tropicalization morphism}
\begin{equation} \label{eqn:tropicalization of mgn}
\trop_{g,n}\colon \calM_{g,n}^{\log}\longrightarrow a^\ast\calM_{g,n}^{\trop}
\end{equation}
that associates to a logarithmic curve $X/S$ its family of dual tropical curves is strict, smooth, and surjective. 
\end{maintheorem}

Theorem \ref{thm_tropicalization}, in particular, shows that the realization problem for abstract tropical curves is always solvable: given an abstract tropical curve $\Gamma$ (for simplicity with edge lengths in $\N$), since $\trop_{g,n}$ is surjective, we may find a logarithmic curve over the standard logarithmic point $\big(\Spec k, k^\ast \oplus \N\big)$, whose dual tropical curve is $\Gamma$. Moreover, as $\trop_{g,n}$ is also strict and smooth (i.e. logarithmically smooth), we may smooth this nodal curve in a logarithmically smooth family. 

Of course, for logarithmically smooth curves, this is a very well-known fact (see e.g. \cite[Section 3.4]{Gross_book}), but the model of Theorem~\ref{thm_tropicalization} is a convenient frame in which to consider other tropical realization problems, which can thus be divided into questions of surjectivity and smoothness of tropicalization morphisms analogous.  For example, the realization question for tropical genus~$1$ curves in toric varieties was answered by this method~\cite{RanganathanSantosParkerWiseII}.

%We believe that our framework will be  useful when studying the realizability problem for other moduli spaces, such as the moduli space of logarithmically stable maps in the sense of \cite{Chen_DeligneFaltingsI, AbramovichChen_DeligneFaltingsII,GrossSiebert_logGWinv} (see e.g. \cite{NishinouSiebert_correspondence,  CheungFantiniParkUlirsch_faithfulrealizability, Ranganathan_superabundanceArtinfan, Ranganathan_ratcurtorvar&nonArch, Ranganathan_superabundantgeometries}).

In \cite[Theorem 4.4]{FosterRanganathanTalpoUlirsch_logPic} it is shown that  the moduli stack $\calM_{g,n}^{\log}$ of logarithmic curves may be viewed as a fine moduli stack parameterizing of {\em metrized curve complexes} in the sense of \cite{AminiBaker_metrizedcurvecomplexes} over the category $\mathbf{LogSch}$ of logarithmic schemes. Theorem~\ref{thm_tropicalization} shows that $a^\ast \calM_{g,n}^{\trop}$ functions as a lift of the moduli stack $\calM_{g,n}^{\trop}$  to $\mathbf{LogSch}$, thereby realizing both of these objects within the same category. From this point of view, the tropicalization map $\trop_{g,n}$  is the map that associates to a metrized curve complex its underlying tropical curve. 

\subsection*{Non-Archimedean tropicalization revisited}
 Theorem \ref{thm_tropicalization} allows for a reinterpretation of the non-Archimedean tropicalization map of \cite{AbramovichCaporasoPayne_tropicalmoduli}. Let $\calM_{g,n}^{an}$ be the non-Archimedean analytic stack\footnote{Although this is a non-Archimedean analytic stack, we always think of $\calM_{g,n}^{an}$  and other non-Archimedean analytic stacks in terms of its underlying topological space (see \cite[Section 3]{Ulirsch_trop=quot}).} associated to the moduli stack $\calM_{g,n}$ of smooth $n$-pointed algebraic curves of genus $g$. Denote by $M_{g,n}^{trop}$ the (set-theoretic) moduli space of stable $n$-marked tropical curves. The tropicalization map 
\begin{equation*}
\trop_{g,n}^{\ACP}\colon \calM_{g,n}^{an}\longrightarrow M_{g,n}^{trop}
\end{equation*}
associates to a point in $\calM_{g,n}^{an}$, which corresponds to a smooth curve over a non-Archimedean field extending the base field, the dual tropical curve of a stable degeneration. For a node given by $xy=f$ in local coordinates, the length of the associated edge is the real number given by $\val(f)$, i.e. by the valuation of the deformation parameters.

By applying Thuillier's \cite{Thuillier_toroidal} generic fiber functor $(.)^\beth$ over trivially valued fields, Theorem \ref{thm_tropicalization} now immediately implies the following Corollary \ref{cor_trop=anal} (see Section \ref{section_nonArchtrop}). 

\begin{maincorollary}\label{cor_trop=anal}
There is a natural continuous, open, and dense injection  $M_{g,n}^{trop}\hookrightarrow\big(a^\ast \calM_{g,n}^{trop}\big)^\beth$ that makes the following diagram commute:
\begin{equation*}\begin{CD}
\calM_{g,n}^{an}@>\trop_{g,n}^{\ACP}>> M_{g,n}^{trop}\\
@V\subseteq VV @VV\subseteq V\\
\overline{\calM}_{g,n}^\beth @>\trop_{g,n}^\beth>> \big(a^\ast\calM_{g,n}^{trop}\big)^\beth
\end{CD}\end{equation*}

\end{maincorollary}

In fact, one may show that the inclusion $M_{g,n}^{trop}\subset \big(a^\ast \calM_{g,n}^{trop}\big)^\beth$ extends to a homeomorphism $\Mbar_{g,n}^{trop} \cong \big(a^\ast \calM_{g,n}^{trop}\big)^\beth$, where $\Mbar_{g,n}^{trop}$ denotes the set-theoretic moduli space of \emph{extended tropical curves} (with edge lengths allowed to be $\infty$).  In this case, the analytic map $\trop_{g,n}^\beth$ is nothing but the extended non-Archimedean tropicalization map of \cite{AbramovichCaporasoPayne_tropicalmoduli}.

\subsection*{Skeletons and tropicalization} The authors of  \cite{AbramovichCaporasoPayne_tropicalmoduli} also show that $M_{g,n}^{trop}$ is the non-Archi\-medean skeleton of $\calM_{g,n}^{an}$ in the sense of Thuillier \cite{Thuillier_toroidal} (defined with respect to the coordinates coming from the Deligne-Mumford boundary) and that the non-Archimedean tropicalization map is nothing but the natural retraction onto this skeleton. We refer the reader to \cite{Ulirsch_nonArchArtin} where the third author explains how one can recover this result via the theory of Artin fans. 

By \cite[Proposition 3.1.1]{AbramovichChenMarcusWise_boundedness} to every logarithmic stack $\calX$ there is an Artin fan $\calA_\calX$ with faithful monodromy as well as a strict morphism $\calX\rightarrow \calA_\calX$ that is initial among all strict morphism to Artin fans with faithful monodromy. In this language \cite[Theorem 1.3]{Ulirsch_nonArchArtin} says that that the tropicalization morphism induces an equivalence between the Artin fan associated to $\calM_{g,n}^{log}$ and the $a^\ast\calM_{g,n}^{trop}$ (thought of as a logarithmic stack). 

An Artin fan $\calA$ is said to have \emph{faithful monodromy} if the natural morphism $\calA\rightarrow \LOG_k$ to Olsson's classifying stack  of logarithmc structures (see \cite{Olsson_LOG}) is represenatable. In combinatorial terms this means that the morphisms in the associated  cone stack are all proper face morphisms. We may therefore think of this construction as a version of a ``coarse moduli space" associated to a cone stack and \cite[Theorem 1.3]{Ulirsch_nonArchArtin} says that the Artin fan associated to $\calMbar_{g,n}^{log}$ is naturally equivalent to the ``coarse moduli space" of $\calM_{g,n}^{trop}$. We point out that, in general, $\calM_{g,n}^{trop}$ does not have faithful monodromy, since, for example, the loop flip in $\calM_{1,1}^{trop}$ cannot be distinguished from the identity by considering only their actions on the set of lengths of edges in the graph.

%\martincolor{By \cite[Theorem 1.2.1]{AbramovichCaporasoPayne_tropicalmoduli},} the analytic tropicalization \martincolor{map has a continuous section making it into a strong strong deformation retraction.} %morphism is a deformation retraction onto a subcomplex.  
For a logarithmically regular scheme $X$ (defined in the Zariski topology), the initial strict morphism $X\rightarrow{\calA_X}$ can be phrased as a morphism of locally monoidal spaces to a Kato fan, which can be identified with the map $\rho\colon X\rightarrow X$ that sends points of $X$ to the generic points of the corresponding logarithmic strata of $X$ (see \cite{Kato_toricsing} and \cite[Lemma~4.18]{AbramovichChenMarcusUlirschWise}). This defines a strong deformation retraction of X (with the Zariski topology) via the continuous homotopy $H\colon X\times [0,1]\rightarrow X$ given by 
\begin{equation*}
    H(x,t)=\begin{cases} r(x) &\textrm{ if } t<1\\
    x &\textrm{ if } t=0\ .
    \end{cases}
\end{equation*}
At this point we know of no natural topology in which the tropicalization morphism is a homotopy equivalence in more general situations.

\subsection*{The difference between analytic and logarithmic tropicalization}%, and the absence of extended tropicalizations}

In this paper, we think of tropical geometric objects as set- or groupoid-valued covariant functors on the category of commutative monoids.  This is analogous to the way schemes and algebraic stacks can be viewed as set- or groupoid-valued functors on the category of commutative rings.
As with schemes over $\mathbb C$, where the set underlying the complex analytification is recovered as the set of $\mathbb C$-valued points, the complex of real polyhedral cones that one generally thinks of as a tropical variety is recovered from a functor $F : \mathbf{Mon} \to \mathbf{Sets}$ by evaluating $F$ on the monoid $\mathbb R_{\geq 0}$ of nonnegative real numbers.

This shift in perspective allows one to define a morphism directly from a logarithmic scheme to its tropicalization, without any intermediate analytification step.  We explain in Section~\ref{sec:log-generalities} how a functor from monoids to sets can be extended to a presheaf on the category of monoidal spaces.  Since logarithmic schemes are monoidal spaces with additional structure, this makes it possible to define morphisms directly from logarithmic schemes to tropical varieties.

These tropicalization morphisms are natural transformations of functors. If one wants to pass to the $\mathbb R_{\geq 0}$-points of the target, one must also do so on the source.  This amounts to looking at morphisms from a logarithmic point whose characteristic monoid is $\mathbb R_{\geq 0}$ into a given logarithmic scheme.

This  process is similar, but not identical, to analytification. Indeed, taking the $\mathbb R_{\geq 0}$-points of a logarithmic scheme does not require a generic fiber to analytify.  Moreover, the $\mathbb R_{\geq 0}$-points of the functor represented by a rational polyhedral cone $\sigma$ are simply the members of the underlying cone of $\sigma$.  This means that the logarithmic tropicalization map sends the $\mathbb R_{\geq 0}$-points of the boundary of a logarithmic scheme to the cone complex underlying its tropicalization and \emph{not to the extended cone complex} considered by Abramovich, Caporaso, and Payne \cite{AbramovichCaporasoPayne_tropicalmoduli}, among others.  (This does not come up in Corollary~\ref{cor_trop=anal} because the tropical moduli functor is also analytified there.)

To see how logarithmic tropicalization works in a concrete example,
let $X$ be a stable non-smooth curve over a non-Archimedean field $K$ extending the base field $k$, i.e., a $K$-valued point in the boundary of $\calMbar_{g,n}$. By the valuative criterion for the properness of $\calMbar_{g,n}$, after replacing $K$ by a finite extension, there is a stable curve $\calX$ over the valuation ring $R$ of $K$ extending $X$. The analytic tropicalization map $\trop_{g,n}^{ACP}$ sends the point $[X]$ to a tropical curve that has an edge of infinite length, since the equation $xy=0$ is not smoothed in the generic fiber.

The $R$-valued point $[\calX]$ of $\calMbar_{g,n}$ only lifts to a point of $\calM_{g,n}^{log}$ if $S=\Spec R$ is endowed with a  logarithmic structure $M_S$ that is non-trivial over the generic point. The tropicalization of $\calX\rightarrow\Spec R$ is then a tropical curve with edge lengths in the monoid $\Mbar_{S}$ and not an extended tropical curve.  Notably, $\Mbar_S$ must contain an element $\ell$, the length of the edge that is not smoothed, that is not contained in the preimage of the valuation monoid of $R$, and therefore $\Mbar_S$ cannot be $\mathbb R_{\geq 0}$ if $\calX$ is to be a logarithmic curve over $S$.

%This is not as strange as it may seem, as logarithmic geometry sees a ghost of a nearby fiber of a degeneration, whether that degeneration is actually physically present or not. 
%In the situation where $X$ is a semistable degeneration over the spectrum $S$ of a discrete valuation ring, the $\mathbb R_{\geq 0}$-points of the special fiber of $X$ correspond to certain equivalence classes of points of the generic fiber, valued in an extension of $S$ with value group $\mathbb R$.  

One can recover the extended theory by looking at $(\mathbb R_{\geq 0} \cup \{ \infty \})$-valued points instead of $\mathbb R_{\geq 0}$-valued points, both of $X$ and of its tropicalization,  as we explain in detail  in Section~\ref{sec:com-mon}.  In the example above, the monoid $\mathbb R_{\geq 0} \cup \{ \infty \}$ does contain an element not in the preimage of the valuation monoid of $R$, namely $\infty$.

Working with pointed monoids like $\R_{\geq 0}\cup\{\infty\}$ requires the more general framework of pointed logarithmic structures introduced in \cite{HuszarMarcusUlirsch_troplogclutch&glue}. %\jonathancolor{While working with monoids like $(\mathbb R_{\geq 0} \cup \{ \infty \})$ \rcolor{is a valuable perspective, w
With an eye towards certain applications, we find it advantageous to work consistently with integral monoids in this article.  For example, logarithmic curves are more neatly characterized over integral bases than non-integral ones~\cite{Kato_logsmoothcurves}, with a more manageable deformation theory, and it seems impossible to define a reasonable logarithmic or tropical Picard group~\cite{MolchoWise} over a non-integral base.  

Restricting attention to integral base monoids does not make it impossible to study a family of nodal curves with non-smooth generic fiber.  Rather, as in the example above, we must adjoin a new element $\ell$ to $\overline M_S$ that is independent of the valuation monoid of $R$ to account for the node that is not smoothed in the generic fiber. In the partial order on $\overline M_S^{\rm gp}$ induced from $\overline M_S$, this element is larger than any element of the preimage of the valuation monoid of $R$, which we interpret to mean that $\ell$ is infinite \emph{relative to} the valuation monoid of $R$.

\subsection*{Tautological morphisms} Knudsen \cite{Knudsen_projectivityII} introduced and studied a collection of natural morphisms between the algebraic moduli stacks $\calM_{g,n}$ that have since become a fundamental tool to study the geometry of these moduli spaces:
\begin{itemize}
\item the forgetful morphism 
\begin{equation*}
\pi_{g,n+1}\colon\calMbar_{g,n+1}\longrightarrow\calMbar_{g,n}
\end{equation*}
already mentioned above, given by forgetting the $(n+1)$-st marked point (and possibly stabilizing) as well as
\item the clutching maps 
\begin{equation*}
\calMbar_{g_1,n_1+1}\times\calMbar_{g_2,n_2+1}\longrightarrow\calMbar_{g_1+g_2,n_1+n_2} \ ,
\end{equation*}
given by gluing two curves by identifying the $(n_1+1)$-st marked point of one curve with the $(n_2+1)$-st of the other, and 
\begin{equation*}
\calMbar_{g-1,n+2}\longrightarrow \calMbar_{g,n} \ ,
\end{equation*}
given by gluing a curve with itself by identifying the last two marked points. 
\end{itemize}
Expanding on \cite{AbramovichCaporasoPayne_tropicalmoduli}, we introduce tropical and logarithmic analogues of those maps in Sections \ref{section_families} and \ref{sec:tropicalization} respectively. The following Theorem \ref{thm_tautological}  is an informal summary of Theorem \ref{thm:clutch} and Theorem \ref{thm:twodiag}. 

\begin{maintheorem}\label{thm_tautological}
The logarithmic and tropical  tautological morphisms commute with the tropicalization morphism.
\end{maintheorem}

The main difficulty in Section \ref{sec:tropicalization} stems from the fact that the clutching morphisms are not  logarithmic maps, as they factor through embeddings of closed logarithmic boundary strata. Our approach is to not describe those maps as morphisms but rather as correspondences. We refer the reader to  \cite{HuszarMarcusUlirsch_troplogclutch&glue} for an alternative resolution of this issue using \emph{pointed logarithmic structures}, i.e., logarithmic structures that contain an absorbing element.

\subsection*{Further developments}
This work is a first step  in the development of the foundations of tropical moduli stacks. We note that while rational polyhedral cones are certainly part of an as yet not properly defined class of ``tropical varieties'', they  are far from exhaustive. So we would expect that our moduli functor is a special case of a moduli functor that is defined on a much larger category. The recent progress towards a notion of a \emph{tropical scheme} (see \cite{GiansiracusaGiansiracusa_tropicalschemes}, \cite{MacLaganRincon_tropicalschemes}, \cite{Lorscheid_tropicalschemes}, and \cite{MacLaganRincon_tropicalideals}) might eventually provide us with a candidate for both a base category and a notion of families of tropical curves over tropical base schemes.

Furthermore, as opposed to the tropical schemes appearing in \cite{GiansiracusaGiansiracusa_tropicalschemes}, a particular feature of our construction is that the cone stacks really have no structure at the origin. The reason for this is that Giansiracusa and Giansiracusa deal with tropicalization of closed subschemes of toric varieties, while our framework is suitable for all logarithmic schemes, which may not admit a logarithmic embedding into a toric variety. In \cite[Section 15]{Lorscheid_tropicalschemes} Lorscheid proposes a partial solution for this discrepancy based on his theory of blueprints.

We {envision} that our precise stack-theoretic description of the moduli spaces of tropical curves will eventually foster the development of a robust tropical intersection theory that applies to all tropical moduli spaces. Such a theory is, in particular, expected to give a firm algebraic geometric foundation to the many enumerative applications of tropical geometry, as  exemplified by Ranganathan's proof  \cite{Ranganathan_ratcurtorvar&nonArch} of the Nishinou-Siebert correspondence theorem \cite{NishinouSiebert_correspondence} for rational curves on toric varieties. A first step in this direction has been undertaken by Gross in \cite{Gross_toroidalintersectiontheory}, where he develops a tropical intersection theory for toroidal embeddings without self-intersections and uses this theory to prove a correspondence theorem for relative rational descendent Gromov-Witten invariants. 

Since this article has appeared on the arXiv, several other projects have made use of the techniques developed here: In \cite{RanganathanSantosParkerWiseI, RanganathanSantosParkerWiseII}, Ranganathan, Santos-Parker, and the fourth author give a modular reinterpretation of the Vakil-Zinger resolution of the moduli space of elliptic stable maps to projective space and give a complete solution of the realizability problem for elliptic tropical stable maps using Speyer's well-spacedness condition. In \cite{MolchoWise}, Molcho and the fourth author study the properties of the logarithmic Picard group and its tropicalization in the framework of this article. As already hinted upon above, in \cite{HuszarMarcusUlirsch_troplogclutch&glue}, Huszar, Marcus, and the third author extend the moduli-theoretic framework developed here to the categories of pointed logarithmic schemes and extended cone complexes. Finally, expanding on Corollary \ref{cor_trop=anal}, in \cite{Ulirsch_nonArchArtin} the third author is giving a new proof of the main result of \cite{AbramovichCaporasoPayne_tropicalmoduli} comparing the Artin fan of $\calM_{g,n}^{log}$ with the Artin fan of $\calM_{g,n}^{trop}$.

\subsection*{Acknowledgements}
We thank Dan Abramovich, Mattias Jonsson, Eric Katz, Oliver Lorscheid, Diane Maclagan, Steffen Marcus, Margarida Melo, Dhruv Ranganathan, and Filippo Viviani for several important discussions during the development of this article. We also thank Zhi Jin for catching an inaccuracy in an earlier version of Figure \ref{figure_M12stack}.

Much of this article was written while the authors were participating in a SQuaRE at the American Institute of Mathematics. We thank AIM for hosting us and providing us with this unique opportunity. Parts of this project were also completed while R.~C., M.~C., and M.~U. were visiting the Fields Institute in Toronto, whose hospitality we gratefully acknowledge. The final revisions were undertaken when M.~U. was visiting the Max Planck Institute for Mathematics in the Sciences in Leipzig; he would like to thank his host Bernd Sturmfels. M.~C. and J.~W. were supported in part by NSA Young Investigators Grants. R.~C. gratefully acknowledges support from the Simons Foundation through Collaboration Grant 420720.
Many thanks are also due to the anonymous referee(s) whose comments improved the paper in many ways.

%%%%%%%%%%%%%%%%%%%%%%%%%%%%%%%%%%%%%%%%%%%%%%%%%%%%%%

\bigskip

\part{Foundations of tropical moduli stacks}
\bigskip

\section{Geometric stacks in arbitrary contexts}\label{section_geometricstacks}

In \cite{Simpson_nstacks} Simpson (based on conversations with Walter) proposed the concept of a \emph{geometric stack}, which gives meaning to a geometric theory of stacks over any site that contains a reasonable class of ``smooth morphisms". The notion of a \emph{geometric context} formalizes these minimal requirements needed in order to build a meaningful theory of %(higher) 
geometric spaces and geometric stacks (see e.g. \cite[Section 2.2]{ToenVaquie_algebrisation}, \cite[Section 1.3.2]{ToenVezzosi_HAGII}, and \cite[Section 2.2]{PortaYu_higherGAGA}).  This section is foundational, and the reader may wish to skip to Section 2, referring back as necessary.

\begin{definition}\label{def_context}
A \emph{geometric context} $(\bfC, \tau, \PP)$ consists of a category $\mathbf{C}$, a Grothendieck topology $\tau$ on $\mathbf{C}$, as well as a class $\PP$ of morphisms in $\bfC$ that fulfill the following axioms:
\begin{enumerate}[(i)]
\item The category $\bfC$ admits finite limits as well as finite disjoint unions. 
\item The topology $\tau$ on $\bfC$ is subcanonical, i.e. for every object $X$ in $\bfC$ the presheaf 
\begin{equation*}\begin{split}
h_X\mathrel{\mathop:}\bfC^{op}&\longrightarrow \mathbf{Sets}\\
Y&\longmapsto \Hom(Y, X)
\end{split}\end{equation*}
is a sheaf on the site $\bfC_\tau$. 
\item The class of morphisms in $\PP$ is stable under composition and base change, it contains all isomorphisms, and every $\tau$-covering has a refinement consisting of morphisms in $\PP$. 
\item Let $f\mathrel{\mathop:} Y\rightarrow X$ be a morphism in $\bfC$. If there is a $\tau$-covering $V_i\rightarrow Y$ such that the compositions $V_i\rightarrow X$ are in $\PP$, then $f$ is also a morphism in $\PP$. 
\end{enumerate} 
\end{definition}

\begin{remark}
This definition of a geometric context is a variant of the one given in \cite[Definition 2.3]{ToenVaquie_algebrisation}. In particular, unlike in \cite[Section 1.3.2]{ToenVezzosi_HAGII} and \cite[Section 2.2]{PortaYu_higherGAGA}, we do not use the language of $\infty$-categories and we only work with $1$-stacks (or $(2,1)$-sheaves). 
\end{remark}

By Yoneda's Lemma and Definition \ref{def_context} (ii), the association $X\mapsto h_X$ identifies the objects in $\bfC$ with a full and faithful subcategory of the category $\mathbf{Sh}(\bfC,\tau)$ of sheaves on the site $(\bfC,\tau)$. We say a sheaf on $\bfC$ is \emph{representable} if it is isomorphic to the sheaf $h_X$ associated to an object $X$ in $\bfC$. In a common abuse of notation we do not distinguish between an object $X$ and its associated sheaf $h_X$ and simply denote both by $X$.

A morphism $Y\rightarrow X$ is \emph{representable (by objects in $\bfC$)} if for every morphism $S\rightarrow X$ from an object $S$ in $\bfC$, the fibered product $Y\times_XS$ is representable by an object $T$ in $\bfC$. It is a standard fact that the diagonal morphism $\Delta_X\mathrel{\mathop:}X\rightarrow X\times X$ of a sheaf $X$ on $\bfC$ is representable if and only every morphism $S\rightarrow X$ from an object $S$ is representable (see e.g. \cite[Tag 045G]{stacks-project} or \cite[Proposition 7.13]{Vistoli_intersectiontheorystacks} for a proof in the algebraic context).

\begin{definition}
A \emph{geometric space} in the context $(\bfC,\tau,\PP)$ is a sheaf $X\mathrel{\mathop:}\bfC_\tau^{op}\rightarrow\mathbf{Sets}$ that fulfills the following conditions:
\begin{enumerate}[(i)]
\item The diagonal morphism $\Delta\mathrel{\mathop:}X\rightarrow X\times X$ is representable by objects in $\bfC$. 
\item There is a (necessarily representable) morphism $U\rightarrow X$ from an object $U$ in $\bfC$ that is a covering and in $\PP$. 
\end{enumerate}
\end{definition}

The category of geometric spaces is the full and faithful subcategory of $\mathbf{Sh}(C,\tau)$ whose objects are isomorphic to a geometric space. We may associate to every geometric space $X$ a category $(\bfC/X)$ fibered in groupoids over $\bfC$, whose fiber over an object $S$ in $\bfC$ consists of all morphisms $S\rightarrow X$ in $\bfC$. It is a stack over $(\bfC,\tau)$, since $h_X$ is sheaf on $\bfC_\tau$ (see \cite[Tag 0430]{stacks-project}). We say that a category fibered in groupoids over $\bfC$ is \emph{representable by a geometric space $X$} if it is equivalent to $(\bfC/X)$. Again, in a very common abuse of notation, we use the letter $X$ to denote both the sheaf $\bfC^{op}_\tau\rightarrow\mathbf{Sets}$ and the stack $(\bfC/X)$. 

A morphism $\calX\rightarrow \calY$ of categories fibered in groupoids over $\bfC$ is said to be \emph{representable by geometric spaces} if for every morphism $V\rightarrow\calY$ from a geometric space $V$ the ($2$-)fiber product $\calX\times_\calY V$ is representable by a geometric space $U$. Again, the diagonal morphism $\Delta_\calX\mathrel{\mathop:}\calX\rightarrow\calX\times\calX$ of a category fibered in groupoids is representable by geometric spaces if and only if every morphism $U\rightarrow \calX$ from a geometric space is representable by geometric spaces. 

\begin{definition}\label{def:geometric-stack}
A \emph{geometric stack} in the context $(\bfC, \tau, \PP)$ is a stack $\calX$ over the site $\bfC_\tau$ that fulfills the following two axioms:
\begin{enumerate}[(i)]
\item The diagonal morphism $\Delta\mathrel{\mathop:}\calX\rightarrow\calX\times\calX$ is representable by geometric spaces.
\item There is a (necessarily representable) morphism $U\rightarrow \calX$ from a geometric space $U$ that is a covering and in $\PP$. 
\end{enumerate}
\end{definition}

The $2$-category of geometric stacks in $(\bfC,\tau,\PP)$ is the full and faithful subcategory of the $2$-category of categories fibered in groupoids over $\bfC$ whose objects are equivalent to a geometric stack. Recall that a \emph{setoid} is a groupoid with at most one arrow between any two objects or, equivalently, a category that is equivalent to a set. A geometric stack is equivalent to (the stack associated to) a geometric space if and only if it is fibered in setoids. 

\begin{example}
Let $\bfC_\tau$ be the category $\mathbf{Sch}_{et}$ of schemes, endowed with the \'etale topology, and let $\PP$ be the class of smooth morphisms. Then our definitions yield the familiar theory of \emph{algebraic spaces} and \emph{algebraic stacks} (as in \cite{stacks-project}). 
\end{example}

\begin{example}\label{example_topstacks}
Let $\bfC$ be the category $\mathbf{Top}$ of topological spaces. There is a natural Grothendieck topology $\tau_{open}$ on $\mathbf{Top}$, whose coverings consists of coverings by open subsets.  In order to define the notion of a \emph{topological stack} there are many choices for $\PP$ (see \cite{Noohi_topstacksI}, \cite{Metzler_top&smoothstacks}, and \cite{BehrendGinotNoohiXu_stringtopology} for details). 
\end{example}

%\begin{example}
%Similarly, one can take $\bfC$ to be the category of differentiable manifolds, endowed with the topology coming from open submanifolds, and $\PP$ to be the class of smooth submersions. In this case we obtain the notion of a differentiable stack (see \cite{Metzler_top&smoothstacks} and \cite{BehrendXu_diffstacks}).
%\end{example}

%\begin{example}
%Let $\bfC$ be the category of Stein complex analytic spaces, endowed with the analytic topology, and take $\PP$ to be the class of smooth morphisms. In this case we obtain the theory of \emph{complex analytic stacks} (see \cite{PortaYu_higherGAGA} as well as \cite[Section 3.2]{BehrendNoohi_uniformizationDMcurves} and \cite[Chapitre 5]{Toen_thesis} for alternative approaches).
%\end{example}

\begin{example}
Let $k$ be a non-Archimedean field and let $\bfC$ be the category $\mathbf{An}_k$ of Berkovich analytic spaces, endowed with the \'etale topology (see \cite{Berkovich_book} and \cite{Berkovich_etalecoho}). If we consider $\PP$ to be the class of universally submersive $G$-smooth morphisms (as defined in \cite{ConradTemkin_descent}), then geometric stacks are nothing but the non-Archimedean analytic stacks defined in \cite{Ulirsch_trop=quot}. We also refer the reader to \cite{Yu_Gromovcompactness} and \cite{PortaYu_higherGAGA}, where the authors propose an alternative version of this theory.
\end{example}

\begin{example}
One may enlarge the category of toric varieties to a category of \emph{toric varieties with torsion} that also includes split diagonalizable groups. Its topology is generated by torus-invariant open subsets and, taking $\PP$ be the class of smooth toric morphisms, we obtain a geometric context. \emph{Toric stacks} (in the sense of \cite{GeraschenkoSatriano_toricstacksI}) are then nothing but geometric stacks in this context.
\end{example}

A \emph{geometric groupoid} is a groupoid object $(U,R,s,t,c,i,e)$ in the category of geometric spaces, such that the two projection morphisms $s$ and $t$  are covers and in $\PP$. So $(U,R,s,t,c,i,e)$ consists of two geometric spaces $R$ and $U$ as well as 
\begin{itemize}
\item a \emph{source morphism} $s\mathrel{\mathop:} R\rightarrow U$ in $\PP$,
\item a \emph{target morphism} $t\mathrel{\mathop:} R\rightarrow U$ in $\PP$, 
\item a \emph{composition morphism} $c\mathrel{\mathop:} R\times_{s,U,t}R\rightarrow R$,
\item an \emph{inverse morphism} $i\mathrel{\mathop:} R\rightarrow R$, and 
\item a \emph{unit morphism} $e\mathrel{\mathop:} U\rightarrow R$, 
\end{itemize}
such that for every object $T$ in $\bfC$ the septuple $\big(U(T), R(T), s,t,c,i,e)$ is a groupoid in the category of sets --- in other words, describes a category in which all morphisms are invertible. We will  write a geometric groupoid simply as $(R\rightrightarrows U)$, avoiding the reference to all involved morphisms. 

A geometric groupoid $(R\rightrightarrows U)$ defines a category fibered in groupoids  whose fiber over  an object $T$ in $\bfC$ is given by $R(T)\rightrightarrows U(T)$. Its stackification is called the \emph{quotient stack} of $(R\rightrightarrows U)$  and is denoted by $\big[U/R\big]$ (see \cite[Tag 02ZO]{stacks-project}).

\begin{example}\label{example_discretegroupquotient}
Let $G$ be a (discrete) group acting on an object $X$ of $\mathbf{C}$ by automorphisms. The operation of $G$ defines a geometric groupoid $(G\times X\rightrightarrows X)$ as follows:
\begin{itemize}
\item the source morphism $s\mathrel{\mathop:}G\times X\rightarrow X$ by 
\begin{equation*}
(g,x)\longmapsto x \ ,
\end{equation*}
\item the target morphism $t\mathrel{\mathop:}G\times X\rightarrow X$ by 
\begin{equation*}
(g,x)\longmapsto g\cdot x \ ,
\end{equation*}
\item the composition morphism $c\mathrel{\mathop:}(G\times X)\times_X (G\times X)\rightarrow (G\times X)$ by 
\begin{equation*}
\big((g,x),(g',x')\big)\longmapsto (g'g,x) \ ,
\end{equation*}
\item the inverse morphism $i\mathrel{\mathop:}G\times X\rightarrow G\times X$ by 
\begin{equation*}
(g,x)\longmapsto (g^{-1},g\cdot x) \ ,
\end{equation*} and
\item the unit morphism by $e \mathrel{\mathop:}X\rightarrow G\times X$ by 
\begin{equation*}
x\longmapsto (1,x) \ .
\end{equation*}
\end{itemize}
We write $\big[X/G\big]$ for its quotient stack. {By the next proposition, the quotient $\big[X/G\big]$ is a geometric stack.}
\end{example}

\begin{proposition}\label{prop_groupoidquot}
Let $(R\rightrightarrows U)$ be a geometric groupoid. Then the quotient stack $\big[U/R\big]$ is a geometric stack over $(\bfC,\tau,\PP)$. 
\end{proposition}

The proof of Proposition \ref{prop_groupoidquot} is well-known in particular contexts. We refer the reader to \cite[Tag 04TK]{stacks-project} for a proof in the algebraic context and to \cite[Proposition 2.14]{Ulirsch_trop=quot} for a proof in the context of non-Archimedean analytic geometry. 

Conversely, every geometric stack $\calX$ admits a {\em groupoid presentation}, by which we mean an equivalence $\big[U/R\big]\simeq \calX$ for a geometric groupoid $(R\rightrightarrows U)$. This may be seen as follows (see e.g. \cite[Tag 04T3]{stacks-project} for a proof for algebraic stacks): Let $U\rightarrow \calX$ be a morphism from an object $U$ in $\bfC$ onto $\calX$ that is a cover and in $\PP$. Then the fiber product $U\times_\calX U$ is representable by a geometric space $R$. The two projections $R\rightrightarrows U$ together with the natural composition map 
\begin{equation*}
R\times_U R\simeq U\times_\calF U\times_\calF U \xlongrightarrow{\ \ \ \ \pr_{02}\ \ \ } U\times_\calF U\simeq R \ ,
\end{equation*}
make $(R\rightrightarrows U)$ into a geometric groupoid, the \emph{symmetry groupoid}, such that $\big[U/R\big]\simeq \calX$. 

\begin{definition}
A \emph{morphism $f\mathrel{\mathop:}(\bfC,\tau,\PP)\rightarrow(\bfD,\sigma,\QQ)$ of geometric contexts} is a functor $u\mathrel{\mathop:}\bfD\rightarrow\bfC$ that respects fiber products, sends covers in $\sigma$ to covers in $\tau$ and morphisms in $\QQ$ to morphisms in $\PP$. 
\end{definition}

Given a geometric stack over $\bfD$ and a morphism of contexts $f$ as above, one may define a \emph{pullback $f^\ast\calX$} of $\calX$ along $f$; see \cite[Tag 04WA]{stacks-project} for details on this rather technical construction.  However, it is relatively easy to describe the pullback in the presence of a groupoid presentation:

\begin{proposition}\label{prop_pullback}
Let $f\mathrel{\mathop:}(\bfC,\tau,\PP)\rightarrow(\bfD,\sigma,\QQ)$ be a morphism of geometric contexts, let $\calX$ a geometric stack in the context $(\bfD,\sigma,\QQ)$, and $\big[U/R\big]\simeq \calX$ a geometric groupoid presentation of $\calX$. Then the induced groupoid $(f^\ast R\rightrightarrows f^\ast U)$ is geometric and we have a natural equivalence
\begin{equation*}
\big[f^\ast R/f^\ast U\big]\simeq f^\ast \calX \ .
\end{equation*}
\end{proposition}

So, in particular, by Proposition \ref{prop_groupoidquot} the pullback $f^\ast\calX$ is a geometric stack. The proof of Proposition \ref{prop_pullback} is completely formal and will be left to the avid reader; a particular instance of it can for example be found in \cite[Proposition 2.19]{Ulirsch_trop=quot} in the special case of the analytification functor for non-Archimedean analytic stacks. 

%Recall from \cite[Tag 04WA]{stacks-project} that the pullback $f^\ast\calX$ of $\calX$ along $f$ is the category fibered in groupoids defined as follows: Consider the category $u_{pp}\calX$, fibered over $\bfC$, whose fiber over an object $U$ in $\bfC$ is the groupoid of triples $(V,\phi\mathrel{\mathop:}U\rightarrow u(V), x)$, where 
%\begin{itemize}
%\item $V$ is an object of $\bfD$, 
%\item $\phi\mathrel{\mathop:}U\rightarrow u(V)$ is a morphism in $\bfC$, and
%\item $x$ is an object in $\calX/V$. 
%\end{itemize}
%A morphism 
%\begin{equation*}
%\big(V_1,\phi_1\mathrel{\mathop:}U_1\rightarrow u(V_1), x_1\big)\longrightarrow\big(V_2,\phi_2\mathrel{\mathop:}U_2\rightarrow u(V_2), x_2\big)
%\end{equation*}
%in $u_p\calX$ is given by a triple $(a,b,\alpha)$ where 
%\begin{itemize}
%\item $a\mathrel{\mathop:}V_1\rightarrow V_2$ is a morphism in $\bfD$, 
%\item $b\mathrel{\mathop:}U_1\rightarrow U_2$ is a morphism in $\bfC$, and
%\item $\alpha\mathrel{\mathop:}x_1\rightarrow x_2$ is a morphism in $\calX$
%\end{itemize}
%such that the image of $\alpha$ in $\bfD$ is equal to $a$ and the natural diagram
%\begin{equation*}
%\begin{CD}
%U_1@>b>> U_2\\
%@V\phi_1VV @VV\phi_2V\\
%u(V_1)@>u(a)>> u(V_2)
%\end{CD}
%\end{equation*}
%commutes in $\bfC$. Denote by $u_p\calX$ the localization of $u_{pp}\calX$ with respect to the right-multiplicative class of morphisms $(a,\id,\alpha)$ with $\alpha$ strongly cartesian. Then $f^\ast\calX$ is defined as the stackification of $u_p\calX$. 

\section{Cone spaces and cone stacks}\label{section_conestacks}
In this section we describe the structure of a geometric context on the category of rational polyhedral cone complexes and, using the language developed in Section \ref{section_geometricstacks}, the notions of cone spaces and cone stacks. This is the setting in which we will describe the tropical moduli stack in Section \ref{section_geometrization}. In Section \ref{section_combinatorialconestacks}  we provide an alternative combinatorial point of view on cone spaces and cone stacks that, in particular, naturally incorporates the notion of \emph{generalized cone complexes}, as introduced in \cite{AbramovichCaporasoPayne_tropicalmoduli}.

\subsection{Cone spaces and cone stacks} A \emph{rational polyhedral cone} is a pair $(N,\sigma)$ consisting of a finitely generated free abelian group $N$ and a finite intersection of half spaces in $N_\R=N\otimes\R$.
We assume rational polyhedral cones to be {\em strictly convex}, i.e. $\sigma$ contains no nontrivial linear spaces of $N_\R$.
A $\Z$-\emph{linear morphism} $(N,\sigma)\rightarrow (N',\sigma')$ of rational polyhedral cones is a homomorphism $f\mathrel{\mathop:}N\rightarrow N'$ of abelian groups such that $f(\sigma)\subseteq \sigma'$. 

Given a rational polyhedral cone $(N,\sigma)$, we write $M$ for the dual lattice $\Hom_\Z(N,\Z)$ and denote by $\langle\:-\:,\:-\:\rangle$ the duality pairing between $M$ and $N$. We may associate to a rational polyhedral cone $(N,\sigma)$ a {\em toric monoid} $S_\sigma=\sigma^\vee\cap M$ where we denote by
\begin{equation*}
\sigma^\vee=\big\{v\in M_\R\: \big\vert \: \langle v,u\rangle \ge 0 \textrm{ for all } u\in \sigma\big\}
\end{equation*}
the {\em dual cone} of $\sigma$. It is well-known (see, e.g., \cite[Proposition 2.2]{Ulirsch_functroplogsch}) that the association $(N,\sigma)\mapsto S_\sigma$ yields an equivalence between the category of rational polyhedral cones and the category of toric monoids. 

Under this equivalence, we have that $\Span{\sigma}=N_\R$ if and only if the dual monoid $S_\sigma$ is \emph{sharp}, i.e., if it has no non-trivial units. Denote the category of rational polyhedral cones $(N,\sigma)$ such that $S_\sigma$ is sharp with $\Z$-linear morphisms by $\RPC$; from now on we refer to an object in $\RPC$ simply as a \emph{rational polyhedral cone} or a \emph{cone}.  There is a faithful {\em topological realization} functor $\RPC\rightarrow \mathbf{Top}$ sending $(N,\sigma)$ to $\sigma$, thought of as a topological space.

A morphism $\tau\rightarrow\sigma$ in $\RPC$ is said to be a \emph{face morphism} if it induces an isomorphism of $\tau$ onto a face of $\sigma$. A \emph{proper face morphism} $\tau\rightarrow\sigma$ is a face morphism that is not an isomorphism onto $\sigma$, i.e., $\tau$ is isomorphic to a proper face of $\sigma$.

In the following Definition \ref{definition_conecomplexes} we give a new description of the classical notion of a \emph{(rational polyhedral) cone complex}, as introduced in \cite[Definition II.1.5]{KKMSD_toroidal} and, e.g., in \cite[Section 2.2]{AbramovichCaporasoPayne_tropicalmoduli} (see Remark \ref{rem:acp-def}). 

\begin{definition}\label{definition_conecomplexes}
A \emph{(rational polyhedral) cone complex} $\Sigma$ is a collection of cones $\mathcal{C}=\{\sigma_\alpha\}$ in $\RPC$ together with a collection of face morphisms $\mathcal{F}=\{\phi_{\alpha\beta}\mathrel{\mathop:}\sigma_\alpha\rightarrow\sigma_\beta\}$, closed under composition, such that the following three axioms are fulfilled:
\begin{enumerate}[(i)]
\item For each $\sigma_\alpha\in\mathcal{C}$, the identity map $\operatorname{id}\colon \sigma_\alpha\rightarrow\sigma_\alpha$ is in $\mathcal{F}$.
\item Every face of a cone $\sigma_\beta$ is the image of exactly one face morphism $\phi_{\alpha\beta}$ in $\mcF$.
\item There is at most one morphism $\sigma_\alpha\rightarrow\sigma_\beta$ between two cones $\sigma_\alpha$ and $\sigma_\beta$ in $\mcF$.
\end{enumerate}
\end{definition}

A \emph{morphism $f\mathrel{\mathop:}\Sigma_1\rightarrow\Sigma_2$ of cone complexes} is a choice, for each $\sigma_{\alpha_1}\in\mcC_1$, of a cone $\sigma_{\alpha_2}\in\mcC_2$ and a morphism $f_{\alpha_1\alpha_2}\mathrel{\mathop:}\sigma_{\alpha_1}\rightarrow\sigma_{\alpha_2}$ that does not factor through any proper face morphism $\tau\rightarrow\alpha_2$, such that for every $\phi_{\alpha_1\beta_1}$ in $\mcF_1$, there exists $\phi_{\alpha_2\beta_2}\colon \sigma_{\alpha_2} \rightarrow\sigma_{\beta_2}$ in $\mcF_2$
making the diagram
\begin{equation*}\xymatrix{
\sigma_{\alpha_1} \ar[r]^-{f_{\alpha_1\alpha_2}} \ar[d]_{\phi_{\alpha_1\beta_1}} & \sigma_{\alpha_2} \ar[d]^{\phi_{\alpha_2\beta_2}} \\
\sigma_{\beta_1} \ar[r]^-{f_{\beta_1\beta_2}} & \sigma_{\beta_2}
} \end{equation*}
%\begin{equation*}
%\begin{CD}
%\sigma_{\alpha_1}@>f_{\alpha_1\alpha_2}>>\sigma_{\alpha_2}\\
%@V\phi_{\alpha_1\beta_1}VV @VV\phi_{\alpha_2\beta_2}V\\
%\sigma_{\beta_1}@>{f_{\beta_1\beta_2}}>> \sigma_{\beta_2}
%\end{CD}\end{equation*}
commute. 

Denote the category of (rational polyhedral) cone complexes by $\RPCC$. There is again a topological realization functor $\RPCC\to\mathbf{Top}$ associating to a cone complex $\Sigma$ its \emph{topological realization} $\vert\Sigma\vert$, the colimit (in the category $\mathbf{Top}$) of the diagram of spaces defined by $\Sigma$. Note that $\Sigma_1\rightarrow\Sigma_2$ naturally induces a continuous map $\vert\Sigma_1\vert\rightarrow \vert\Sigma_2\vert$ that restricts to the defining maps on the faces. Condition (ii) in Definition \ref{definition_conecomplexes} 
implies that for all $\alpha$ the natural map $\sigma_\alpha\rightarrow \vert\Sigma\vert$ is injective.
%In other words, the diagram $\Sigma_{\geq 1}$, obtained by removing all the $0$-dimensional cones $\simeq\{pt\}$, is a tree.   
%%%%% 
We refer to the collection $\{\sigma_\alpha\}$ as the \emph{faces} of $\Sigma$.

\begin{remark}\label{rem:acp-def} 
The category $\RPCC$ is equivalent to the category of {\em rational polyhedral cone complexes}, or simply {\em cone complexes}, that was defined  in \cite[\S2.2]{AbramovichCaporasoPayne_tropicalmoduli}.  The definition there, mildly rephrased, is that a cone complex is a topological space $\vert\Sigma\vert$ % with a collection of closed subspaces identified with rational polyhedral cones, such that under this identification the intersection of any two cones is a union of faces of each.  
presented as the colimit of a diagram of rational polyhedral cones and face morphisms, with the property that for each cone $\sigma$ the natural map $\sigma\to \vert \Sigma\vert$ is injective.  
A morphism of cone complexes $f\colon \Sigma\to\Sigma'$ is a morphism of topological spaces with the property that for every cone $\sigma$ in $\Sigma$, there is a cone $\sigma'$ in $\Sigma'$ such that $f|_\sigma$ factors through a morphism $\sigma \to \sigma'$.  The equivalence is obtained by associating to $\Sigma$ in $\RPCC$ the topological realization $|\Sigma|$ together with the category $\Sigma$.
\end{remark}

There is a full and faithful embedding $\RPC\rightarrow \RPCC$ that associates to a cone $\sigma$ the cone complex consisting of all of its faces, which we again denote by $\sigma$. Note that for every cone $\sigma_\alpha$ of a cone complex $\Sigma$, there is a natural morphism $\sigma_\alpha\rightarrow\Sigma$. As $\sigma_\alpha$ ranges over the cones of $\Sigma$, the morphisms $\sigma_\alpha\rightarrow\Sigma$ generate a Grothendieck topology $\tau_{face}$ on $\RPCC$. We refer to this topology as the \emph{face topology} and write $\big(\RPCC,\tau_{face}\big)$ for the corresponding site. 

The face topology on $\RPCC$ is very coarse. As a consequence we have the following useful characterization of stacks.

\begin{proposition} \label{prop:cfg-stack}
The $2$-category of stacks on $\RPCC$ is equivalent to the $2$-category of categories fibered in groupoids on $\RPC$ via the natural restriction.
\end{proposition}
\begin{proof}
Since every rational polyhedral cone complex has a cover by rational polyhedral cones, stacks on $\RPCC$ are equivalent to stacks on $\RPC$.  But rational polyhedral cones have no nontrivial covers in the face topology, so stacks on $\RPC$ are the same as categories fibered in groupoids on $\RPC$.
\end{proof}

\begin{remark}\label{remark_explicitextension}
Given a category fibered in groupoids $\mathcal F$ over $\RPC$, the fiber $\mathcal F(\Sigma)$ of its extension to a stack over $\RPCC$ over a cone complex $\Sigma$ is the groupoid of collections $(\xi_\sigma)_{\sigma \in \Sigma}$ together with a collection of isomorphisms $\xi_\tau \cong i^\ast \xi_\sigma$ for every face morphism $i : \tau\hookrightarrow\sigma$ in $\Sigma$ that are compatible with composition of face morphisms.
\end{remark}

\begin{definition}
A morphism $f\mathrel{\mathop:}\Sigma_1\rightarrow\Sigma_2$ of cone complexes is called \emph{strict} if the morphisms $f_{\alpha_1\alpha_2}\mathrel{\mathop:}\sigma_{\alpha_1}\rightarrow\sigma_{\alpha_2}$ are all isomorphisms. 
\end{definition}

In the following we write $\SS$ for the class of strict morphisms. 

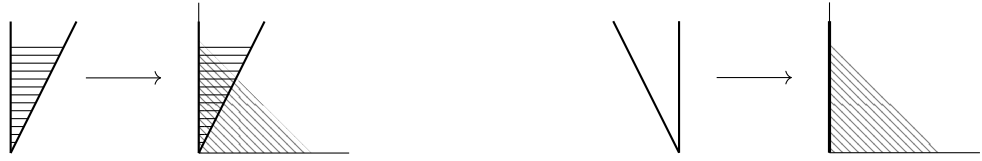
\begin{figure}
\begin{minipage}{0.49\textwidth}\centering
\begin{tikzpicture}

\fill[thick, pattern=horizontal lines] (0,0)-- (0.75, 1.5) -- (0,1.5) -- (0,0);
\draw[thick] (0,0) -- (0,1.75);
\draw[thick] (0,0) -- (0.875,1.750);

\draw[->] (1,1) -- (2,1);

\fill[pattern color=gray,pattern=north west lines] (2.5,0) -- (4,0) -- (2.5,1.5) -- (2.5,0);
\draw (2.5,0) -- (2.5,2);
\draw (2.5,0) -- (4.5,0);

\fill[thick, pattern=horizontal lines] (2.5,0) -- (3.25,1.5) -- (2.5,1.5) -- (2.5,0);
\draw[thick] (2.5,0) -- (2.5,1.75);
\draw[thick] (2.5,0) -- (3.375,1.75);

\end{tikzpicture}
\end{minipage}
\begin{minipage}{0.49\textwidth}\centering
\begin{tikzpicture}
\draw[thick] (0,0) -- (0,1.75);
\draw[thick] (0,0) -- (-0.875,1.75);

\draw[->] (0.5,1) -- (1.5,1);

\fill[pattern color=gray,pattern=north west lines] (2,0) -- (3.5,0) -- (2,1.5) -- (2,0);
\draw (2,0) -- (2,2);
\draw (2,0) -- (4,0);

\draw[very thick] (2,0) -- (2,1.75);
\end{tikzpicture}
\end{minipage}
\caption{A non-strict morphism (on the left) and a strict morphism (on the right).}
\end{figure}

\begin{proposition}
The triple $\big(\RPCC,\tau_{face},\SS\big)$ defines a geometric context. 
\end{proposition}

\begin{proof}
We  show properties (i)-(iv) from Definition \ref{def_context}. For (i) it is enough to show that $\RPCC$ admits fiber products. Let $\Sigma_i\rightarrow T$ be two morphisms of cone complexes. In order to construct the fibered product $\Sigma_1\times_T\Sigma_2$ along two morphisms $\phi_i\mathrel{\mathop:}\Sigma_i\rightarrow T$ we may assume that $\Sigma_i=(N_i,\sigma_i)$ and $T=(N_\tau,\tau)$ are cones (and then proceed in analogy with the construction from affine schemes to schemes in classical scheme-theoretic algebraic geometry). In this case, the fiber product $\sigma_1\times_\tau\sigma_2$ lives in the the vector space spanned by the lattice $N_1\times_{N_\tau} N_2$, and it is given by  
\begin{equation*}
\sigma_1\times_\tau\sigma_2=\big\{(s_1,s_2)\in\sigma_1\times\sigma_2\:\big\vert\: \phi_1(s_1)=\phi_2(s_2)\big\} \ .
\end{equation*}

Property (ii), i.e. that for every cone complex $\Sigma$ the presheaf 
\begin{equation*}\begin{split}
h_\Sigma\mathrel{\mathop:}\RPCC^{op}&\longrightarrow\mathbf{Sets}\\
T&\longmapsto \Hom(T,\Sigma)
\end{split}\end{equation*}
is a sheaf in the face topology, is an immediate consequence of the local definition of $\Sigma$. For (iii) it is clear that the class of strict morphisms is stable under compositions and base changes, it contains the face morphisms, and all isomorphisms are strict. Finally, for (iv) we note that the property of being strict is local on the faces of a cone complex.
\end{proof}

\begin{definition}\label{def_conestacks}
A geometric space in the context $(\RPCC,\tau_{face},\mathbb{S})$ is called a \emph{cone space}, and a geometric stack a \emph{cone stack}.
\end{definition}

The category of cone spaces is the full subcategory of the category of presheaves on $\RPCC$ whose objects are isomorphic to a cone space, and the $2$-category of cone stacks is the full subcategory of the $2$-category of categories fibered in groupoids over $\RPCC$ whose objects are equivalent to cone stacks.

\begin{example}
Let $G$ be a (discrete) group acting on a cone complex $\Sigma$ by automorphisms. As explained in Example \ref{example_discretegroupquotient}  the operation of $G$ on $\Sigma$ defines a strict groupoid object $(G\times \Sigma\rightrightarrows \Sigma)$ in $\RPCC$. The quotient stack $\big[\Sigma\big/G\big]$ is a cone stack by  Proposition \ref{prop_groupoidquot}.
\end{example}

\begin{remark}
An operation of a (discrete) group $G$ on a cone complex $\Sigma$ by automorphisms is the same thing as an admissible operation in the sense of \cite{ChanMeloViviani_tropicalmoduli}.
\end{remark}

In the following Example \ref{example_loopcone}, we define a cone space $\calW$, the so-called \emph{waffle cone}, that is used to model loop edges in the construction of the cone $\Cone(\Gamma)$ over a tropical curve $\Gamma$ in Section \ref{section_families}. In Example \ref{example_combinatorialwafflecone}  we provide an alternative construction of $\calW$, from which we can see that it is not representable by a cone complex.

\begin{example}[Waffle cone]\label{example_loopcone}
Consider the cone complex $\Sigma$ consisting of two copies $\sigma_i\simeq\R_{\geq 0}^2$ (for $i=1,2$) that are glued along the $x$-face and $y$-face respectively. In addition to the identity morphisms we define additional relations by sending $(x,y)\in\sigma_1$ to $(y,x)\in\sigma_2$ (and conversely) provided that $(x,y)\neq (0,0)$.  This induces a groupoid presentation $R=\Sigma\sqcup_{\{pt.\}}\Sigma\rightrightarrows \Sigma$ of a stack $\calW$ that is fibered in setoids, and therefore a cone space (see Figure \ref{figure_waffle}). 
\end{example}

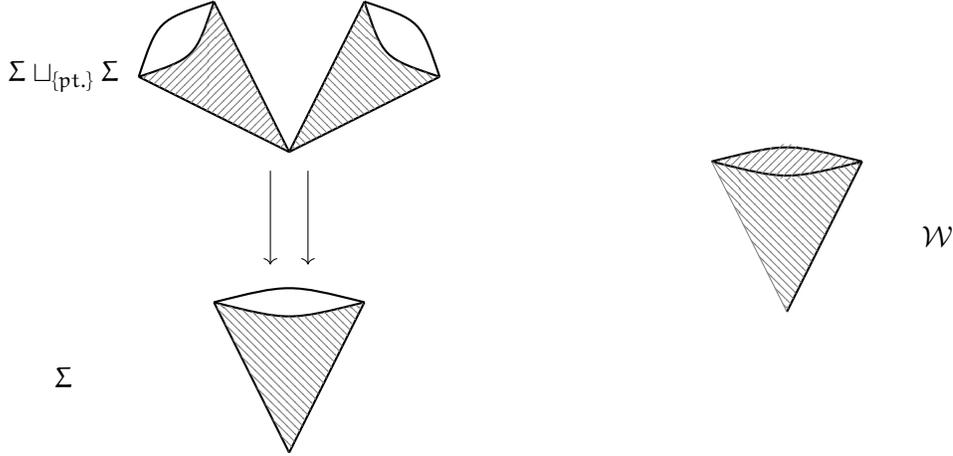
\begin{figure}[h]
\begin{minipage}{0.49\textwidth}
\centering
\begin{tikzpicture}

%left cone
\draw[thick] (0,0) -- (-2,1);
\draw[thick] (0,0) -- (-1,2);
\draw[thick] (-2,1) .. controls (-1.25,1.25) and (-1.25,1.25) .. (-1,2);
\draw[thick] (-2,1) .. controls (-1.75,1.75) and (-1.75,1.75) .. (-1,2);
\fill[pattern color=gray,pattern=north east lines] (-2,1) .. controls (-1.25,1.25) and (-1.25,1.25) .. (-1,2) -- (0,0) -- (-2,1);

%right cone
\draw[thick] (0,0) -- (2,1);
\draw[thick] (0,0) -- (1,2); 
\draw[thick] (2,1) .. controls (1.25,1.25) and (1.25,1.25) .. (1,2);
\draw[thick] (2,1) .. controls (1.75,1.75) and (1.75,1.75) .. (1,2);
\fill[pattern color=gray,pattern=north west lines] (2,1) .. controls (1.25,1.25) and (1.25,1.25) .. (1,2) -- (0,0) -- (2,1);

%central cone
\draw[thick] (-1,-2) -- (0,-4);
\draw[thick] (1,-2) -- (0,-4);
\draw[thick] (-1,-2) .. controls (0,-1.75) and (0,-1.75) .. (1,-2);
\draw[thick] (-1,-2) .. controls (0,-2.25) and (0,-2.25) .. (1,-2);

\fill[pattern color=gray,pattern=north west lines] (-1,-2) .. controls (0,-2.25) and (0,-2.25) .. (1,-2) -- (0,-4) -- (-1,-2);

%arrows
\draw[->] (-0.25,-0.25) -- (-0.25,-1.5);
\draw[->] (0.25,-0.25) -- (0.25,-1.5);

\node at (-3,1) {$\Sigma\sqcup_{\{pt.\}}\Sigma$};
\node at (-3,-3) {$\Sigma$};
\end{tikzpicture}
\end{minipage}
\begin{minipage}{0.49\textwidth}
\centering
\begin{tikzpicture}
\draw[thick] (-1,0) .. controls (0,0.25) and (0,0.25) .. (1,0);
\draw[thick] (-1,0) .. controls (0,-0.25) and (0,-0.25) .. (1,0);
\draw[thick] (1,0) -- (0,-2);
\draw[gray] (-1,0) -- (0,-2);

\fill[pattern color=gray,pattern=north west lines] (-1,0) .. controls (0,-0.25) and (0,-0.25) .. (0,-0.225) -- (0,-2) -- (-1,0);
\fill[pattern color=gray,pattern=north west lines] (1,0) .. controls (0,-0.25) and (0,-0.25) .. (0,-0.225) -- (0,-2) -- (1,0);
\fill[pattern color=gray, pattern=north east lines] (-1,0) .. controls (0,0.25) and (0,0.25) .. (0,0.225) -- (0,-0.25) -- (-1,0);
\fill[pattern color=gray, pattern=north east lines] (1,0) .. controls (0,0.25) and (0,0.25) .. (0,0.225) -- (0,-0.25) -- (1,0);

\node at (2,-1) {$\calW$};
\end{tikzpicture}
\end{minipage}
\caption{The groupoid $R\rightrightarrows \Sigma$ from Example \ref{example_loopcone} (on the left) and its quotient, the waffle cone $\calW$ (on the right).}\label{figure_waffle}
\end{figure}

The following Lemma \ref{lemma_diagonalrepresentable} will be used in the proof of Theorem \ref{thm_geometric} in Section \ref{section_geometrization}.  It says that to check that a stack over $(\RPCC,\tau_{face})$ is a cone stack, it is sufficient to check that it has a covering by a cone complex.

\begin{lemma}\label{lemma_diagonalrepresentable}
Let $\calC$ be a stack over $(\RPCC,\tau_{face})$ and suppose that there is a cone complex $U$ and a surjective, strict morphism $U\rightarrow \calC$ that is representable by cone complexes (respectively, by cone spaces). Then the diagonal morphism $\Delta\mathrel{\mathop:}\calC\rightarrow\calC\times\calC$ is representable by cone complexes (respectively, by cone spaces).
\end{lemma}

\begin{proof}

Given a rational polyhedral cone $\sigma$ and a morphism $\sigma\rightarrow \calC\times \calC$, it suffices to show that the fiber product
\begin{equation*}
X=\sigma\times_{\calC\times\calC}\calC
\end{equation*}
is representable by a cone complex (resp.~cone space). There is a natural $2$-cartesian diagram 
\begin{equation*}\xymatrix{
\displaystyle U\times_{\mathcal C} U \ar[r] \ar[d] & \mathcal C \ar[d]^\Delta \\
U\times U \ar[r] & \mathcal C \times \mathcal C 
}\end{equation*}
%\begin{equation*}\begin{CD}
%U\times_\calC U@>>> \calC\\
%@VVV @VV\Delta V\\
%U\times U@>>> \calC\times\calC
%\end{CD}\end{equation*}
identifying the fiber product with $U\times_\calC U$. Notice that $U\times_\calC U$ may be represented by a cone complex (resp.~cone space), because $U\times_\calC U \to U$ is a base change of the morphism $U\to \calC$.

Next we note that 
$U \times U\rightarrow\calC \times \calC$ is strict, surjective, and representable by cone complexes (resp.~cone spaces), as it is the composition of the morphisms $U \times U \rightarrow U \times \calC \rightarrow \calC \times \calC$, which are both obtained by base change from the morphism $U \rightarrow \calC$.  Then we claim that the map $\sigma\to\calC\times\calC$ factors (not necessarily uniquely) through $U\times U$, since in the following 2-cartesian diagram

\begin{equation*}\xymatrix{
\displaystyle \sigma \times_{\calC \times \calC} U\times U \ar[r] \ar[d] &  U\times U \ar[d] \\
\sigma\ar[r] & \calC\times\calC
}\end{equation*}
the morphism $\sigma \times_{\calC\times\calC}U\times U \to \sigma$ is a strict surjective morphism from a cone complex (resp.~cone space) to a cone, hence admits a section.

Therefore we have a natural composition of $2$-cartesian diagrams
\begin{equation*}\xymatrix{
X \ar[r] \ar[d] & \displaystyle U \mathop\times_{\calC} U \ar[r] \ar[d] & \calC \ar[d]^\Delta \\
\sigma \ar[r] & U \times U \ar[r] & \calC \times \calC
} \end{equation*}
%\begin{equation*}\begin{CD}
%X@>>>U\times_\calC U@>>> \calC\\
%@VVV @VVV @VV\Delta V\\
%\sigma @>>> U\times U@>>> \calC\times\calC
%\end{CD}\end{equation*}
and, since $U\times_\calC U$ is represented by a cone complex (resp.~cone space), so is $X$.

\end{proof}

\subsection{Combinatorial cone spaces and combinatorial cone stacks}\label{section_combinatorialconestacks} We give a combinatorial characterization of cone spaces in the spirit of Definition \ref{definition_conecomplexes}. The only change is that we drop the assumption that there is at most one morphism between any two cones.  %We call this the category of combinatorial cone spaces; we will then show it is equivalent to the category of cone spaces.
We will then define combinatorial cone stacks.

\begin{definition} \label{def:C-space}
A {\em combinatorial cone space} is a collection of cones $\mathcal{C}=\{\sigma_\alpha\}$ in $\RPC$ together with a collection of face morphisms $\mathcal{F}=\big\{\phi_{\alpha\beta}\mathrel{\mathop:}\sigma_\alpha\rightarrow\sigma_\beta\big\}$, closed under composition, such that the following two axioms are fulfilled:
\begin{enumerate}[(i)]
\item For each $\sigma_\alpha\in\mathcal{C}$, the identity map $\operatorname{id}\colon \sigma_\alpha\rightarrow\sigma_\alpha$ is in $\mathcal{F}$.
\item Every face of a cone $\sigma_\beta$ is the image of exactly one face morphism $\phi_{\alpha\beta}$ in $\mcF$. 
\end{enumerate}
\end{definition}
A morphism of combinatorial cone spaces is defined exactly as in the definition of morphisms of cone complexes in Definition~\ref{definition_conecomplexes}.

\begin{proposition} \label{prop:C-spaces}
The category of combinatorial cone spaces is equivalent to the category of cone spaces.
\end{proposition}

This will be proved as part of Proposition~\ref{prop:C-stacks}.

\begin{example}[Combinatorial waffle cone]\label{example_combinatorialwafflecone}
The waffle cone $\calW$ from Example \ref{example_loopcone} is the combinatorial cone space generated by the two cones $\sigma=\R_{\geq 0}^2$ and $\tau=\R_{\geq 0}$ and the two proper face morphisms $\tau\rightarrow\sigma$ (see Figure \ref{figure_combinatorialwafflecone}). We have suppressed the initial $0$-dimensional cone in our description. By Proposition \ref{prop:C-spaces} this description of $\calW$ is equivalent to a cone space that is not representable by a cone complex. 
\end{example}

\begin{figure}[h]
\begin{minipage}{0.49\textwidth}
\centering
\begin{tikzpicture}
%left cone
\draw[thick] (0,0) -- (2,0);
\fill (0,0) circle (0.05);
\node at (-0.5,0) {$\tau$};

%right cone
\draw[thick] (3,0) -- (5,1);
\draw[thick] (3,0) -- (5,-1);
\fill (3,0) circle (0.05);
\fill[pattern color=gray,pattern=north east lines] (3,0) -- (4.5,0.75) -- (4.5,-0.75) -- (3,0); 
\node at (5,0) {$\sigma$};

%arrows
\draw[thick,->] (1,0.5) .. controls (2,1.25) and (2,1.25) .. (3.5,0.75);
\draw[thick,->] (1,-0.5) .. controls (2,-1.25) and (2,-1.25) .. (3.5,-0.75);
\end{tikzpicture}
\end{minipage}
\begin{minipage}{0.49\textwidth}
\centering
\begin{tikzpicture}
\draw[thick] (-1,0) .. controls (0,0.25) and (0,0.25) .. (1,0);
\draw[thick] (-1,0) .. controls (0,-0.25) and (0,-0.25) .. (1,0);
\draw[thick] (1,0) -- (0,-2);
\draw[gray] (-1,0) -- (0,-2);
\fill (0,-2) circle (0.05);

\fill[pattern color=gray,pattern=north west lines] (-1,0) .. controls (0,-0.25) and (0,-0.25) .. (0,-0.225) -- (0,-2) -- (-1,0);
\fill[pattern color=gray,pattern=north west lines] (1,0) .. controls (0,-0.25) and (0,-0.25) .. (0,-0.225) -- (0,-2) -- (1,0);
\fill[pattern color=gray, pattern=north east lines] (-1,0) .. controls (0,0.25) and (0,0.25) .. (0,0.225) -- (0,-0.25) -- (-1,0);
\fill[pattern color=gray, pattern=north east lines] (1,0) .. controls (0,0.25) and (0,0.25) .. (0,0.225) -- (0,-0.25) -- (1,0);

\node at (2,-1) {$\calW$};
\end{tikzpicture}
\end{minipage}
\caption{The waffle cone $\calW$ (on the right) and its presentation as a combinatorial cone space (on the left) as in Example \ref{example_combinatorialwafflecone}.}\label{figure_combinatorialwafflecone}
\end{figure}
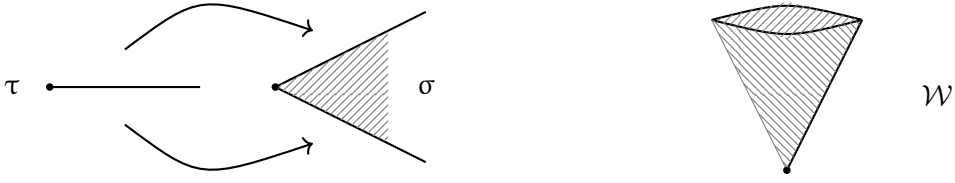

In Definition \ref{def_combinatorialconestack} we provide a combinatorial characterization of cone stacks. This should be viewed as a categorical perspective on the \emph{generalized cone complexes} of \cite[Section 2.6]{AbramovichCaporasoPayne_tropicalmoduli} (also see Remark \ref{rem:make_your_own_cone_stack} for a special case that is inspired by \cite{ChanGalatiusPayne_tropicalmoduliII}).

\begin{definition}\label{def_combinatorialconestack}
Let $\RPC^f$ denote the category of rational polyhedral cones with face inclusions as morphisms.  A \emph{combinatorial cone stack} is a category fibered in groupoids $\sigma : \Sigma \rightarrow \RPC^f$.  A morphism of combinatorial cone stacks $(\Sigma,\sigma) \rightarrow (T,\tau)$ consists of a functor $\Phi : \Sigma \rightarrow T$ and morphism $\phi_\alpha : \sigma_\alpha \rightarrow \tau_{\Phi(\alpha)}$ for each $\alpha \in \Sigma$, such that
\begin{itemize}
\item for each $\alpha \rightarrow \beta$ in $\Sigma$, the diagram
\begin{equation*} \xymatrix{
\sigma_\alpha \ar[r] \ar[d] & \tau_{\Phi(\alpha)} \ar[d] \\
\sigma_\beta \ar[r] & \tau_{\Phi(\beta)}
} \end{equation*}
commutes, and
\item the image of $\sigma_\alpha \rightarrow \tau_{\Phi(\alpha)}$ is not contained in any proper face.
\end{itemize}
\end{definition}

We recall that in this setting, the condition that $\sigma : \Sigma \rightarrow \RPC^f$ be a category fibered in groupoids means that it satisfies the following two conditions:
\begin{enumerate}[(i)]
\item for each $\alpha \in \Sigma$ and each face inclusion $\tau \to \sigma_\alpha$, there is a morphism $u : \beta \rightarrow \alpha$ of $\Sigma$ such that $\sigma_\beta = \tau$ and $\sigma_u$ is the inclusion of $\tau$ in $\sigma_\alpha$;

\item if $u : \gamma \rightarrow \alpha$ and $v : \beta \rightarrow \alpha$ are morphisms of $\Sigma$ such that the map $\sigma_\gamma \to \sigma_\alpha$ factors through $\sigma_\beta \to \sigma_\alpha$ then there is a unique morphism $w : \gamma \rightarrow \beta$ such that $uw = v$ and 
$\sigma_w$ is the inclusion of $\sigma_\gamma$ in $\sigma_\beta$.

\end{enumerate}

\begin{remark} 
\label{rem:cone-complex}
\label{rem:combinatorial-cone-space}
Cone complexes and combinatorial cone spaces may now be seen as special kinds of categories fibered in groupoids over $\RPC$.  
A combinatorial cone space is a category fibered in groupoids $\sigma\colon \Sigma\to \RPC^f$ in which $\Sigma$ has no nontrivial automorphisms, and a cone complex is a category fibered in groupoids whose underlying category is equivalent to a partially ordered set.
\end{remark}

We shall prove that combinatorial cone stacks and cone stacks are 2-equivalent.
The following lemma is an analogue of a statement proved for Artin fans in \cite[Proposition~2.3.11]{AbramovichWise_invariance}, and is equivalent to that statement by Theorem~\ref{thm_artinfans}.  It will be key to the equivalence between cone stacks and combinatorial cone stacks, to be proved in Proposiiton~\ref{prop:C-stacks}.  The proof of the lemma is essentially the same here as in \cite{AbramovichWise_invariance}, but we give the details here, free of the additional baggage from logarithmic geometry and algebraic stacks. 

\begin{remark}\label{rem:caution,caution}
In Lemma~\ref{lem:initial} and Proposition~\ref{prop:C-stacks}, it will be important to remember that a commutative diagram of stacks is not determined by the morphisms that appear in it, but requires the additional specification of $2$-isomorphisms.  For example, a commutative square
\begin{equation*} \xymatrix{
X \ar[r]^f \ar[d]_p & Y \ar[d]^q \\
Z \ar[r]^g & W
} \end{equation*}
of stacks involves a $2$-isomorphism between $qf$ and $gp$.

\end{remark}

\begin{lemma} \label{lem:initial}
Let $\mathcal C$ be a cone stack.  For any cone $\sigma$ and any morphism $\sigma \rightarrow \mathcal C$, there is an initial factorization $\sigma \rightarrow \tau \rightarrow \mathcal C$ where $\tau$ is a cone and $\tau \rightarrow \mathcal C$ is strict.
\end{lemma}
\begin{proof}
First we find $\tau$.  Since $\mathcal C$ is a cone stack, it admits a strict surjection $\coprod \omega_i \rightarrow \mathcal C$ from a disjoint union of cones.  Therefore $\coprod (\omega_i \mathbin\times_{\mathcal C} \sigma) \rightarrow \sigma$ is covering.  Since $\sigma$ is a cone, it has no nontrivial covers, so at least one of the $\omega_i \mathbin\times_{\mathcal C} \sigma \rightarrow \sigma$ has a section.  Composing with the projection $\omega_i \mathbin\times_{\mathcal C} \sigma \rightarrow \omega_i$, we obtain a commutative diagram:
\begin{equation*} \xymatrix{
& \omega_i \ar[d] \\
\sigma \ar[ur] \ar[r] & \mathcal C
} \end{equation*}
Let $\tau$ be the smallest face of $\omega_i$ containing the image of $\sigma$.  Note that $\tau \rightarrow \mathcal C$ is strict because it is the composition of the inclusion of a face with the strict map $\omega_i \rightarrow \mathcal C$.

Now we prove that $\tau$ is initial.  Suppose we have another factorization $\sigma \rightarrow \tau' \rightarrow \mathcal C$.  Then we have a commutative diagram of solid arrows:
\begin{equation*} \xymatrix{
\sigma \ar[r] \ar[d] & \tau' \ar[d] \\
\tau \ar@{-->}[ur] \ar[r] & \mathcal C
} \end{equation*}
We wish to show that there is a unique dashed arrow making the diagram commute.  The projection $\tau \mathbin\times_{\mathcal C} \tau' \rightarrow \tau$ is strict.  As $\tau \times_\calC \tau'$ is a cone space, it has a strict cover by a disjoint union of cones $\coprod \omega_i$.  As in the last paragraph, the map $\sigma \to \tau\times_\calC \tau'$ must factor through at least one $\omega_i$ because $\sigma$ has no nontrivial strict covers.  We therefore have a commutative diagram:
\begin{equation*} \xymatrix{
\sigma \ar[r] \ar@/^15pt/[rr] \ar[dr] & \omega_i \ar[r] \ar[d] & \tau' \ar[d] \\
& \tau  \ar[r] & \mathcal C
} \end{equation*}
But the morphism of cones $\omega_i \rightarrow \tau$ is strict, as it is the composition of the strict maps $\omega_i \rightarrow \tau \times_\calC \tau' \rightarrow \tau$, and surjective, as its image contains $\sigma$, which is in the interior of $\tau$, hence is an isomorphism.  Composing with the inverse of this isomorphism gives the desired arrow $\tau \rightarrow \tau'$.

All that remains is to show this lift is unique.  Consider a pair of dashed arrows $f$ and $g$ completing the diagram.  These give us a commutative diagram of solid arrows:
\begin{equation*} \xymatrix{
\sigma \ar[r] \ar[d] & \tau' \ar[d]^{\Delta} \\
\tau \ar@{-->}[ur] \ar[r]_-{(f,g)} & \tau' \mathbin\times_{\mathcal C} \tau'
} \end{equation*}
We wish to find a lift to show that $f = g$.  The diagonal $\tau' \rightarrow \tau' \mathbin\times_{\mathcal C} \tau'$ is strict, since $\tau'$ was assumed to be strict over $\mathcal C$.  We write $\operatorname{eq}(f,g) := \tau' \mathbin\times_{\tau' \mathbin\times_{\mathcal C} \tau'} \tau$ (the equalizer of $f$ and $g$).  The map $\operatorname{eq}(f,g) \rightarrow \tau$ is thus strict as it is the pullback of a strict map.  But it is also surjective:  it has a strict cover by cones $\eta_i \rightarrow \operatorname{eq}(f,g)$, so the maps $\eta_i \rightarrow \tau$ are strict.  But at least one of these must contain $\sigma$ in its image.  As $\sigma$ is in the interior of $\tau$ --- it is not contained in any face of $\tau$ by assumption --- this means at least one $\eta_i$ must surject onto $\tau$.  Therefore $\operatorname{eq}(f,g)$ admits a section over $\tau$; composing with the projection to $\tau'$ gives the desired lift and completes the proof.
\end{proof}

\begin{proposition} \label{prop:C-stacks}
The $2$-category of cone stacks is equivalent to the $2$-category of combinatorial cone stacks. Under this equivalence cone spaces correspond to combinatorial cone spaces.
\end{proposition}
\begin{proof}
Note that cone complexes are examples of combinatorial cone %spaces, 
stacks, so a combinatorial cone stack $\sigma\colon\Sigma\to\RPC^f$ induces a stack %$\calC = \operatorname{Hom}_{\text{$C$-space}}(-,\Sigma)$ 
$\calC = \operatorname{Hom}_{\text{$C$-stack}}(-,\sigma)$
on $\RPCC$.  We  argue that this stack is a cone stack by showing that $\coprod_{\alpha\in\Sigma} \sigma(\alpha) \to \Sigma$  is representable by cone complexes, strict, and surjective.  From this and Lemma \ref{lemma_diagonalrepresentable} it follows that the diagonal is representable and therefore that $\calC$ is a cone stack.

This association is a full and faithful functor $F$ from combinatorial cone stacks to cone stacks. The fact that $F$ is full is a consequence of Yoneda's lemma together with the fact that morphisms of combinatorial cone stacks are defined locally on cones.  Faithfulness of $F$ follows from the local definition of morphisms of combinatorial cone stacks.
We will check essential surjectivity of $F$ at the end of the proof.

Let $\sigma : \Sigma \rightarrow \RPC^f$ be a combinatorial cone stack and let $\mathcal C$ be the associated category fibered in groupoids over $\RPCC$, as constructed above.  

\vskip1ex

\textsc{$\calC$ is a cone stack.}  We must argue that the maps $\sigma(\alpha) \rightarrow \calC$ are representable, strict, and surjective.  Surjectivity is immediate from the definition of morphisms of cone stacks.  To see that these maps are strict and representable by cone complexes, consider a cone $\tau$ and a map $\tau \rightarrow \calC$.  By definition, there is a unique $\beta \in \Sigma$ and a map $\tau \rightarrow \sigma(\beta)$.  The fiber product $\sigma(\alpha) \times_\calC \tau$ may be identified with $\sigma(\alpha) \times_\calC \sigma(\beta) \times_{\sigma(\beta)} \tau$, so it is sufficient to show that $\sigma(\alpha) \times_\calC \sigma(\beta)$ is a cone complex.
 
Let $\Sigma / \{ \alpha, \beta \}$ be the category of triples $(\gamma, u, v)$ where $\gamma \in \Sigma$ and $u : \gamma \rightarrow \alpha$ and $v : \gamma \rightarrow \beta$.  By composing with the forgetful map to $\Sigma$, we can view $\Sigma / \{ \alpha, \beta \}$ as a category fibered in groupoids over $\RPC^f$.  It is immediate that the associated category fibered in groupoids over $\RPCC$ is $\sigma(\alpha) \times_{\calC} \sigma(\beta)$.  Moreover, the underlying category of $\Sigma / \{ \alpha, \beta \} \rightarrow \RPC^f$ is a partially ordered set.  Indeed, suppose that we have two morphisms $a, b : (\gamma, u, v) \rightarrow (\gamma', u', v')$ in $\Sigma / \{ \alpha, \beta \}$.  Since $\RPC^f / \{ \sigma(\alpha), \sigma(\beta) \}$ is equivalent to a partially ordered set 
(namely, the set of common faces of $\sigma(\alpha)$ and $\sigma(\beta)$),
the two maps $a$ and $b$ must have the same projection in $\RPC^f$.  Then we get a commutative triangle
\begin{equation*} \vcenter{\xymatrix{
\gamma \ar[dr]^u \ar@<-1.5pt>[d]_a \ar@<1.5pt>[d]^b \\
\gamma' \ar[r]_{u'} & \alpha
}} \text{\qquad above\qquad} 
\vcenter{\xymatrix{
\sigma(\gamma) \ar[dr]^{\sigma(u)} \ar[d]_{\sigma(a) = \sigma(b)} \\
\sigma(\gamma') \ar[r]_{\sigma(u')} & \sigma(\alpha)
}} . \end{equation*}
As $\Sigma$ is fibered in groupoids over $\RPC^f$, there is a unique arrow $\gamma \rightarrow \gamma'$ with these properties, so $a = b$.  Thus the underlying category of $\Sigma / \{ \alpha, \beta \}$ is a partially ordered set, so its associated category fibered in groupoids over $\RPCC$ is a cone complex, by Remark~\ref{rem:cone-complex}.
Moreover, the map to $\sigma(\beta)$ is strict by construction.

\vskip1ex

\textsc{Essential surjectivity.} We have now constructed a functor from combinatorial cone stacks to cone stacks and shown it is fully faithful. All that remains is to show it is also essentially surjective.  

Suppose that $\mathcal C$ is a cone stack.  Let $\Sigma$ be the category of all \emph{strict} morphisms from cones to $\mathcal C$.  A morphism in $\Sigma$ is a commutative triangle
\begin{equation}\label{eq:sigsigSig} \xymatrix@R=8pt{
\sigma' \ar[dr] \ar[dd] \\
& \mathcal C \\
\sigma \ar[ur]
} \end{equation}
in which the vertical arrow is a face morphism.  It is immediate that this is a combinatorial cone stack.  

We verify that the associated cone stack of $\Sigma$ is $\mathcal C$.  In other words, we need to show that $\Hom_{\text{$C$-space}}(\sigma, \Sigma) = \Hom(\sigma, \mathcal C)$ for every cone $\sigma$.  But if $\sigma \rightarrow \mathcal C$ is any morphism then Lemma~\ref{lem:initial} gives a unique initial factorization through a strict map $\tau \rightarrow \mathcal C$, and therefore a uniquely determined morphism of combinatorial cone stacks from $\sigma$ to $\Sigma$, as required.

Finally, we argue that it is a combinatorial cone space when $\calC$ is a cone space.  Suppose that $\tau \rightarrow \mathcal C$ is a strict morphism and that we have two face inclusions $\sigma \rightrightarrows \tau$ \emph{with the same image} lifting the same map $\sigma \rightarrow \mathcal C$.  Let $\beta$ be the barycenter of $\sigma$ and note that the two inclusions of $\sigma$ in $\tau$ agree on $\beta$.  Each of the arrows $\sigma \to \tau$ gives a commutative diagram:
\begin{equation*} \xymatrix{
\beta \ar[r] \ar[d] & \tau \ar[d] \\
\sigma \ar[ur] \ar[r] & \mathcal C
} \end{equation*}
But $\beta$ is in the interior of $\sigma$ and $\sigma \rightarrow \mathcal C$ is strict.  By Lemma~\ref{lem:initial}, the commutative diagram above is uniquely determined by its outer square.  But because $\mathcal C$ is a cone space, the outer square is determined by the morphisms around its periphery (there is no additional $2$-isomorphism).  Thus the two maps $\sigma \rightrightarrows \tau$ must agree. 

\end{proof}

\begin{remark}
The proof of Proposition~\ref{prop:C-stacks} shows that the diagonal of a cone stack is always representable by cone complexes.
\end{remark}

\begin{example}
We offer a small example that illustrates some of the concepts at play in the proof of Proposition~\ref{prop:C-stacks}.  
Consider the combinatorial cone stack $\tau\colon \Sigma\to\RPC^f$ drawn below on the left.  

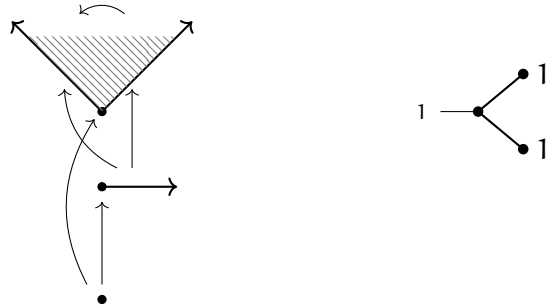
\begin{figure}[h!]

\begin{tikzpicture}
\begin{scope}[shift={(-3,0)}]

% vertex
\draw[fill=black] (0,-2.5) circle (.3ex);

%edge
\draw[fill=black] (0,-1) circle (.3ex);
\draw[thick, ->] (0,-1) -- (1,-1);

%flap
\draw[fill=black] (0,0) circle (.3ex);
\draw[thick,->] (0,0) -- (1.2,1.2);
\draw[thick,->] (0,0) -- (-1.2,1.2);
\fill[pattern color=gray,pattern=north west lines] (0,0)--(1,1)--(-1,1)--(0,0);

%arrows
% e to f
\draw[->] (0.4,-.75) -- (0.4,0.3);
\draw[->] (0.2,-.75) to [bend left] (-.5,0.3);

%v to e,f
\draw[->] (0,-2.3) -- (0,-1.2);
\draw[->] (-.2,-2.3) to [bend left] (-.1,-.1);

%f to f
\draw[->] (0.3,1.3) to [bend right=40] (-0.3, 1.3);
\end{scope}
\begin{scope}[shift={(3,0)}]
\draw[thick] (-1,0) -- (-0.4,0.5);
\draw[thick] (-1,0) -- (-0.4,-0.5);
\node[draw=none,label={0:$1$}] (mark1) at (-0.5,.5) {};
\node[draw=none,label={0:$1$}] (mark2) at (-0.5,-.5) {};
\fill (-1,0) circle (0.07);
\fill (-0.4,0.5) circle (0.07);
\fill (-0.4,-0.5) circle (0.07);
\draw (-1,0) -- (-1.5,0);
\node at (-1.75, 0) {$\scriptstyle 1$};
\end{scope}
\end{tikzpicture}
\end{figure}

\noindent (Technically, we have only drawn a finite category that is equivalent to the desired category fibered in groupoids. We suppress all identity morphisms to simplify the picture.)  Let $\calC$ be the associated stack.

This stack is like the $\ZZ/2\ZZ$ quotient of $\RR^2_{\ge0}$, except that it has no stabilizer at the cone point.  (In fact, we note in passing that $\calC$ represents the moduli functor $\mathcal{M}_G\colon \RPC^{\mathrm{op}}\to \mathbf{Groupoids}$ defined as follows.  Let $G$ be the weighted, 1-marked, genus 2 graph shown on the right in the above figure. Then the fiber of $\mathcal{M}_G$ over a cone $\tau$ is the groupoid of tropical curves whose underlying graph $\mathbb{G}(\Gamma)$ is isomorphic to a weighted edge contraction of $G$. See Section~\ref{sec:tropcurves}.)

Now as an example, let $\sigma = \R^2_{\ge0}$, and let us study strict morphisms $\sigma\to\calC$.  We note that there are exactly two strict maps $u_1,u_2\colon \sigma\to\calC$, corresponding to the two isomorphisms from $\sigma$ to the 2-dimensional cone in the picture. And for each $i,j=1,2$, there are two $2$-isomorphisms $u_i\cong u_j$; one of the two comes from the nontrivial automorphism of the 2-cone shown in $\Sigma$.  Keeping in mind Remark~\ref{rem:caution,caution}, then, we have the commutative diagrams shown, e.g.~when $i=j=1$:

\begin{figure}[h!]
\begin{tikzcd}[row sep=tiny,column sep = large]
\sigma  \arrow[dd,"1"] \arrow[dr, ""{name=U,below}, "u_1"{above}] &\\
 & \calC \\
 \sigma \arrow[ur, ""{name=D,above}, "u_1"{below}]&
\arrow[Rightarrow, from=U, to=D, "1"]
\end{tikzcd}
\qquad
\qquad
\begin{tikzcd}[row sep=tiny,column sep = large]
\sigma  \arrow[dd,"-1"] \arrow[dr, ""{name=U,below}, "u_1"{above}] &\\
 & \calC \\
 \sigma \arrow[ur, ""{name=D,above}, "u_1"{below}]&
\arrow[Rightarrow, from=U, to=D, "-1"]
\end{tikzcd}
\end{figure}

The presence of the second diagram verifies once again that the combinatorial cone stack associated to $\calC$ is manifestly not a combinatorial cone space, since there is a nonidentity automorphism on $\sigma$.

\end{example}

\begin{remark} \label{rem:make_your_own_cone_stack}
Let $FI$ denote the category of finite sets with injections as arrows, and suppose $\Sigma$ is any category and $X\colon \Sigma \to FI$ makes $\Sigma$ into a category fibered in groupoids over $FI$.  Then $X$ determines a combinatorial cone stack, all of whose cones are smooth.  This cone stack is constructed as follows.  First, there is a fully faithful functor
$$\operatorname{sm}\colon FI \to \RPC^f$$ associating to a set $E$ the rational polyhedral cone $\RR^E_{\ge 0}$, and associating to an injective map of sets $E'\to E$ the corresponding face map.  We write $\operatorname{sm}$ for this functor, since it associates to each set a cone that is {\em smooth} in the sense of toric geometry.  As $\operatorname{sm}$ is fibered in groupoids, and the composition of functors fibered in groupoids is fibered in groupoids, the composition $$\operatorname{sm}\circ \,X \colon \Sigma \to \RPC^f$$
produces a diagram of cones and face morphisms that is fibered in groupoids over $\RPC^f$.
\end{remark}

%%%%%%%%%%%%%%%%%%%%%%%%%%%%%%%%%%%%%%%%%%%%%%%%%%%%%%

\section{Tropical curves and their moduli}\label{sec:tropcurves}

This section introduces the main player of this article, the tropical moduli stack $\calM_{g,n}^{\trop}$ (see Section \ref{section_modulifunctor}). In Section \ref{section_geometrization}, we provide a proof of Theorem \ref{thm_geometric} from the introduction. Finally, in Section \ref{subsec:groupoidexamples}, we describe an explicit groupoid presentation of $\calM_{g,n}^{\trop}$.

%%%%%%%%%%%%%%%%%%%%%%%%%%%%%%%%%%%%%%%%%%%%%%%%%%%%%%

%%%%%%%%%%%%%%%%%%%%%%%%%%%%%%%%%%%%%%%%%%%%%%%%%%%%%%

\subsection{The moduli stack of tropical curves}\label{section_modulifunctor}
Expanding on the definition given in \cite{Serre_trees}, a \emph{(finite) graph} $G$ consists of the following data:
\begin{itemize}
\item a finite set $X(G) = V(G)\sqcup F(G)$, where $V(G)$ and $F(G)$ are called the \emph{vertices} $V(G)$ and  \emph{flags} of $G$, respectively;
\item a {\em root map} $r_G\colon X(G)\to X(G)$ which is idempotent and whose image is $V(G)$; and
\item an involution $\iota_G\mathrel{\mathop:} X(G)\rightarrow X(G)$ whose fixed point set contains $V(G)$.
\end{itemize}
We picture the root map as associating to a flag $f$ the vertex $v$ from which it emanates. The involution $\iota_G$ partitions $F(G)$ into a collection of subsets $\{f,\iota(f)\}$ of size $1$ or $2$. We refer to the sets consisting of two distinct flags as the \emph{finite edges} or simply as the \emph{edges} $E(G)$ of $G$, and to the sets consisting of one flag as the \emph{infinite edges} or \emph{legs} $L(G)$ of $G$. 
A {\em loop} is a finite edge $\{f,\iota_G(f)\}$ such that $r_G(f) = r_G(\iota_G(f))$.

There is a natural geometric realization functor taking a graph to a 1-complex (or a 1-complex with attached infinite legs). We say that a graph $G$ is {\em connected} if its geometric realization is.  Alternatively, $G$ is connected if, for any two vertices $v,w\in V(G)$, there is a sequence of vertices and flags $v_0\!=\!v,f_0,v_1,f_1,\ldots,v_n\!=\!w$ where $r(f_j)=v_j$ and $r(\iota(f_j))=v_{j+1}$ for each $j$.  

Let us recall some basic notions:

\begin{itemize}
\item The \emph{valence} $\val(v)$ of a vertex $v$ of $G$ is given as 
\begin{equation*}
\val(v)=\#\{f\in F(G)\:\vert\: r(f)=v\} \ , 
\end{equation*}
i.e., the number of flags emanating from $v$. 

\item A \emph{vertex weighting} on a graph $G$ is a function $h\colon V(G)\rightarrow \Z_{\geq 0}$. We refer to the pair $(G,h)$ as a \emph{(finite) weighted graph}. Given a weighted graph, we define its \emph{genus} as 
\begin{equation*}
g(G)=b_1(G)+\sum_{v\in V(G)} \hskip-1ex h(v) \ ,
\end{equation*}
where $b_1(G)$ denotes the first Betti number of $G$. A weighted graph $(G,h)$ is said to be \emph{stable} if for every vertex $v$ in $G$ we have 
\begin{equation*}
2h(v)-2+\val(v) > 0 \ .
\end{equation*}
%An \emph{isomorphism} $f\mathrel{\mathop:}(G,h)\xrightarrow{\sim} (G',h')$ of weighted graphs is an isomorphism of graphs $f\mathrel{\mathop:}G\xrightarrow{\sim} G'$ such that $h(v)=h'(f(v))$ for all $v\in V(G')$. 

\item Let $n=\# L(G)$. A \emph{marking} of a graph $G$ is a bijection 
\begin{equation*}
l\mathrel{\mathop:}\{1,\ldots, n\}\xlongrightarrow{\sim}L(G)
 \end{equation*}
 given by $i\mapsto l_i$.  We refer to a triple $(G,h,m)$, where $h$ is a vertex-weighting and $m$ is a marking, as an \emph{weighted ($n$-)marked graph}. %An \emph{isomorphism} $f\mathrel{\mathop:}(G,m)\xrightarrow{\sim} (G',m')$ of an isomorphism of graphs $f\mathrel{\mathop:}G\xrightarrow{\sim} G'$ such that $f(l_i)=l_i'$. 
 
\end{itemize}
Let $(G,h,m)$ and $(G',h',m')$ be two weighted graphs with marked legs.  A {\em morphism} $(G,h,m)\to (G',h',m')$ consists of a function $\pi\colon X(G)\to X(G')$ with the property that $\pi \circ r_G = r_{G'} \circ \pi$ and $\pi \circ i_G = i_{G'} \circ \pi$, which satisfies the following additional conditions:
\begin{itemize}
\item The map $\pi$ sends the legs of $G$ bijectively to the legs of $G'$, and preserves their markings; in other words $\pi(l_i) = l'_i$.
\item For each flag $f\in F(G')$, its preimage $|\pi^{-1}(f)|$ has exactly one element, which is automatically an element of $F(G)$.
\item For each vertex $v\in V(G')$, the weighted graph $\pi^{-1}(v)$ is connected with genus $h(v)$. 
\end{itemize}
The reader may verify that if there is a morphism $(G,h,m)\to (G',h',m')$ of weighted, marked graphs, then $G$ and $G'$ have the same genus (and the same number of markings).
The following notation follows \cite[Section 2.2]{ChanGalatiusPayne_tropicalmoduliII} as well as \cite[Definition 2.2]{Ulirsch_tropicalHassett}:

\begin{definition}\label{def:jgn}
We define $J_{g,n}$ to be the category whose objects are stable weighted, $n$-marked graphs $(G,h,m)$ of genus $g$, and whose morphisms are described above.  
\end{definition}

We will often pass, implicitly, to a skeleton category for $J_{g,n}$, picking one object arbitrarily for each isomorphism class. 

Note that a morphism $(G,h,m)\to(G',h',m')$ in  $J_{g,n}$ may be regarded as a composition of weighted edge contractions and isomorphisms of weighted, $n$-marked graphs, where the {\em weighted edge contraction} of an edge $e\in E(G)$ is the weighted, $n$-marked graph obtained by the following operation:%in $(G,h,m)$ is the following operation on marked, weighted graphs:  
\begin{itemize}
\item If $e$ is not a loop, incident to vertices $v\ne w$, then delete $e$ and identify its endpoints $v$ and $w$ into a new vertex $x$, setting the weight of $x$ to be $h(v) + h(w)$. 
\item If $e$ is a loop based at $v$, then delete $e$ and increment the weight of its base vertex by $1$.
\end{itemize}

Given a subset $S\subseteq E(G)$ of the set of edges of $G$, we define $(G,m,w)/S$ (or $G/S$, when unambiguous) to be the weighted, marked graph obtained by contracting the edges in $S$, in any order.  

Note that this notation should not be confused with the notation for contracting a subspace of a topological space:  the connected components of the image of $S$ in $(G,h,m)/S$ are in bijection with the connected components of $S$.

\begin{definition}\label{def_tropicalcurve} Let $P$ be a sharp monoid. A \emph{(generalized) metric (with values in $P$)} on a graph $G$ is a function $d\mathrel{\mathop:}E(G)\rightarrow P$. An \emph{(abstract) tropical curve $\Gamma$ with edge lengths in $P$} of genus $g$ with $n$ marked legs consists of
\begin{itemize}
\item a connected $n$-marked, weighted graph $(G,h,m)$ of genus $g$, and
\item a generalized metric $d$ with values in $P$ such that $d(e)\neq 0$ for all $e\in E(G)$. 
\end{itemize}
Given a rational polyhedral cone $\sigma$ with dual monoid $S_\sigma=\sigma^\vee\cap M$, a tropical curve with edge lengths in $S_\sigma$ is said to be a \emph{tropical curve over $\sigma$}.
\end{definition}

Given a tropical curve $\Gamma$ with edge lengths in $P$, we write $\GG(\Gamma)$ for its underlying marked weighted graph $(G,h,m)$. We say that $\Gamma$ is \emph{stable} if $\GG(\Gamma)$ is stable.

Let $\mathcal M_{g,n}^{\trop}(\RPC)$ be the category of pairs $(\sigma, \Gamma)$ where $\sigma$ is a rational polyhedral cone and $\Gamma$ is a genus $g$ tropical curve with $n$ marked legs over $\sigma$.  A morphism $(\sigma,\Gamma) \rightarrow (\sigma',\Gamma')$ in $\mathcal M_{g,n}^{\trop}$ is a morphism of rational polyhedral cones $f \colon \sigma \rightarrow \sigma'$ and a morphism of vertex-weighted, marked graphs $g \colon \GG(\Gamma') \rightarrow \GG(\Gamma)$ such that for each edge $e'$ of $\Gamma'$, we have:
\begin{enumerate}
\item $g$ contracts $e'$ if and only if $f^*d'(e') = 0 \in S_\sigma$, and
\item if $g(e') = e \in E(\Gamma)$ then $f^*d'(e') = d(e).$
\end{enumerate}
Here $f^*\colon S_{\sigma'} \to S_\sigma$ denotes the map of toric monoids associated to $f$.
 Note that
the category $\mathcal M_{g,n}^{\trop}(\RPC)$ is fibered in groupoids over $\RPC$ under the projection $(\sigma,\Gamma) \mapsto \sigma$.

\begin{definition}\label{def_modulistack}
Let $2g-2+n>0$. Using Proposition~\ref{prop:cfg-stack}, we define $\calM_{g,n}^{\trop}$ to be the unique stack over $\RPCC$ whose fiber over $\sigma$ is the groupoid of stable tropical curves over $\sigma$ that have genus $g$ and $n$ marked legs.
\end{definition}

By Remark \ref{remark_explicitextension}, the fiber of $\calM_{g,n}^{trop}$ over a cone complex $\Sigma$ is the groupoid of collections $\Gamma_\sigma$ for each face $\sigma$ of $\Sigma$ together with isomorphisms $(\Gamma_\sigma)\vert_\tau\simeq\Gamma_\tau$ % $\Gamma\vert_\tau\simeq\Gamma_\tau$ 
whenever $\tau$ is a face of $\sigma$.

%%%%%%%%%%%%%%%%%%%%%%%%%%%%%%%%%%%%%%%%%%%%%%%%%%%%%%

\subsection{Proof of Theorem \ref{thm_geometric}}\label{section_geometrization} The goal of this section is to prove Theorem \ref{thm_geometric}, i.e., that the stack  $\calM_{g,n}^{\trop}$ introduced in Definition~\ref{def_modulistack} is a cone stack. In our proof we verify the axioms of a cone stack directly. An alternative approach via combinatorial cone stacks is sketched in Section \ref{sec:ccsc}. A priori this means that we need to show two things:
\begin{itemize}
\item The diagonal map $\Delta\mathrel{\mathop:}\calM_{g,n}^{\trop}\rightarrow \calM_{g,n}^{\trop}\times \calM_{g,n}^{\trop}$ is representable by cone spaces. 
\item  There is a (automatically representable) morphism $U\rightarrow \calM_{g,n}^{\trop}$ from a cone complex $U$ that is surjective and strict. 
\end{itemize}
Lemma \ref{lemma_diagonalrepresentable}, however, implies that we only need to find a cone complex $U$ together with a representable morphism $U\rightarrow\calM_{g,n}^{\trop}$ that is surjective and strict; the condition that the diagonal is representable by cone spaces then follows automatically. 

We will construct $U$ as a disjoint union of stacks $U_G = U_{(G,h,m)}$, for $G$ ranging over stable weighted graphs.  Each $U_G$ is a ``rigidified'' moduli functor, for tropical curves $\Gamma$ {\em together with a particular morphism from $G$ to the dual graph of $\Gamma$}. This marking by $G$ ensures that $U_G$ is actually represented by a rational polyhedral cone. One may think of this in rough analogy with the choice of a tricanonical embedding (as in \cite{DeligneMumford_moduliofcurves}) or a Teichm\"uller level structure (as in \cite{AbramovichCortiVistoli}) to rigidify the moduli problem for smooth algebraic curves.

Fix a stable weighted graph $(G,h,m)$ of genus $g$ with $n$ marked points.
Consider the functor
\begin{equation*}
U_G=U_{(G,h,m)}\mathrel{\mathop:}\RPC^{op}\longrightarrow \mathbf{Groupoids}
\end{equation*}
that associates to a rational polyhedral cone $\sigma$ the groupoid whose objects are pairs $(\Gamma, \phi)$ consisting of
\begin{itemize} 
\item a tropical curve $\Gamma$ over $\sigma$ and 
\item a morphism $\phi\colon G\rightarrow\GG(\Gamma)$ of weighted, marked graphs.
\end{itemize}
An isomorphism between two such pairs $(\Gamma, \phi)$ and $(\Gamma',\phi')$ consists of an isomorphism $f\mathrel{\mathop:}\Gamma\xrightarrow{\sim}\Gamma'$ of tropical curves over $\sigma$ such that $\GG(f)\circ\phi=\phi'$. By Proposition~\ref{prop:cfg-stack}, the functor $U_G$ determines a stack over $(\RPCC,\tau_{face})$, also denoted by $U_G$, whose fiber over a rational polyhedral cone $\sigma$ is $U_G(\sigma)$. 

\begin{lemma}
The stack $U_G$ is represented by the rational polyhedral cone $\sigma_G=\mathbb{R}_{\geq 0}^{E(G)}$.
\end{lemma}

\begin{proof}
Let $\sigma$ be a rational polyhedral cone. An automorphism of an  object $(\Gamma, \phi\colon G \to \GG(\Gamma))$ in the groupoid $U_G(\sigma)$ consists of an automorphism $f$ of $\Gamma$ such that $\GG(f) \circ \phi = \phi$; this condition implies that $f$ may only be the identity. It follows  that $U_G(\sigma)$ is a setoid, and to prove the lemma it suffices to show that there is a natural bijection 
\begin{equation*}
\Hom(\sigma,\sigma_G)\xlongrightarrow{\sim}U_G(\sigma) \ .
\end{equation*}
Let $f\mathrel{\mathop:}\sigma\rightarrow \sigma_G$ be a morphism of cones and denote by $f^\#$ the induced morphism $\N^E\rightarrow S_\sigma$ on the level of dual monoids. For an edge $e$ of $G$ write $[e]$ for the unique generator of $S_{\sigma_G}=\N^E$ corresponding to $e$. Consider the generalized metric
\begin{equation*}\begin{split}
\widetilde{d}\mathrel{\mathop:}E(G)&\longrightarrow S_\sigma\\
e&\longmapsto f^\#\big([e]\big)
\end{split}\end{equation*}
on $G$ and define the set $X_f\subseteq E$ as the subset of edges $e$ for which $\widetilde{d}(e)=0$. The generalized metric $\widetilde{d}$ descends to a generalized metric $d$ on the weighted edge contraction $G/X_f$ thereby defining a tropical curve $\Gamma_f$ with $\GG(\Gamma)=G/X_f$. The association $f\mapsto\big(\Gamma_f, \phi_f\colon G\to G/X_f\big)$ defines the desired bijection.
\end{proof}

There is a natural morphism $U_G\rightarrow \calM_{g,n}^{\trop}$ that, over a rational polyhedral cone $\sigma$, is given by associating to $(\Gamma, \phi)$ the tropical curve $\Gamma$ over $\sigma$. In the following two Lemmas \ref{lemma_atlas=strict} and \ref{lemma_atlas=surjective} we show that these morphisms define a strict and surjective chart of $\calM_{g,n}^{trop}$.

\begin{lemma}\label{lemma_atlas=strict}
The morphism $U_G\rightarrow \calM_{g,n}^{\trop}$ is representable by cone spaces, strict, and quasicompact. 
\end{lemma}

\begin{proof}
Let $\sigma\rightarrow \calM_{g,n}^{\trop}$ be a morphism. We will show that the fiber product 
\begin{equation*}
X_{G,\Gamma}=U_G\times_{\calM_{g,n}^{\trop}}\sigma
\end{equation*}
is representable by a cone complex with finitely many cones and that the induced map $X_{G,\Gamma}\rightarrow \sigma$ is strict. Denote by $\Gamma$ the tropical curve over $\sigma$ that corresponds to $\sigma\rightarrow\calM_{g,n}^{\trop}$ and let $\alpha$ be a rational polyhedral cone. 
The setoid $X_{G,\Gamma}(\alpha)$ is equivalent to the set of quadruples
\begin{equation}\label{eq:Xonalpha}
\big\{(i\colon\alpha\to \sigma, \Gamma', G\to \GG(\Gamma'), \Gamma' \xrightarrow{\cong} i^*\Gamma) \big\}
\end{equation}
where $\Gamma'$ is a tropical curve over $\alpha$ and $G \rightarrow \GG(\Gamma')$ is an edge contraction.  Identifying $\Gamma'$ with $i^\ast \Gamma$, we may pare these data down to
\begin{equation*}
\big\{ (i : \alpha \rightarrow \sigma, G \rightarrow \mathbb G(i^\ast \Gamma)) \big\} .
\end{equation*}

We now produce a cone complex $\Sigma = \Sigma_{G,\Gamma}$ which represents $X_{G,\Gamma}$.  Let $\Sigma$ have cones indexed by pairs
\begin{equation*}
\big( \tau \preceq \sigma, \phi : G \rightarrow \mathbb G( \Gamma \vert_\tau ) \big)
\end{equation*}
where $G \rightarrow \mathbb G(\Gamma \vert_\tau)$ is an edge contraction and $\tau \preceq \sigma$ means $\tau$ is a face of $\sigma$.

%where $\tau$ is a face of $\sigma$ and $\rho$ is a face of $\sigma_G = \RR_{\ge 0}^{E(G)}$.  
We associate to $(\tau,\phi)$ a cone $C(\tau,  \phi)$ that is a copy of $\tau$.  Given two cones $C = C(\tau, \phi)\cong \tau$ and $C' = C(\tau',\phi') \cong \tau'$, suppose 
$\tau'\preceq \tau$  and the diagram 

\begin{equation*} \xymatrix@R=10pt{
& \mathbb G(\Gamma \vert_\tau) \ar[dd] \\
G \ar[ur]^-\phi \ar[dr]_-{\phi'} \\
& \mathbb G(\Gamma \vert_{\tau'})
} \end{equation*}
commutes.  Then we glue $C'$ to a face of $C$ via $C'\cong \tau'\preceq \tau\cong C$.

Then $\Sigma$ is a cone complex according to Definition~\ref{definition_conecomplexes}. Indeed, a face $\tau'\preceq \tau \cong C(\tau,\phi)$ is the image of exactly one face morphism $C(\tau',\phi')$, since $\phi'$ is uniquely determined from $\tau'\preceq \tau$ and $\phi$ as the composition $G \xrightarrow{\phi} \G(\Gamma\vert_\tau) \to \G(\Gamma\vert_{\tau'})$. Moreover, $\Sigma$ has finitely many cones because a cone has finitely many faces and a weighted marked graph has finitely many edge contractions.

Now we claim $\Sigma$ represents $X_{G,\Gamma}$.  Let $\alpha\in \RPC$.  Then $\Hom(\alpha, \Sigma)$ is the set of pairs consisting of a cone $C(\tau\preceq \sigma, \phi\colon G \rightarrow \GG(\Gamma\vert_\tau))$ together with a morphism $\alpha\to \tau$ in $\RPC$ that does not factor through any proper face $\tau'\prec\tau$.  This set is naturally in bijection with
\begin{equation}\label{eq:Sigmaalpha}
\big\{(i\colon \alpha\to\sigma, G \rightarrow \GG(i^*\Gamma))\big\} = X_{G,\Gamma}(\alpha) ,
\end{equation}
as required.

To conclude, we observe that $\Sigma\to \sigma$ is a strict morphism of cone complexes: each cone $C(\tau,\rho,\phi)$ of $\Sigma$ is sent isomorphically to $\tau\preceq \sigma$. 
\end{proof}

\begin{example}
Let $(g,n) =(1,2)$.  Let $\sigma = \RR_{\ge0}$, and let $\Gamma \in \mathcal{M}_{1,2}^{\trop}(\sigma)$ be a tropical curve of combinatorial type shown in the horizontal ray of $\mathcal{M}_{1,2}^{\trop}$ in Figure~\ref{figure_M12}, equipped with length $1$.  It consists of a length $1$ loop at a vertex supporting both markings $1$ and $2$. Next, let $G$ be the graph shown in the upper right of Figure~\ref{figure_M12}; it has two parallel edges $e_1$ and $e_2$ between vertices marked $1$ and $2$ respectively.  

Then $X_{G,\Gamma}$ is represented by a $1$-dimensional cone complex consisting of four rays attached at their origins.  The four rays are naturally labeled by the one of the two possible choices of an edge $e_i$ of $G$ (for $i=1,2$), together with one of the two possible isomorphisms $G/e_i \to \G(\Gamma)$. 
\end{example}

Now set $U=\coprod_{G}U_G$, where the disjoint union is taken over all stable weighted $n$-marked graphs $(G,h,m)$ of genus $g$. By Lemma \ref{lemma_atlas=strict} the induced map $U\rightarrow \calM_{g,n}^{\trop}$ is representable and strict and the following Lemma \ref{lemma_atlas=surjective} finishes the proof of Theorem \ref{thm_geometric}. 

\begin{lemma}\label{lemma_atlas=surjective}
The representable morphism $U\rightarrow \calM_{g,n}^{\trop}$ is surjective. 
\end{lemma}

\begin{proof}
Let $\sigma\rightarrow\calM_{g,n}^{\trop}$ be a morphism from rational polyhedral cone $\sigma$ corresponding to  a tropical curve $\Gamma$ in $\calM_{g,n}^{\trop}(\sigma)$. Write $X=\sigma\times_{\calM_{g,n}^{\trop}}U$. We need to show that the induced map $X\rightarrow\sigma$ is surjective. Using the notation from the proof of Lemma \ref{lemma_atlas=strict}, we have that $X = \coprod X_{G,\Gamma}$. Taking $G=\GG(\Gamma)$, we find that the cone $C\big(\sigma, G\xrightarrow{\sim} \GG(\Gamma)\big)$ maps isomorphically to $\sigma$, so $X\to\sigma$ is surjective.
\end{proof}

\begin{remark}\label{rem:maximalonly}
In Lemma~\ref{lemma_atlas=surjective} we may alternatively take $U = \coprod_G U_G$ where $G  =(G,h,m)$ ranges over all stable, weighted $n$-marked graphs of genus $g$ that are maximal with respect to contraction. In this case, $U\to \calM_{g,n}^{\trop}$ is still surjective; this follows directly from the fact that for any tropical curve $\Gamma \in \calM_{g,n}^{\trop}(\sigma)$, the graph $G(\Gamma)$ is a contraction of some maximal $G$.
\end{remark}

\subsection{Groupoid presentations of $\calM_{g,n}^{\trop}$}\label{subsec:groupoidexamples}
The proof of Theorem~\ref{thm_geometric}, in particular, produces an explicit description of each $\calM_{g,n}^{\trop}$ as a groupoid of cone complexes, which we explain now.  Let $U = \coprod_{G} U_G$ where $G$ ranges over all stable weighted, $n$-marked, graphs of genus $g$ that are maximal with respect to edge contraction (one for each isomorphism class).  By Remark~\ref{rem:maximalonly}, the natural map $U\to\calM_{g,n}^{\trop}$ is a strict cover, so we have $\calM_{g,n}^{\trop} \cong (R\rightrightarrows U)$, where
$R = U\times_{\calM_{g,n}^ {\trop}} U$.
By definition of $U$ as a disjoint union, $R$ also decomposes as a disjoint union
$$R = \coprod_{G_1,G_2} R_{G_1,G_2},$$
where 
$$R_{G_1,G_2} =  U_{G_1} \times_{\calM_{g,n}^ {\trop}} U_{G_2}$$
with $G_1$ and $G_2$ ranging over all ordered pairs of maximal weighted, marked graphs.

We may describe a cone complex $\Sigma_{G_1,G_2}$ representing $R_{G_1,G_2}$ as follows:  The cones of $\Sigma_{G_1,G_2}$ are indexed by triples $\big(S_1\subseteq E(G_1), S_2\subseteq E(G_2), \phi\colon G_1/S_1 \xrightarrow{\sim} G_2/S_2\big)$. To such a triple we associate a cone $C(S_1,S_2,\phi) \cong \sigma_{G_1/S_1}$; note that this cone has natural maps to $U_{G_1} = \sigma_{G_1}$ and $U_{G_2} = \sigma_{G_2}$, the latter via composition with $\phi$.  Moreover, if $E(G_1)\supseteq S_1'\supseteq S_1$ and $E(G_2)\supseteq S_2'\supseteq S_2,$ and $\phi'\colon G_1/S_1' \xrightarrow{\cong} G_2/S_2'$ is an isomorphism such that the diagram 
\begin{equation*} \xymatrix{
G_1/S_1 \ar[r]^-\phi \ar[d] &   G_2/S_2 \ar[d] \\
G_1/S_1' \ar[r]^-{\phi'} & G_2/S_2'
} \end{equation*}
commutes, then we glue $C(S_1',S_2',\phi')$ to $C(S_1,S_2,\phi)$ by the face morphism induced by $G_1/S_1\to G_2/S_2$.  

We show the groupoid presentations of $\calM_{1,1}^{\trop}$ and $\calM_{1,2}^{\trop}$ in Figures~\ref{figure_M12stack} and~\ref{figure_M11stack}, taking advantage of Remark~\ref{rem:maximalonly}.  These pictures may be compared with Figure~\ref{figure_M12}.  The pictures may temporarily be viewed monochromatically.  The red part of the figures will be explained in Section~\ref{section_families}; it illustrates the way in which $\calM_{1,2}^{\trop}$ functions as a universal curve over $\calM_{1,1}^{\trop}$, as asserted by Theorem~\ref{thm_universalcurve}.

\begin{figure}
\begin{tikzpicture}
%R11
\draw[ultra thick,->, color=red] (-4.5,0) -- (-3.5,0.5);
\draw[ultra thick,->, color=red] (-4.5,0) -- (-3.5,-0.5);
\fill[pattern color=red,pattern=north east lines] (-4.5,0) -- (-3.6,0.45) -- (-3.6,-0.45) -- (-4.5,0);

\draw[thick,->] (-4.5,0) -- (-5.5,0.5);
\draw[thick,->] (-4.5,0) -- (-5.5,-0.5);
\fill[pattern color=gray,pattern=north east lines] (-4.5,0) -- (-5.4,0.45) -- (-5.4,-0.45) -- (-4.5,0);

\draw[thick,->] (-4.5,0) -- (-5,1);
\draw[thick,->] (-4.5,0) -- (-4,1);
\draw[thick,->] (-4.5,0) -- (-4.7,1.1);
\draw[thick,->] (-4.5,0) -- (-4.3,1.1);
\node at (-4.5,1.5) {$R_{G_1,G_1}$};

%R12
\draw[ultra thick,->, color=red] (-1.5,0) -- (-0.5,0.5);
\draw[ultra thick,->, color=red] (-1.5,0) -- (-2.5,0.5);
\draw[thick,->] (-1.5,0) -- (-1,1);
\draw[thick,->] (-1.5,0) -- (-2,1);
\node at (-1.5,1.5) {$R_{G_1,G_2}$};
\node[color=red] at (-1,0) {$\lambda_1$};
\node[color=red] at (-2,0) {$\lambda_2$};

%R21
\draw[ultra thick,->, color=red] (1.5,0) -- (0.5,0.5);
\draw[ultra thick,->, color=red] (1.5,0) -- (2.5,0.5);
\draw[thick,->] (1.5,0) -- (1,1);
\draw[thick,->] (1.5,0) -- (2,1);
\node at (1.5,1.5) {$R_{G_2,G_1}$};
\node[color=red] at (1,0) {$\mu_1$};
\node[color=red] at (2,0) {$\mu_2$};

%R22
\draw[thick,->] (4.5,0) -- (3.5,0);
\draw[ultra thick,->, color=red] (4.5,0) -- (5.5,0);
\draw[ultra thick,->, color=red] (4.5,0) -- (4.5,1);

\fill[pattern color=gray,pattern=north east lines] (4.5,0) -- (3.6,0) -- (3.6,0.9) -- (4.5,0.9) -- (4.5,0);
\fill[pattern color=red,pattern=north west lines] (4.5,0) -- (5.4,0) -- (5.4,0.9) -- (4.5,0.9) -- (4.5,0);
\node at (4.5,1.5) {$R_{G_2,G_2}$};

%U1
\draw[ultra thick,->, color=red] (-4.5,-4) -- (-2.5,-5);
\draw[ultra thick,->, color=red] (-4.5,-4) -- (-2.5,-3);
\fill[pattern color=red, pattern=north east lines] (-4.5,-4) -- (-2.7,-3.1) -- (-2.7,-4.9) -- (-4.5,-4);
\node at (-3.5,-2.75) {$U_{G_1}$};
\node[color=red] at (-3.8,-3.4) {$\tau_1$};
\node[color=red] at (-3.8,-4.8) {$\tau_2$};

%U2
\draw[ultra thick,->, color=red] (1,-4) -- (3,-5);
\draw[ultra thick,->, color=red] (1,-4) -- (3,-3);
\fill[pattern color=red,pattern=north east lines] (1,-4) -- (2.8,-3.1) -- (2.8,-4.9) -- (1,-4);
\node at (2,-2.75) {$U_{G_2}$};
\node[color=red] at (2,-4.8) {$\rho$};

%arrows
\draw[thick,->] (-1,-1) -- (-1,-2);
\draw[thick,->] (1,-1) -- (1,-2);

%types
% G1
\draw[gray] (-1.9,-4) -- (-2.3,-4);
\draw[gray] (-1.5,-4) circle (0.4);
\draw[gray] (-1.1,-4) -- (-0.7,-4);
\node[gray] at (-2.5,-4) {$1$};
\node[gray] at (-0.5,-4) {$2$};
\fill[gray] (-1.1,-4) circle (0.05);
\fill[gray] (-1.9,-4) circle (0.05);

% G2

\draw[gray] (3.4,-4) circle (0.4);
\draw[gray] (3.8,-4) -- (4.4,-4);
\draw[gray] (4.4,-4) -- (4.9,-3.5);
\draw[gray] (4.4,-4) -- (4.9,-4.5);
\node[gray] at (5.1,-3.5) {$1$};
\node[gray] at (5.1,-4.5) {$2$};
\fill[gray] (3.8,-4) circle (0.05);
\fill[gray] (4.4,-4) circle (0.05);

\end{tikzpicture}
\caption{Groupoid presentation of $\calM_{1,2}^{\trop}$, as described in \S\ref{subsec:groupoidexamples}. The non-pure-dimensionality of the cone complex $R_{G_1,G_1}$ reflects the fact that the 1-edge contractions of $G_1$ have new automorphisms that are not simply restrictions of automorphisms of $G_1$.  We illustrate an instance of Theorem~\ref{thm_universalcurve} in red, depicting a groupoid presentation of the cone space $\Cone{\Gamma}$ where $\Gamma \in \calM_{1,1}^{\trop}(\RR_{\ge 0})$ consists of a single loop of length 1 on a once-marked vertex.  See Example~\ref{ex:universal}. The cone space $\Cone{\Gamma}$ is shown in Figure~\ref{figure_ConeGamma}.}\label{figure_M12stack}

\begin{tikzpicture}
% U_G
\draw[ultra thick,->, color=red] (0,0) -- (2,0);
\node at (-1,0) {$U_G$};

%combinatorial type
\draw[gray] (2.8,0) circle (0.4);
\draw[gray] (3.2,0) -- (3.8,0);
\node[gray] at (4,0) {$1$};
\fill[gray] (3.2,0) circle (-0.05);

%R_G
\draw[ultra thick,->, color=red] (0,3) -- (2,2);
\draw[ultra thick,->] (0,3) -- (2,4);
\node at (-1,3) {$R_{G,G}$};

%arrows
\draw[thick,->] (0.75,1.75) -- (0.75,0.5);
\draw[thick,->] (1.25,1.75) -- (1.25,0.5);

\end{tikzpicture}
\caption{Groupoid presentation of $\calM_{1,1}^{\trop}$ as described in \S\ref{subsec:groupoidexamples}. The cone complex $R_{G,G}$ evinces the fact that the graph $G$ has a nontrivial automorphism, whereas the contraction of $G$ by its unique edge does not.}\label{figure_M11stack}

\begin{tikzpicture}
\draw[thick, color=red] (-1,0) .. controls (0,0.25) and (0,0.25) .. (1,0);
\draw[thick, color=red] (-1,0) .. controls (0,-0.25) and (0,-0.25) .. (1,0);
\draw[->,thick, color=red] (0,-2) -- (1,0);
\draw[->,thick, color=red] (0,-2) -- (2,-3);
\draw[ultra thin, color=red] (0,-2) -- (-1,0);
\fill[pattern color=red, pattern=north west lines] (1,0)--(0,-2)--(2,-3)
--(3,-1)--(1,0);

\fill[color=red] (0,-2) circle (0.05);
\fill[pattern color=red,pattern=north west lines] (-1,0) .. controls (0,-0.25) and (0,-0.25) .. (0,-0.225) -- (0,-2) -- (-1,0);
\fill[pattern color=red,pattern=north west lines] (1,0) .. controls (0,-0.25) and (0,-0.25) .. (0,-0.225) -- (0,-2) -- (1,0);
\fill[pattern color=red, pattern=north east lines] (-1,0) .. controls (0,0.25) and (0,0.25) .. (0,0.225) -- (0,-0.25) -- (-1,0);
\fill[pattern color=red, pattern=north east lines] (1,0) .. controls (0,0.25) and (0,0.25) .. (0,0.225) -- (0,-0.25) -- (1,0);

\end{tikzpicture}
\caption{The cone space $\operatorname{Cone}(\Gamma)$ for $\Gamma\in\mathcal{M}_{1,1}^{\trop}(\RR_{\ge0})$ a once-marked loop of edge length $1.$ A groupoid presentation of $\operatorname{Cone}(\Gamma)$ is shown in Figure~\ref{figure_M12stack} in red, as described in Example~\ref{ex:universal}.}
\label{figure_ConeGamma}
\end{figure}

\medskip

\subsection{Combinatorial cone stack construction}
\label{sec:ccsc}

We end this section by noting the following alternative construction of the cone stack $\calM_{g,n}^{\trop}$ using the combinatorial characterization of cone stacks in Section~\ref{section_conestacks}.  Recall the category $J_{g,n}$ from Definition~\ref{def:jgn}, whose objects are all stable weighted, $n$-marked, graphs of genus $g$, and whose morphisms are all possible edge contractions followed by isomorphisms.  It will in fact be harmless to pass to a finite category by 
picking one object from each isomorphism class of $J_{g,n}$, making it easier to picture this category. However, at the moment it is convenient, when defining the following functor, not to do so.

Denote by $FI$ the category of finite sets and injections and consider the functor
$J_{g,n}^{op}\to FI$ sending an object in $J_{g,n}$, which has an underlying graph $G$, to its edge set $E(G)$.   As explained in Remark~\ref{rem:make_your_own_cone_stack},
composing with the functor $FI\to \RPC^f$ produces a diagram of cones which is a combinatorial cone stack that describes exactly the cone stack $\mathcal{M}_{g,n}^{\trop}$.

%%%%%%%%%%%%%%%%%%%%%%%%%%%%%%%%%%%%%%%%%%%%%%%%%%%%%%

\section{Families and tautological morphisms}\label{section_families}

%%%%%%%%%%%%%%%%%%%%%%%%%%%%%%%%%%%%%%%%%%%%%%%%%%%%%%

Having set up the foundations, we show in this section that the family of moduli stacks we introduced enjoys  desirable properties  which are  parallel to the classical theory.  We introduce forgetful morphisms and show they realize the universal family over the moduli space. There are clutching morphisms that make it possible to glue two legs to obtain a new compact edge in a tropical curve. In the identification of the universal family of $\calM_{g,n}^{\trop}$ with $\calM_{g,n+1}^{\trop}$, the natural sections of the forgetful morphism giving the markings are special cases of a clutching map. 
In Section~\ref{sec:tropicalization} we extend these parallels beyond mere analogies by way of a tropicalization map from the moduli space of logarithmic curves to the moduli space of tropical curves, to be introduced in Section~\ref{sec:log-curves}.  

We begin by giving a geometric realization of a tropical curve over a cone $\sigma$ as a family of tropical curves. The reader may follow the upcoming construction looking at Figure \ref{figure_ConeGamma} for an example.
\vspace{0.2cm}

Let $\Gamma$ be a tropical curve over a rational polyhedral cone $\sigma$.  
Consider the following three types of cones associated to the vertices and edges of $\Gamma$. 
\begin{enumerate}[(i)]
\item For every vertex $v$ of $\Gamma$,  take a copy of the cone $\sigma$
\begin{equation*}
\Cone(v) = \sigma\ .
\end{equation*}
\item For every infinite edge $l$, take a copy of the rational polyhedral cone 
\begin{equation*}
\Cone(l)= \sigma \times\R_{\geq 0} \ .
\end{equation*}
\item For every finite edge $e$ of $\Gamma$, consider a copy of the rational polyhedral cone
\begin{equation*}
\Cone(e)= \sigma \times_{\R_{\geq 0}}\R_{\geq 0}^2
\end{equation*}
where the morphism $\R_{\geq 0}^2\rightarrow\R_{\geq 0}$ is given by $(a,b)\mapsto a+b$ and $\sigma\rightarrow\R_{\geq 0}$ is dual to the homomorphism $\N\rightarrow S_\sigma$ given  by $1\mapsto d(e)\in S_\sigma$. 
\end{enumerate}

To each flag $f$ of $\Gamma$ we may associate a face $\sigma(f)$ of each of the three types of cones $\Cone(v)$, $\Cone(l)$, or $\Cone(e)$: In case (i) we take $\sigma(f) = \Cone(v)$.  In case (ii) we assign to $f$ the cone $\sigma \times 0$.  In case (iii) we arbitrarily choose a bijection between the two faces $\sigma \times_{\R_{\geq 0}}(\R_{\geq 0}\times {0})$ and $\sigma \times_{\R_{\geq 0}}(0\times \R_{\geq 0})$ and the two flags of $e$.  In all cases, the projection $\sigma(f) \to \sigma$ is an isomorphism, so its inverse is a section of $\Cone(v)$, $\Cone(l)$, or $\Cone(e)$ over $\sigma$.

% For each of the last three type of cones, we may associate a face of $\Cone(e)$ or $\Cone(l)$ to each flag of $\Gamma$. In case (ii) we assign to $f$ the cone $\sigma \times 0$; in case (iii) we arbitrarily choose a bijection between the two faces $\sigma \times_{\R_{\geq 0}}(\R_{\geq 0}\times {0})$ and   $\sigma \times_{\R_{\geq 0}}(0\times \R_{\geq 0})$ and the two flags of $e$; in case (iii$'$) we assign to both flags either of the following two compositions (which coincide):
% \begin{gather*}
% \sigma \simeq \sigma \times_{\R_{\geq 0}} \bigl( 0 \times \R_{\geq 0} \bigr) \rightarrow \sigma \times_{\R_{\geq 0}} \R_{\geq 0}^2 \rightarrow \sigma \times_{\R_{\geq 0}} \calW \\
% \sigma \simeq \sigma \times_{\R_{\geq 0}} \bigl( \R_{\geq 0} \times 0 \bigr) \rightarrow \sigma \times_{\R_{\geq 0}} \R_{\geq 0}^2 \rightarrow \sigma \times_{\R_{\geq 0}} \calW \ .
% \end{gather*}
% In all cases these faces are naturally identified with the cone $\sigma$, hence we denote $\sigma(f)$ the face assigned to the flag $f$. 

{\begin{definition}\label{definition_coneoverGamma}
The \emph{cone} $\Cone(\Gamma)=\Cone_\sigma(\Gamma)$ of $\Gamma$ over $\sigma$ is the cone space defined 
by identifying $\sigma(f)$ with $\Cone(r(f))$, for every flag $f$ of $\Gamma$. 
For every infinite edge $l$, the inclusion $\Cone(l) \to \Cone(\Gamma)$ is called the \emph{$l$-th marking} of the family and denoted by $s_l$.
\end{definition}

\begin{remark}
In the case of a loop $e$, the definition has the effect of identifying the two edges of $\Cone(e)$ into the waffle cone $\sigma \mathop\times_{\R_{\geq 0}} \calW$, pulled back from Example~\ref{example_loopcone}.

It is possible to reformulate the construction more explicitly to introduce $\calW$ without this sleight of hand.  We use the construction (iii), above, only for those edges $e$ that are not loops.  For the loops, we use (iii$'$):

 \begin{enumerate}
\item[(iii$'$)] For every loop $e$ of $\Gamma$ take a copy of the cone space 
\begin{equation*}
\Cone(e)=\sigma \times_{\R_{\geq 0}}\calW
\end{equation*}
where $\calW$ is the waffle cone from Example \ref{example_loopcone}. The morphism $\calW\rightarrow \R_{\geq 0}$ is induced by the morphisms $(a,b)\mapsto a+b$ on charts and the morphism $\sigma\rightarrow\R_{\geq 0}$ by $\N\ni1\mapsto d(e)\in S_\sigma$ as above.  We think of the waffle cone $\calW$ as a ``cone over the loop $e$.'' 
\end{enumerate}

When $e$ is a loop, the section $\sigma(f)$ of $\Cone(e)$ over $\sigma$ may be constructed from either of the following coincident maps:
\begin{gather*}
\sigma \simeq \sigma \times_{\R_{\geq 0}} \bigl( 0 \times \R_{\geq 0} \bigr) \rightarrow \sigma \times_{\R_{\geq 0}} \R_{\geq 0}^2 \rightarrow \sigma \times_{\R_{\geq 0}} \calW \\
\sigma \simeq \sigma \times_{\R_{\geq 0}} \bigl( \R_{\geq 0} \times 0 \bigr) \rightarrow \sigma \times_{\R_{\geq 0}} \R_{\geq 0}^2 \rightarrow \sigma \times_{\R_{\geq 0}} \calW \ .
\end{gather*}

\end{remark}

By construction there is a natural structure morphism $c\mathrel{\mathop:}\Cone(\Gamma)\rightarrow\sigma$ that restricts to the projection to $\sigma$ on each component. For each leg $l$, the composition $c \circ s_l: \sigma \times\R_{\geq 0} \to \sigma$ is the left projection.

\begin{definition} Call $q$ the quotient map from the disjoint union of all the cones that form the identification space $\Cone(\Gamma)$.
We define a function $H_{\sigma,\Gamma}$ from the set of sections of $\Cone(\Gamma)$ to $\Z$, by setting 
\begin{equation} H_{\sigma,\Gamma}(f)= \left\{\begin{array}{cl} h(v) & \text{if there exists $v\in \Gamma $ such that 
$f(\sigma) = q(\Cone(v))$}
%$\Cone({v}) \subseteq q^{-1}(\mathrm{Im}(f))$
\\ 0 & \text{else.}\end{array}\right. \end{equation} 
 We call $H$ the union of all the $H_{\sigma, \Gamma}$'s, as $\sigma$ varies among all possible cones, and $\Gamma$ over all tropical curves over $\sigma$. 
\end{definition}

Consider a face morphism $\phi: \tau\to \sigma$, and the corresponding edge contraction $\Gamma \to \Gamma_\tau$. There is a pullback morphism from the set of $\sigma$-sections of $\Cone(\Gamma)$ to the set of $\tau$-sections of $\Cone(\Gamma_\tau)$. Then we have
$$
 H_{\sigma, \Gamma}(f) \leq H_{\tau, \Gamma_\tau}(\phi^\ast(f)).
$$
In other words, $H$ is upper semi-continuous with respect to pullback via face morphisms. This is illustrated in Figure \ref{fig:usc}.}

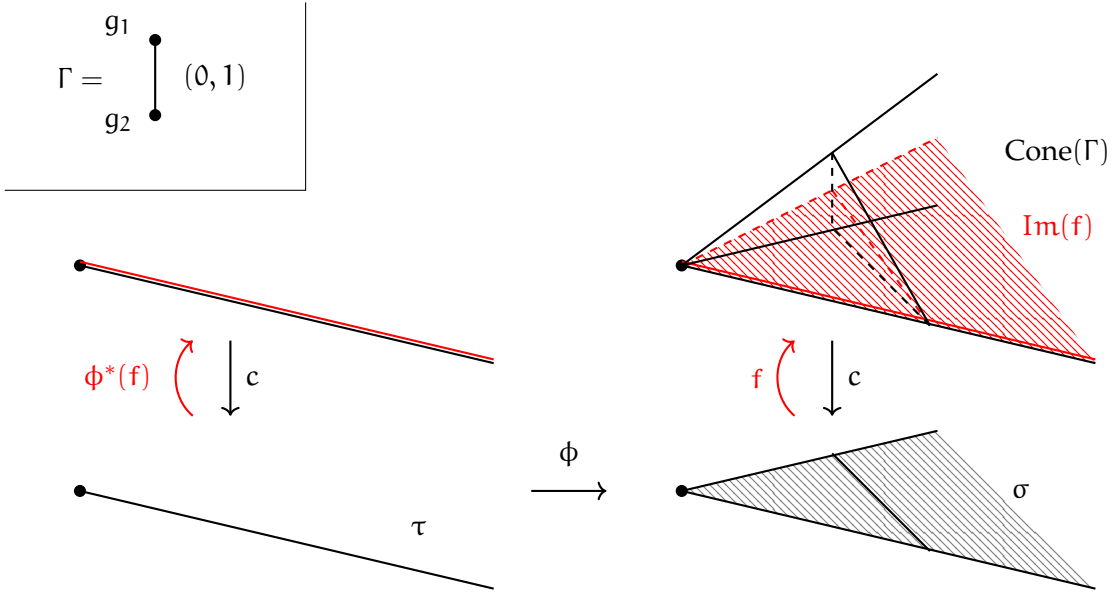
\begin{figure}[t]
%\centering
\begin{tikzpicture}
% the cone $\sigma$
\fill[pattern color=gray, pattern=north west lines] (0,0) -- (3.4,0.8) -- (5.5,-1.3) -- (0,0);
\fill (0,0) circle (0.08 cm);
\draw[thick] (0,0) -- (5.5,-1.3);
\draw[thick] (0,0) -- (3.4,.8);
\draw[thick] (2,.5) -- (3.3,-0.8);
\node at (4.5,0) {$\sigma$};

%arrows
\draw[thick,->] (-2,0) -- (-1,0);
\node at (-1.5,0.5) {$\phi$};
\draw[thick,->] (2,2) -- (2,1);
\draw[thick, color=red,->] (1.5,1) .. controls (1.2,1.25) and (1.2,1.75) .. (1.5,2);
\node at (2.3,1.5) {$c$};
\node[color=red] at (1,1.5) {$f$};
%the section $f$
\draw[dashed, thick, color = red] (0,3) -- (3.4,4.7);
\draw[dashed, thick, color = red] (2,4) -- (3.3,2.2);
\fill[pattern color=red, pattern= north west lines] (0,3) -- (3.4,4.7) -- (5.5,1.7)-- (0,3);

% the family $\Cone(\Gamma)$
\fill (0,3) circle (0.08 cm);
\draw[thick] (0,3) -- (5.5,1.7);
\draw[thick, color = red] (0,3.05) -- (5.5,1.75);
\draw[thick] (0,3) -- (3.4,3.8);
\draw[thick] (0,3) -- (3.4, 5.55);
\draw[dashed, thick]  (3.3,2.2)-- (2,3.5) -- (2,4.5);
\draw[thick]   (2,4.5) -- (3.3,2.2);
\node [color = red] at (5,3.5) {$Im(f)$};
\node at (5,4.5) {$\Cone(\Gamma)$};
% the cone $\tau$ and the family over it

\fill (-8,0) circle (0.08 cm);
\draw[thick] (-8,0) -- (-2.5,-1.3);
\draw[thick,->] (-6,2) -- (-6,1);
\draw[thick, color=red,->] (-6.5,1) .. controls (-6.8,1.25) and (-6.8,1.75) .. (-6.5,2);
\node at (2.3-8,1.5) {$c$};
\node[color=red] at (-7.5,1.5) {$\phi^\ast(f)$};
\node at (-3.5,-0.5) {$\tau$};
\fill (-8,3) circle (0.08 cm);
\draw[thick] (-8,3) -- (5.5-8,1.7);
\draw[thick, color = red] (-8,3.05) -- (5.5-8,1.75);

%\Gamma
\draw[thick] (-7,5) -- (-7,6);
\fill (-7,5) circle (0.08 cm);
\fill (-7,6) circle (0.08 cm);
\node at (-8,5.5) {$\Gamma =$};
\node at (-7.5,6.2) {$g_1$};
\node at (-7.5,4.9) {$g_2$};
\node at (-6.2,5.5) {$(0,1)$};
\draw (-9,4) -- (-5,4) -- (-5,6.5);

\end{tikzpicture}
\caption{The tropical curve $\Gamma$ over $\sigma$, on the northwest corner of the figure, consists of one edge of length $(0,1)$ and two vertices of genera $g_1 $ and $g_2$.  The right hand side of the picture depicts  $\Cone(\Gamma)$, with the structure morphism $c$ and the section $f(x,y) = (x,y, y/2)$. The face morphism $\phi: \tau \to \sigma$ is illustrated on the left hand side. The tropical curve $\Gamma_\tau$ consists of a vertex of genus $g_1+g_2$, and therefore $\Cone(\Gamma_\tau)$ is itself a ray. Observe that $H(f) = 0$, while $H(\phi^\ast(f)) = g_1+g_2$. }
\label{fig:usc}

\end{figure}

\begin{remark}
For the following proposition, we implicitly make use of the topological realization $|\mathcal C|$ of a cone space $\mathcal C$, defined as the colimit (in topological spaces)
\begin{equation*}
|\mathcal C| = \varinjlim_{\substack{\sigma \to \mathcal C \\ \text{strict}}} |\sigma|
\end{equation*}
where $|\sigma|$ is the usual topological realization of a rational polyhedral cone $|\sigma|$.  See Section~\ref{sec:real-to-top} for more about the topological realization.
\end{remark}

\begin{proposition}\label{prop_preimageisometry}
%Let $\sigma=\R_{\geq 0}$, which implies $\Gamma$ is a tropical curve with edge lengths in $\N$. 
Let $\Gamma$ be a tropical curve over $\sigma=\RR_{\geq 0}$, so that $\Gamma$ has edge lengths in $\N.$
The preimage $c^{-1}(1)  \subset \Cone(\Gamma)$  of $1\in\R_{\geq 0}$ is the tropical curve $\Gamma$, i.e. it is a  graph isometric to $\Gamma$ with respect to the lattice length on each component of $\Cone(\Gamma)$, and the vertex weighting of $\Gamma$ is recovered as $H|_{|c^{-1}(1)}$.
\end{proposition}

\begin{figure}[tb]
\begin{minipage}{0.49\textwidth}
\centering
\begin{tikzpicture}
\fill[pattern color=gray, pattern=north west lines] (0,0) -- (0,3.6) -- (3.6,3.6)--(3.6,0) -- (0,0);
\fill (0,0) circle (0.08 cm);
\fill (0,1.8) circle (0.08 cm);
\fill (1.8,0) circle (0.08 cm);
\fill (0.9,0.9) circle (0.08 cm);
\fill (1.8,1.8) circle (0.08 cm);
\fill (3.6,0) circle (0.08 cm);
\fill (0,3.6) circle (0.08 cm);
\fill (2.7,2.7) circle (0.08 cm);
\fill (3.6,3.6) circle (0.08 cm);
\fill (3.6,1.8) circle (0.08 cm);
\fill (1.8,3.6) circle (0.08 cm);
\fill (0.9,2.7) circle (0.08 cm);
\fill (2.7,0.9) circle (0.08 cm);
\draw[thick,->] (0,0) -- (4,0);
\draw[thick,->] (0,0) -- (0,4);
%\fill[pattern = north west lines] (0,0) -- (0,3.6) -- (3.6,0) -- (0,0);
\end{tikzpicture}
\end{minipage}
\begin{minipage}{0.49\textwidth}
\centering
\begin{tikzpicture}
\fill[pattern color=gray, pattern=north west lines] (0,0) -- (0,3.6) -- (3.6,3.6)--(3.6,0) -- (0,0);
\fill (0,0) circle (0.08 cm);
\fill (0,1.8) circle (0.08 cm);
\fill (1.8,0) circle (0.08 cm);
\fill (0.9,0.9) circle (0.08 cm);
\fill (1.8,1.8) circle (0.08 cm);
\fill (3.6,0) circle (0.08 cm);
\fill (0,3.6) circle (0.08 cm);
\fill (2.7,2.7) circle (0.08 cm);
\fill (3.6,3.6) circle (0.08 cm);
\fill (3.6,1.8) circle (0.08 cm);
\fill (1.8,3.6) circle (0.08 cm);
\fill (0.9,2.7) circle (0.08 cm);
\fill (2.7,0.9) circle (0.08 cm);
\draw[thick,->] (0,0) -- (4,0);
\draw[thick,->] (0,0) -- (0,4);
\draw[thick] (1.8,0) -- (0,1.8);
\node at (-0.9,1.8) {$c^{-1}(1)$};
%\fill[pattern = north west lines] (0,0) -- (0,3.6) -- (3.6,0) -- (0,0);
\end{tikzpicture}
\end{minipage}
\caption{The monoid $\N^2\oplus_{\Delta,\N,d(e)}\N$ (on the left) and its dual cone for $d(e)=2$.}
\label{figure_edgelength2}

\vspace{1em}
\begin{minipage}{0.49\textwidth}
\centering
\begin{tikzpicture}
\fill[pattern color=gray, pattern=north west lines] (0,0) -- (0,3.6) -- (3.6,3.6)--(3.6,0) -- (0,0);
\fill (0,0) circle (0.08 cm);
\fill (0,1.8) circle (0.08 cm);
\fill (1.8,0) circle (0.08 cm);
\fill (0.6,0.6) circle (0.08 cm);
\fill (1.2,1.2) circle (0.08 cm);
\fill (1.8,1.8) circle (0.08 cm);
\fill (0,3.6) circle (0.08 cm);
\fill (3.6,0) circle (0.08 cm);
\fill (2.4,2.4) circle (0.08 cm);
\fill (3,3) circle (0.08 cm);
\fill (3.6,3.6) circle (0.08 cm);
\fill (0.6,2.4) circle (0.08 cm);
\fill (1.2,3) circle (0.08 cm);
\fill (1.8,3.6) circle (0.08 cm);
\fill (2.4,0.6) circle (0.08 cm);
\fill (3,1.2) circle (0.08 cm);
\fill (3.6,1.8) circle (0.08 cm);
\draw[thick,->] (0,0) -- (4,0);
\draw[thick,->] (0,0) -- (0,4);
%\fill[pattern = north west lines] (0,0) -- (0,3.6) -- (3.6,0) -- (0,0);
\end{tikzpicture}
\end{minipage}
\begin{minipage}{0.49\textwidth}
\centering
\begin{tikzpicture}
\fill[pattern color=gray, pattern=north west lines] (0,0) -- (0,3.6) -- (3.6,3.6)--(3.6,0) -- (0,0);
\fill (0,0) circle (0.08 cm);
\fill (0,1.8) circle (0.08 cm);
\fill (1.8,0) circle (0.08 cm);
\fill (0.6,1.2) circle (0.08 cm);
\fill (1.2,0.6) circle (0.08 cm);
\fill (1.8,1.8) circle (0.08 cm);
\fill (0,3.6) circle (0.08 cm);
\fill (3.6,0) circle (0.08 cm);
\fill (2.4,3) circle (0.08 cm);
\fill (3,2.4) circle (0.08 cm);
\fill (3.6,3.6) circle (0.08 cm);
\fill (0.6,3) circle (0.08 cm);
\fill (1.2,2.4) circle (0.08 cm);
\fill (1.8,3.6) circle (0.08 cm);
\fill (3,0.6) circle (0.08 cm);
\fill (2.4,1.2) circle (0.08 cm);
\fill (3.6,1.8) circle (0.08 cm);
\draw[thick,->] (0,0) -- (4,0);
\draw[thick,->] (0,0) -- (0,4);
\draw[thick] (1.8,0) -- (0,1.8);
\node at (-0.9,1.8) {$c^{-1}(1)$};
%\fill[pattern = north west lines] (0,0) -- (0,3.6) -- (3.6,0) -- (0,0);
\end{tikzpicture}
\end{minipage}
\caption{The monoid $\N^2\oplus_{\Delta,\N,d(e)}\N$ (on the left) and its dual cone for $d(e)=3$.}
\label{figure_edgelength3}

\end{figure}

\begin{proof}
If $e$ is a finite edge of $\Gamma$ that is not a loop, then by definition of the fiber product we have 
\begin{equation*}
\Cone(e)=\big\{(a,b,s)\in\R^2_{\geq 0}\times\R_{\geq 0}\big\vert a+b=d(e)\cdot s\big\}
\end{equation*}
and therefore $c^{-1}(1)$ is precisely the subset determined by $a+b=d(e)$, an edge of lattice length $d(e)$ (see Figures \ref{figure_edgelength2} and \ref{figure_edgelength3}). 
If $e$ is a loop of $\Gamma$, the same argument shows that $c^{-1}(1)$ consists of
one loop of lattice length $d(e)$. Finally, for every infinite edge the preimage $c^{-1}(1)$ is an infinite edge. The construction of $\Cone(\Gamma)$ ensures that the underlying graph of $c^{-1}(1)$ is equal to the underlying graph of $\Gamma$ and the above argument yields that the edge lengths of $c^{-1}(1)$ are equal to the lengths of $\Gamma$. It follows from the definition of $H$ that the function $H|_{|c^{-1}(1)}$ is equal to the vertex weight of $\Gamma$. \end{proof}

%%%%%%%%%%%%%%%%%%%%%%%%%%%%%%%%%%%%%%%%%%%%%%%%%%%%%%

\begin{definition}\label{def:uc}
The \emph{universal curve} $\pi: \calX_{g,n}^{\trop} \to \calM_{g,n}^{\trop}$ is the fibered category whose fiber over the object $u_{[\Gamma]}: \sigma \to  \calM_{g,n}^{\trop}$ are the sections of $Cone(\Gamma) \to \sigma$. 
\end{definition}

A section of $\Cone(\Gamma)$ describes a "point" on a graph that is metrized by an abstract monoid $S_\sigma$.  Such a point is described either by a vertex or by a distance $d \in S_\sigma$ away from the vertex incident to a flag.  If the flag has a finite length $d(e)$, then the distance $d$ must be bounded by $d(e)$, as is expressed by the requirement that there be a $d' \in S_\sigma$ such that $d + d' = d(e)$. We formalize this observation in Proposition \ref{prop:univ-curve}.

%%%%%%%%%%%%%%%%%%%%%%%%%%%%%%%%%%%%%%%%%%%%%%%%%%%%%%

%The universal curve may be characterized as follows.

\begin{proposition} \label{prop:univ-curve}
%To give an element of $\mathcal X_{g,n}^{\trop}(\sigma)$ it is equivalent to give 
For a tropical curve $\Gamma \in \mathcal M_{g,n}^{\trop}(\sigma)$ a section of $\Cone(\Gamma)$ is given by the following datum:
\begin{enumerate}[(i)]
\item a vertex of $\Gamma$, or
\item a leg $f$ of $\Gamma$ and a nonzero $d \in S_\sigma$, or
\item an ordered pair of distinct flags $(f, f')$ with $f' = \iota(f)$ (in other words, an edge $e$ of $\Gamma$ with an orientation), and a pair $(d,d')$ of nonzero elements in $S_\sigma$ such that $d + d' = d(e).$
\end{enumerate}
\end{proposition}

Here we work with the  understanding that 
in (iii) 
the choice of $(f,f')$ and $(d,d')$ is equivalent to the choice of $(f',f)$ and $(d',d)$.

\begin{proof}[Proof of Proposition \ref{prop:univ-curve}]
%The proposition is equivalent to the assertion that the enumerated data characterize sections $s\colon \sigma \to \Cone(\Gamma)$ of $\Cone(\Gamma)$ over $\sigma$ when $\Gamma$ is a tropical curve over $\sigma$.  
This is a direct consequence of the construction of $\Cone(\Gamma)$.

\begin{enumerate}[(i)]
\item If $s(\sigma) = \Cone(v) = \sigma$ for a vertex $v$ of $\Gamma$, then choose $v$.
\item Suppose $s$ sends $\sigma$ into the interior of $\Cone(f) = \sigma\times \RR_{\ge 0}$ for a leg $f$ of $\Gamma$. Then by right projection $\sigma\times \RR_{\ge 0} \to \RR_{\ge 0}$, the section $s$ exactly determines a nonzero map $$d\colon \sigma \to \RR_{\ge0},$$ that is, a nonzero element $d \in  S_\sigma$.
\item Finally, suppose $s$ sends $\sigma$ into the interior of $\Cone(e) = \sigma\times_{\RR_{\ge 0}} \RR^2_{\ge 0}$ for a finite edge $e = (f,f')$ of $\Gamma$.
Such a section exactly determines a map $$(d,d')\colon \sigma\to \RR_{\ge 0}^2$$ with $d,d'\in S_\sigma\setminus\{0\}$ and $d+d' = d(e)$. 
\end{enumerate}
\end{proof}

\begin{remark} 
%The data in Proposition~\ref{prop:univ-curve} are meant to describe a point on a graph that is metrized by an abstract monoid $S_\sigma$.  Such a point is described either by a vertex or by a distance, $d \in S_\sigma$ away from the vertex incident to a flag.  If the flag has a finite length $d(e)$, then the distance $d$ must be bounded by $d(e)$, as is expressed by the requirement that there be a $d' \in S_\sigma$ such that $d + d' = d(e)$.

Although there is no strictly convex, rational, polyhedral cone $\sigma$ such that $S_\sigma = \mathbb R_{\geq 0}$, one may recognize that if $S_\sigma = \mathbb R_{\geq 0}$ is substituted into the characterization of Proposition~\ref{prop:univ-curve}, it recovers the set of real points of the metrized graph associated to $\Gamma$ (see Section \ref{sec:mon-to-real} for details).
\end{remark}

Definition \ref{def:uc} implies that the universal curve satisfies the universal property that, for any cone $\sigma$ and morphism $u_{[\Gamma]}$ corresponding to a tropical curve $\Gamma$,  we have the $2$-cartesian diagram:
\begin{equation*}\begin{CD}
\Cone(\Gamma)@>U_{[\Gamma]}>> \calX_{g,n}^{\trop}\\
@Vc VV @VV\pi_{g,n}^{\trop} V\\
\sigma @>u_{[\Gamma]}>> \calM_{g,n}^{\trop} .
\end{CD}\end{equation*}

As in the classical theory of  moduli spaces of curves, we  identify the universal curve over $\calM_{g,n}^{\trop}$ with a forgetful morphism from $\calM_{g,n+1}^{\trop}$ to $\calM_{g,n}^{\trop}$.

\begin{definition}\label{def:forgetfulmorph}
Assume that $2g - 2 + n > 0$. Let $\Gamma$ be a tropical curve over $\sigma$ (of genus $g$ with $n+1$ marked legs $l_1,\ldots, l_{n+1}$.  Denote by $\Gamma_\ast$ the tropical curve given by deleting the $(n+1)$-st marked leg $l_{n+1}$ from $\Gamma$.
\begin{enumerate}[(a)]
\item \label{it:forget1}
 If $\Gamma_\ast$ is already stable, then we simply set $\pi_{g,n+1}^{\trop}(\Gamma)=\Gamma_\ast$. 
\item  \label{it:forget2} Suppose $\Gamma_\ast$ is not stable and the vertex $v$ from which $l_{n+1}$ was emanating is connected to the rest of the graph via two distinct edges $e_1$ and $e_2$. In this case we  have $h(v)=0$. In order to define $\pi_{g,n+1}^{\trop}(\Gamma)$  we delete $v$ and the two flags adjacent to it, and create a new edge $e$ of length $d(e_1)+d(e_2)$ by connecting the remaining flags of $e_1$ and $e_2$.
\item  \label{it:forget3} Suppose $\Gamma_\ast$ is not stable and the  genus-zero vertex $v$ to which  $l_{n+1}$ was adjacent is connected to one edge $e$ and one leg $l_i$. Again, we necessarily have $h(v)=0$. In this case, to define $\pi_{g,n+1}^{\trop}(\Gamma)$ we remove  $v$  and the two flags adjacent to it, and label the remaining flag of $e$ with $l_i$. 
\end{enumerate}
Given $\Gamma'\in\calM_{g,n}(\sigma')$ and $\Gamma'\in\calM_{g,n}(\sigma)$, a morphism $\Gamma' \to \Gamma$ naturally induces a morphism $\pi_{g,n+1}^{\trop}(\Gamma') \to \pi_{g,n+1}^{\trop}(\Gamma) $. 
The  \emph{forgetful morphism} is the unique morphism
\begin{equation*}
\pi_{g,n+1}^{\trop}\mathrel{\mathop:}\calM_{g,n+1}^{\trop}\longrightarrow\calM_{g,n}^{\trop} 
\end{equation*}
that is given by $\Gamma\mapsto \pi_{g,n+1}^{\trop}(\Gamma)$ over a cone $\sigma$.
\end{definition}

%it is enough to describe fine it over a cone $\sigma$ the unique morphism of cone stacks whose restriction to a cone $\sigma$ is given by associating to a tropical curve $\Gamma$ over $\sigma$ a stable tropical curve $\pi_{g,n+1}^{\trop}(\Gamma)$ defined as follows: 
%Finally, g
%%%% Deleted $h_{n+1}$: not used.
We also define a function %morphism
\begin{equation*}
h_{n+1}\mathrel{\mathop:}\calM_{g,n+1}^{\trop}\longrightarrow \Z,
\end{equation*}
that, to a tropical curve $\Gamma$ over a rational polyhedral  cone $\sigma$, assigns the integer $h(r(l_{n+1}))$, i.e. the genus of the vertex that the $(n+1)$-th marked leg is attached to. The function $h_{n+1}$ is upper semi-continuous in the face topology, i.e. if $\tau$ is a face of $\sigma$, then $h_{n+1}(\Gamma_{\tau})\geq h_{n+1}(\Gamma)$.

We now provide a precise statement and a proof of Theorem \ref{thm_universalcurve}, namely that the forgetful morphism functions as a universal curve. 
\setcounter{maintheorem}{1}
\begin{maintheorem}%\label{thm_universalcurve}
The categories $\calX_{g,n}^{\trop}$ and $\calM_{g,n+1}^{\trop}$ are isomorphic over $\calM_{g,n}^{\trop}$; 
%The tropical forgetful morphism $\pi_{g,n+1}^{\trop}$ function as a universal curve, 
i.e. for any cone $\sigma$ and tropical curve $\Gamma$ over $\sigma$, there is a $2$-cartesian diagram
\begin{equation*}\xymatrix{
	\Cone(\Gamma) \ar[r]^{U_{[\Gamma]}} \ar[d]_c & \calM_{g,n+1}^{\trop} \ar[d]^{\pi_{g,n+1}^{\trop}} \\
	\sigma \ar[r]^{u_{[\Gamma]}} & \calM_{g,n}^{\trop}
} \end{equation*}
% \begin{equation*}\begin{CD}
% \Cone(\Gamma)@>U_{[\Gamma]}>> \calM_{g,n+1}^{\trop}\\
% @Vc VV @VV\pi_{g,n+1}^{\trop}V\\
% \sigma @>u_{[\Gamma]}>> \calM_{g,n}^{\trop} 
% \end{CD}\end{equation*}
depending contravariantly on $\sigma$, and we have $H = h_{n+1} \circ U_{[\Gamma]}$.
\end{maintheorem}

\begin{proof}
%[Proof of Theorem \ref{thm_universalcurve}]
Let $\sigma$ be a rational polyhedral cone and $\Gamma$ a stable tropical curve $\calM_{g,n}^{\trop}(\sigma)$ that, as an element of $\calM_{g,n}^{\trop}(\sigma)$, gives rise to a morphism $u_{[\Gamma]}\mathrel{\mathop:}\sigma\rightarrow\calM_{g,n}^{\trop}$.  We construct an isomorphism between the fibers of $\mathcal X_{g,n}^{\trop}$ and $\mathcal M_{g,n+1}^{\trop}$ over $u_{[\Gamma]}$.

Suppose that $\gamma \in \mathcal X_{g,n}^{\trop}(\sigma)$ is given by $u_{[\Gamma]}\colon \sigma\to \calM_{g,n}^{\trop}$ and a choice of datum as in Proposition~\ref{prop:univ-curve} (i), (ii), or (iii).
For each of the three cases (i)--(iii) in Proposition~\ref{prop:univ-curve}, we construct a stable tropical curve $\widetilde{\Gamma}$ such that $\pi_{g,n+1}^{\trop}(\widetilde{\Gamma})=\Gamma$ by following the procedure.
\begin{enumerate}[(i)]
\item if $\gamma$ is a vertex of $\Gamma$, add an $(n+1)$-st leg to the tropical curve emanating from $v$;
\item if $\gamma$ is a leg $f$ of $\Gamma$, add a  new vertex  on $f$ at length $d$ away from $r(f)$ and attaching an $(n+1)$-st leg to it; 
\item if $\gamma$ is a finite edge $e$ of $\Gamma$, subdivide the edge $e$ into two edges of lengths $d_1$ and $d_2$ and add an $(n+1)$-st leg to the new vertex.
\end{enumerate}
Note that in cases (ii) and (iii) it is possible that some length $d$ or $d'$ is $0$ when restricted to some proper face of $\tau$. Over such a face the $(n+1)$-st leg is attached at the appropriate vertex.

The inverse of this process is to remove the $(n+1)$-st leg of $\widetilde\Gamma \in \mathcal M_{g,n+1}(\sigma)$ and stabilize the result, marking the position where the $(n+1)$-st leg was attached.
The proof is concluded by observing that the equality of genus weighting functions $H = h_{n+1} \circ U_{[\Gamma]}$ is immediate from the construction above.
\end{proof}

\begin{example}\label{ex:universal}
Let $\sigma = \RR_{\ge 0}$ and let  $\Gamma\in \calM_{1,1}^{\trop}(\sigma)$ be a tropical curve consisting of a single loop based at a once-marked vertex $v$.  We illustrate the result of applying Theorem~\ref{thm_universalcurve} to $\Gamma$, obtaining $\Cone(\Gamma)$ as the fiber product $$ \sigma \times_{\calM_{1,1}^{\trop}} \calM_{1,2}^{\trop}.$$

Consider the groupoid presentations of the cone stacks ${\calM_{1,1}^{\trop}}$ and $\calM_{1,2}^{\trop}$ shown in Figures~\ref{figure_M12stack} and~\ref{figure_M11stack}.  View $\sigma$ as a trivial groupoid, i.e. as $\big(R\rightrightarrows \sigma\big)$ where $R \cong \sigma$ maps isomorphically to $\sigma$ by both source and target maps.  The map $u_{[\Gamma]} \colon \sigma \to U_G$ is an isomorphism of 1-dimensional cones.  The fiber product $ \sigma \times_{\calM_{1,1}^{\trop}} \calM_{1,2}^{\trop}$ is therefore the groupoid shown in red; the subcomplex of $\coprod R_{G_i,G_j}$ (in Figure~\ref{figure_M12stack}) depicted in red consists of the cones which are sent to the ray of $R_{G,G}$ (in Figure~\ref{figure_M11stack}) corresponding to the identity morphism, also shown in red.  

In more detail: write $s,t\colon R\rightrightarrows U$ for the source and target maps of the groupoid presentation of $\mathcal{M}_{1,2}^{\trop}$.  Consider the rays $\lambda_i, \mu_i, \tau_i,$ and $\rho$ as labeled as in Figure~\ref{figure_M12stack}.  For each $i=1,2$, the ray $\lambda_i$ represents the identification of $\tau_i$ isomorphically to $\rho$, and the ray $\mu_i$ represents the inverse identification.  More precisely, we have the following:
$$s(\lambda_i) = \tau_i, \qquad s(\mu_i) = \rho$$
$$t(\lambda_i) = \rho, \qquad t(\mu_i) = \tau_i$$
for $i=1,2$.
The resulting cone space may be described by taking the two red 2-dimensional cones in $U$ and identifying $\tau_1,\tau_2$, and $\rho$. This is shown in Figure~\ref{figure_ConeGamma}.
\end{example}

We now introduce natural tropical analogues of the clutching maps. Let $g= g_1+g_2$ and $[n] = I_1 \cup I_2$  a partition of the index set in two disjoint subsets with $2g_i-2+|I_i|\geq 0$ . We define the morphism of cone stacks
\begin{equation*}
\delta^{\trop}_{g_1,I_1}: \calM_{g_1,I_1 \cup \{\star\}}^{\trop}\times \calM_{g_2, I_2 \cup \{\bullet \}}^{\trop} \times \R_{\geq 0} \longrightarrow \calM_{g, [n]}^{\trop}
\end{equation*}
as follows: Given $\sigma \in \RPC$, an element in $(\calM_{g_1,I_1 \cup \{\star\}}^{\trop}\times \calM_{g_2, I_2 \cup \{\bullet \}}^{\trop} \times \R_{\geq 0})(\sigma)$ consists of a  triple $(\Gamma_1, \Gamma_2, d)$, with $\Gamma_1$ and $\Gamma_2$ tropical curves over $\sigma$  (with the appropriate discrete invariants) and $d\in S_\sigma$. The image 
$\delta^{\trop}_{g_1,I_1}(\Gamma_1, \Gamma_2, d)$ is defined to be the tropical curve whose underlying graph  $\Gamma$ over $\sigma$ is obtained by identifying the leg labeled by $\star$ in $\Gamma_1$ with the leg labeled by $\bullet$ in $\Gamma_2$. Formally, we define $V(\Gamma) = V(\Gamma_1) \cup V(\Gamma_2)$ as well as $F(\Gamma) = F(\Gamma_1) \cup F(\Gamma_2)$, and we redefine  $\iota(\star) = \bullet$ and $\iota(\bullet)=\star$.
Note that the resulting graph is  connected, of genus $g$, with legs labelled by the index set $[n]$, and that there is exactly one new compact edge for which the (generalized) metric is not assigned by the metrics on $\Gamma_1$ and $\Gamma_2$.  We thus extend the generalized metric to all compact edges of $\Gamma$ by assigning length $d$ to the new edge. 

In an analogous way we may also define self-gluing maps:
\begin{equation*}
\delta^{\trop}_{irr}:\calM_{g-1, [n] \cup \{\star,\bullet\}}^{\trop}\times \R_{\geq 0} \longrightarrow\calM_{g,n}^{\trop}.
\end{equation*}

It is worth pointing out that while in the classical theory clutching morphisms are maps between products of moduli spaces, these tropical analogues must be realized as correspondences:
 
\begin{equation*}
\xymatrix@!C=8em{
& \mathcal M^{\trop}_{g_1, I_1 \cup \{\star\}} \times \mathcal M^{\trop}_{g_2, I_2 \cup \{\bullet\}} \times \RR_{\geq 0} \ar[]!<50pt,0pt>;[dr]^{\delta^{\trop}_{g_1,I_1}} \ar[]!<-50pt,0pt>;[dl] \\
\mathcal M^{\trop}_{g_1, I_1 \cup \{\star\}} \times \mathcal M^{\trop}_{g_2, I_2 \cup \{\bullet\}} & & \mathcal M^{\trop}_{g,n}
} \hphantom{\hskip4em}
\end{equation*}
\begin{equation} 
\xymatrix@!C=6em{
& \mathcal M^{\trop}_{g, [n] \cup \{\star,\bullet \}}  \times \RR_{\geq 0} \ar[]!<25pt,0pt>;[dr]!<-10pt,0pt>^{\delta^{\trop}_{irr}} \ar[]!<-25pt,0pt>;[dl] \\
\mathcal M^{\trop}_{g, [n] \cup \{\star,\bullet \}} & & \mathcal M^{\trop}_{g+1,n}
}\label{tropclutch}
\end{equation}

We conclude this section with the observation that in the identification of $\pi_{g,n+1}^{\trop}$ with the universal family, the tautological sections are identified with appropriate clutching morphisms.

\begin{proposition}
Consider the forgetful morphism $\pi_{g,n+1}^{\trop}: \calM_{g,n+1}^{\trop} \to \calM_{g,n}^{\trop}$, and identify it with the universal family over $\calM_{g,n}^{\trop}$ as in  Theorem \ref{thm_universalcurve}. Then the tautological marking $l_{n+1} : \calM_{g,n}^{\trop} \times \R_{\geq 0} \to \calM_{g,n+1}^{\trop}$, attaching the $(n+1)$-st leg to the vertex supporting the $l$-th leg, is identified with the clutching morphism $\delta_{0, \{l,n+1\}}$ through the following commutative diagram.
$$
\xymatrix{
\calM_{g,n+1}^{\trop}  \ar[r]^{\sim} &   \calX_{g,n}^{\trop}\\
 \calM_{0, \{l ,n+1, \bullet \} }^{\trop} \times \calM_{g,[n] \smallsetminus \{l\} \cup \{\star\} }^{\trop}    \times \R_{\geq 0}  \ar[r]_-{\phi}^-{\sim} \ar[u]^{\delta_{0, \{l,n+1\}}} &  \calM_{g,n}^{\trop} \times \R_{\geq 0} \ar[u]_{s_l}
}
$$
\end{proposition}

\begin{proof}
For any cone $\sigma$, we observe that $ \calM_{0, \{l ,n+1, \bullet \} }^{\trop} (\sigma)$ is equal to a single point (a groupoid with one object and no nontrivial automorphisms); the natural isomorphism  $\phi$ in the bottom row of the diagram  forgets the first entry and relabels $l$ the leg marked $\star$ in the second entry. For a pair $(\Gamma, d)$, with $\Gamma$ a tropical curve over $\sigma$ and $d\in S_\sigma$, the image $s_l(\Gamma, d)$ corresponds to the choice of the point $P$ on the $l$-th leg of $\Gamma$ at distance $d$ from the root of $l$. The image $\delta_{0, \{l,n+1\}}(\phi^{-1}(\Gamma, d))$ is the tropical curve $\widetilde\Gamma$ obtained from $\Gamma$ by replacing the leg $l$ with a tripod with one edge of length $d$ attached to the root of $l$ and two unbounded edges labeled $l$ and $n+1$ emanating from the vertex of the tripod. The proof is concluded by noting that  $\widetilde\Gamma$ is the image of  $(\Gamma,P)$ in the bijection used in the proof of Theorem \ref{thm_universalcurve}.
\end{proof}

%%%%%%%%%%%%%%%%%%%%%%%%%%%%%%%%%%%%%%%%%%%%%%%%%%%%%%

%%%%%%%%%%%%%%%%%%%%%%%%%%%%%%%%%%%%%%%%%%%%%%%%%%%%%%

\section{Generalized edge lengths and topological realization}\label{sec:real}

\bigskip

In this section, we explain several canonical extensions of $\mathcal{M}_{g,n}^{\trop}$.   For example, we explain how our construction generalizes to allow real edge lengths on tropical curves, indeed edge lengths from any integral, saturated, but not necessarily finitely generated, monoid. This is a familiar setting for tropical geometers.  We also explain how to interpret $\mathcal{M}_{g,n}^{\trop}$ as a topological stack.  

All of the constructions we consider here apply more generally to cone stacks over $\RPC$. In fact, each extension that we consider here is obtained by pullbacks along the following sequence of morphisms of topoi (using Proposition \ref{prop_pullback}):

\begin{equation*}
\Top \to \RRPC \to \ShpMon^{\mathrm{op}} \to \RPC
\end{equation*}

We note the following:
\begin{itemize}
\item The geometric stacks over $\Top$ are \emph{topological stacks}.% in the sense of Example \ref{example_top2}.
\item The geometric stacks over $\RRPC$, which we dub \emph{real cone stacks}, coincide with geometric stacks over a category of polyhedral cone complexes defined in analogy with rational polyhedral cone complexes by gluing polyhedral cones over their faces.  (One could alternatively define them directly in analogy with what was done in Section~\ref{section_conestacks}, but we prefer to use the pullback and restriction to do the work for us.)
\item The geometric stacks over $\ShpMon^{\mathrm{op}}$ are the same as \emph{Kato stacks} (as in \cite[Section 2]{Ulirsch_nonArchArtin}) over the category of Kato fans (in the sense of \cite{Kato_toricsing}).
\item The geometric stacks over $\RPC$ are, of course, the familiar {\em cone stacks} of Definition~\ref{def_conestacks}, studied in this paper.
\end{itemize}

We explain each pullback in turn, first formulated for an arbitrary moduli problem in the appropriate geometric context, then for $\mathcal{M}_{g,n}^{\trop}$ specifically.

{Furthermore, we also outline how to use the same idea to {relate} our construction {to one over all commutative monoids} in order to recover the \emph{extended cone complex} of the moduli space of tropical curves~\cite{AbramovichCaporasoPayne_tropicalmoduli} and the \emph{extended tropical curves} \cite[Section 3.3]{Caporaso_tropicalmoduli} (see also \cite{HuszarMarcusUlirsch_troplogclutch&glue} for more details.)}

\subsection{Monoids}
\label{sec:monoids}

Recall that a monoid $M$ is said to be \emph{integral} if it can be embedded inside a commutative group, the initial of which is denoted $M^{\rm gp}$.  We say that $M$ is \emph{saturated} if, in addition to being integral, an element $x \in M^{\rm gp}$ such that $nx \in M$ is in fact in $M$.  Finally, $M$ is said to be \emph{sharp} if the only invertible element of $M$ is $0$.

For the rest of this paper, we will use the word \emph{monoid} always to refer to integral, saturated monoids; we write $\ShpMon$ for the category of sharp (integral, saturated) monoids.  There is a functor
\begin{equation*}
\RPC\to \ShpMon^{\mathrm{op}}
\end{equation*}
that sends a rational polyhedral cone $\sigma$ to the monoid $S_\sigma$.  

This functor identifies $\RPC$ with the %category spanned by 
full subcategory of finitely generated monoids inside of $\ShpMon^{\mathrm{op}}$, and accordingly we think of 
the category $\ShpMon^{\mathrm{op}}$ 
as a category of generalized cones. 

We note that the definition of a family of tropical curves over a rational polyhedral cone  (Definition~\ref{def_tropicalcurve}) involves edge lengths drawn from a monoid, namely the dual monoid $S_\sigma$ associated to a cone $\sigma$.  This yields a natural way to extend  $\calM_{g,n}^{\trop}(\RPC)$ to a category fibered in groupoids $\calM_{g,n}^{\trop}(\ShpMon^{\mathrm{op}})$ over $\ShpMon^{\mathrm{op}}$, allowing edge lengths to be drawn from any sharp monoid.  Since there seems to be no risk of confusion from doing so, we will simply write $\calM_{g,n}^{\trop}$ for both $\calM_{g,n}^{\trop}(\RPC)$ and $\calM_{g,n}^{\trop}(\ShpMon^{\mathrm{op}})$.

In the next subsection, we  put this extension in a more general context, showing how to extend any moduli problem on $\RPC$ to one over $\ShpMon^{\mathrm{op}}$, canonically.

\subsection{From rational polyhedral cones to monoids}
\label{sec:rpc-to-mon}

Our goal in this section is to extend moduli problems from $\RPC$, which we identify with the opposite category of finitely generated, integral, saturated, sharp monoids, to all of $\ShpMon^\mathrm{op}$.
By a moduli problem, we  mean a contravariant functor to $\mathbf{Set}$ or a category fibered in groupoids over the appropriate base category.
In Lemma~\ref{lem:ext-lfp}, we will identify the moduli problems over 
$\ShpMon^{\mathrm{op}}$ obtained in this way as precisely those that are of \emph{locally of finite presentation} (in the sense of Definition~\ref{def:lfp}).

\begin{remark}
In fact, the construction we present here is the standard way
to obtain a morphism of topoi $X \to Y$ from a left exact, continuous functor $Y \to X$, with $X = \ShpMon^{\mathrm{op}}$ and $Y = \RPC$.  
The same construction will be used in \S\ref{sec:real-to-top}.
\end{remark}

If $\calC$ is a presheaf (i.e.~a contravariant functor to $\mathbf{Set}$) on $\RPC$,
it may be extended to a presheaf $\widetilde{\calC}$ on $\ShpMon^{\mathrm{op}}$ by the following formula:
\begin{equation} \label{eqn:rpc-to-mon}
\widetilde \calC(M) = \varinjlim_{S_\sigma \to M} \calC(\sigma) = \varinjlim_{\substack{N \subset M \\ \text{$N$ f.g.}}} \calC(N)
\end{equation}
The first colimit is taken over the category of pairs $(\sigma, \varphi)$ where $\sigma \in \RPC$ and $\varphi \colon S_\sigma \to M$ is a homomorphism of monoids, with an arrow $(\sigma,\phi)\to(\sigma',\phi')$ for every $f\colon\sigma\to\sigma'$ with $\phi\circ f^* = \phi'$.  Since $S_\sigma \to M$ always factors through its image, which is also a finitely generated monoid, this colimit coincides with the second colimit displayed above, this time taken over all finitely generated submonoids of $M$.

The indexing category of finitely generated submonoids of $M$ is filtered, for if $N$ and $P$ are finitely generated submonoids of $M$, then so is the saturated monoid that they generate.  This implies that the assignment $\calC \mapsto \widetilde{\calC}$ preserves fiber products.

Next, the formula~\eqref{eqn:rpc-to-mon} makes sense even if $\calC$ is a category fibered in groupoids, provided the colimit is interpreted $2$-categorially.  Since the colimit~\eqref{eqn:rpc-to-mon} is filtered, this interpretation can be made explicit:
\begin{align*}
\mathbf{Ob}(\widetilde \calC(M)) & = \coprod_{\substack{N \subset M \\ \text{$N$ f.g.}}} \mathbf{Ob}(\calC(N)) \\
\Hom_{\widetilde \calC(M)} \bigl( (N,\alpha), (P,\beta) \bigr) & = \varinjlim_{\substack{N+P \subset Q \subset M \\ \text{$Q$ f.g., sat.}}} \Hom_{\calC(Q)}(\alpha|_Q,\beta|_Q) 
\end{align*}
In the above equations, we have written $(N,\alpha)$ for an object of $\coprod_{N \subset M} \calC(N)$, meaning that $\alpha \in \calC(N)$.  The notation $\alpha|_Q$ refers to the image of $\alpha$ under the restriction morphism $\calC(N) \rightarrow \calC(Q)$ associated to the inclusion $N \subset Q$.  The colimit in the second line is taken over finitely generated, saturated monoids contained in $M$ and containing both $N$ and $P$.
We remark again that because the extension is defined using a filtered colimit, it preserves fiber products.

We can now characterize those categories fibered in groupoids $\calC$ over $\ShpMon^{\mathrm{op}}$ that are induced from $\RPC$.

\begin{definition}\label{def:lfp}
We say that a category fibered in groupoids $\calC$ over $\ShpMon^{\mathrm{op}}$ is \emph{locally of finite presentation} if the natural map $\varinjlim \calC(M_i) \rightarrow \calC(\varinjlim M_i)$ is an isomorphism whenever $\{ M_i \}$ is a filtered diagram of monoids.
\end{definition}

\begin{lemma} 
\label{lem:ext-lfp}
A category fibered in groupoids $\widetilde \calC$ over $\ShpMon^{\mathrm{op}}$ is the extension of a moduli problem over $\RPC$ if and only if $\widetilde \calC$ is locally of finite presentation, in which case the moduli problem it extends is unique up to unique isomorphism.
\end{lemma}
\begin{proof}
First we show that if $\calC$ is a category fibered in groupoids over $\RPC$ then $\widetilde\calC$ is locally of finite presentation.  Consider a filtered diagram of monoids, $\{ M_i \}$, indexed by $i \in I$, with colimit $M$.  Let $J$ be the category of pairs $(i, N)$ where $i \in I$, and $N$ is a finitely generated submonoid of $M_i$.  Note that if $P \subset M$ is a finitely generated monoid then $P$ is also finitely presented by R\'edei's theorem~\cite[Theorem~5.12]{RG} so there is some $i \in I$ and a finitely generated submonoid $N \subset M_i$ that maps isomorphically to $P$ under the map $M_i \rightarrow M$.  Thus,
\begin{equation*}
\widetilde\calC(M) = \varinjlim_{\substack{P \subset M \\ \text{$P$ f.g., sat.}}}\hskip-1.6ex\calC(P) = \varinjlim_{(i, N) \in J}\hskip-1ex\calC(N) = \varinjlim_i \widetilde\calC(M_i),
\end{equation*}
as required for $\widetilde\calC$ to be locally of finite presentation.

Conversely, suppose that $\calC$ is locally of finite presentation.  Let $\calD$ be the category fibered in groupoids over $\RPC$ obtained by restricting $\calC$ to the finitely generated monoids in $\ShpMon^{\mathrm{op}}$.  Then $\calC$ and $\widetilde\calD$ are both locally of finite presentation, so they are both uniquely determined by their restrictions to $\RPC$, since every monoid is the colimit of the diagram of its finitely generated submonoids.  But $\calC$ and $\widetilde\calD$ both restrict to the same moduli problem on $\RPC$, so $\calC$ and $\widetilde\calD$ must be isomorphic.  Therefore $\calC$ is extended from $\RPC$.
\end{proof}
Lemma \ref{lem:ext-lfp}  allows us to characterize the extension of $\calM_{g,n}^{\trop}$ to all monoids directly.  We abuse notation and write $\calM_{g,n}^{\trop}$ also for the category fibered in groupoids over $\ShpMon^{\mathrm{op}}$ whose fiber over a sharp monoid $P$ is the groupoid of abstract tropical curves metrized by $P$.  We temporarily write $\widetilde\calM_{g,n}^{\trop}$ for the extension of $\calM_{g,n}^{\trop}$ from $\RPC$ to $\ShpMon^{\mathrm{op}}$, as introduced above.

\begin{proposition}
\label{fireworks} % Huh?? -MC
$\calM_{g,n}^{\trop}$ is locally of finite presentation as a category fibered in groupoids over $\ShpMon^{\mathrm{op}}$.
\end{proposition}
\begin{proof}
Here we must show that if $\Gamma \in \calM_{g,n}^{\trop}(P)$ then $\Gamma$ is induced from $\Gamma' \in \calM_{g,n}^{\trop}(N)$ for some finitely generated submonoid $N$ of $P$, and that $\Gamma'$ is unique up to unique isomorphism and enlargement of $N$.  Both of these claims are immediate from the fact that the metric on $\Gamma$ only requires the specification of finitely many elements of $P$ as the edge lengths.
\end{proof}

\begin{corollary}
The categories fibered in groupoids $\calM_{g,n}^{\trop}$ and $\widetilde\calM_{g,n}^{\trop}$ are isomorphic.
\end{corollary}
\begin{proof}
This is immediate from Lemma~\ref{lem:ext-lfp} and Proposition~\ref{fireworks}.
\end{proof}

\begin{remark}
If $M$ is a monoid, its prime ideals form a topological space called the spectrum.  A Kato fan is a space equipped with a sheaf of monoids such that the pair looks locally like the spectrum of a monoid.  The monoidal analogue of Proposition~\ref{prop:C-stacks} is that the $2$-category of categories fibered in groupoids over monoids is equivalent to the $2$-category of stacks over Kato fans.  Since the category of finitely generated monoids is equivalent to the category of rational polyhedral cones, the category of Kato fans that are locally of finite type is equivalent to the category of rational polyhedral cone complexes \cite[Proposition 3.7]{Ulirsch_functroplogsch}.  This equivalence extends to an equivalence of geometric contexts \cite[Corollary~3.11]{Ulirsch_functroplogsch}, which implies that the $2$-category of cone stacks is equivalent to the $2$-category of (fine and saturated) Kato stacks, defined in \cite[Section~2]{Ulirsch_nonArchArtin}.

In addition, we will see in Theorem~\ref{thm:artin-fans} that cone stacks are also equivalent to {\em Artin fans}.
\end{remark}

\subsection{From monoids to cones}
\label{sec:mon-to-real}

A \emph{polyhedral cone} is a pair $(V,\sigma)$ consisting of a finite-dimensional real vector space $V$ and a strictly convex intersection of finitely many closed half spaces, defined by real linear inequalities, whose linear span is $V$. A morphism $(V_1,\sigma_1)\rightarrow (V_2, \sigma_2)$ is an $\R$-linear map $f\mathrel{\mathop:}V_1\rightarrow V_2$ such that $f(\sigma_1)\subseteq \sigma_2$.  We write $\RRPC$ for the category of polyhedral cones.

Given a polyhedral cone $(V,\sigma)$ we write $\sigma^\vee$ for its dual cone in the dual vector space $V^\ast=\Hom_\R(V,\R)$ of $V$. The dual cone $\sigma^\vee$ may be regarded as a sharp submonoid of $V^\ast$, so the assignment $\sigma \mapsto \sigma^\vee$ determines a functor
\begin{equation*}
\RRPC \to \ShpMon^{\mathrm{op}}.
\end{equation*}So, if $\calC$ is a moduli problem on $\ShpMon^{\mathrm{op}}$ then it may be restricted via this functor to a moduli problem on $\RRPC$.  In particular, we may speak of the real points of $\calC$ by evaluating it on the real cone $\mathbb R_{\geq 0}$, and we may apply this as well to a moduli problem defined a priori only over $\RPC$ using the preceding section.
We note that by virtue of its construction as a restriction, this process preserves fiber products.

\begin{remark}
The construction above is the standard one for a morphism of topoi $X \to Y$ associated to a cocontinuous functor $X \to Y$, with $X = \RRPC$ and $Y = \ShpMon^{\mathrm{op}}$.

It is also possible to construct the morphism of sites $\RRPC \to \RPC$ in one step using the left exact, continuous functor $\RPC \to \RRPC$ that sends a rational polyhedral cone $(N, \sigma)$ to its underlying cone $\sigma$.
However, the order of operations presented here seems to be preferable, since the extension of a moduli problem from $\RPC$ to $\ShpMon^{\mathrm{op}}$ is frequently available ab initio, and passage to $\RRPC$ is then a simple matter of restriction.
\end{remark}

\begin{remark}
In fact, the above construction applies word for word to the larger category of convex cones, i.e.~not-necessarily-finite intersections of real half spaces, since a convex cone $\sigma$, polyhedral or not, produces a sharp monoid $\sigma^\vee$.
\end{remark}

Now, applying this process to the moduli space of tropical curves, we may recognize $\calM_{g,n}^{\trop}(\mathbb R_{\geq 0})$ as the groupoid of graphs metrized by $\mathbb R_{\geq 0}$.

\begin{proposition}
Suppose that $\sigma$ is a polyhedral cone and that $\calM_{g,n}^{\trop,\R}$ denotes the restriction of $\calM_{g,n}^{\trop}$ to $\RRPC$.  Then $\calM_{g,n}^{\trop,\R}(\sigma)$ is the groupoid of tropical curves with edge lengths metrized by $\sigma^\vee$.
\end{proposition}
\begin{proof}
This is immediate from the construction.
\end{proof}

\subsection{From real cones to topological spaces}
\label{sec:real-to-top}\label{section_topstacks}

The functor $\RRPC \to \Top$ that sends a cone $\sigma$ to its underlying topological space $|\sigma|$ preserves fiber products.  Therefore, for any topological space $X$, the category of all pairs $(\sigma, \varphi)$, where $\sigma \in \RRPC$ and $\varphi : X \rightarrow |\sigma|$ is a continuous function, is filtered.  We can therefore extend a moduli problem from $\RRPC$ to $\Top$ in a manner similar to the one by which we extended from rational polyhedral cones to monoids.

We define:
\begin{equation*}
\widetilde\calC(X) = \varinjlim_{\substack{\sigma \in \RRPC \\X \to |\sigma|}}\calC(\sigma)
\end{equation*}

\begin{remark}
The similarity to the construction from Section~\ref{sec:rpc-to-mon} is hardly coincidental, as we are again applying the standard construction of a morphism of topoi from a left exact, continuous functor.  Note that if we want to treat $\Top$ as a geometric context then we must take $\mathbb P$ to be a very general class of local models, such as maps that are locally topological quotients on source and target; we prefer to not pursue this here, since it is only tangentially related to the main objectives of this article.
\end{remark}

One may now stackify $\widetilde\calC$ to obtain a topological stack $\lvert\calC\rvert$.  More generally, if $\calC$ is a moduli problem over $\ShpMon^{\mathrm{op}}$ or over $\RPC$, we define $\lvert\calC\rvert$ by first extending $\calC$ to $\RRPC$ using the sections above and then applying $\lvert\:\cdot\:\rvert$ to the result.  We emphasize that all of these manifestations of $\lvert\:\cdot\:\rvert$ preserve fiber products.

\begin{lemma}
Let $\widetilde\calC$ be constructed as above and let $\ast$ denote a topological space with one point.  Then $\widetilde\calC(\ast) = \calC(\mathbb R_{\geq 0})$.
\end{lemma}

\begin{proof}
We may evaluate the colimit explicitly.  If $\sigma$ is a polyhedral cone and $x \in |\sigma|$ is any point, there is a unique morphism $\mathbb R_{\geq 0} \to \sigma$ that takes $1$ to $x$.  This implies that the pair $(\mathbb R_{\geq 0}, 1)$ is initial in the category of all $(\sigma, \varphi)$ where $\sigma \in \RRPC$ and $\varphi : \ast \rightarrow |\sigma|$ is continuous.  The colimit therefore reduces to the evaluation of $\calC$ on $\mathbb R_{\geq 0}$.
\end{proof}

Applying this to $\calM_{g,n}^{\trop}$, we obtain a topological stack underlying the moduli stack of tropical curves.  Theorem~\ref{thm_universalcurve} now implies

\begin{proposition}\label{cor_topuniversalcurve}
Given a point $[\Gamma]\in\big\vert \calM_{g,n}^{\trop}\big\vert$ corresponding to a tropical curve $\Gamma$. Then the preimage $\vert\pi_{g,n+1}\vert^{-1}\big([\Gamma]\big)$ of $[\Gamma]$ (in the sense of topological stacks) under the forgetful morphism
\begin{equation*}
\vert \pi_{g,n+1}\vert\mathrel{\mathop:}\big\vert\calM_{g,n+1}^{\trop}\big\vert\longrightarrow\big\vert\calM_{g,n}^{\trop}\big\vert
\end{equation*}
is homeomorphic to $\Gamma$. 
\end{proposition}
\begin{proof}
Let $\R_{\geq 0} \to \calM_{g,n}^{\trop}$ be a morphism corresponding to a tropical curve $\Gamma$ with real edge lengths.  We note first that because $\lvert\:\cdot\:\rvert$ was defined using a filtered colimit, it respects fibered products, and there is a natural isomorphism
\begin{equation*}
\lvert\mathbb R_{\geq 0}\rvert \mathop\times_{\lvert \calM_{g,n}^{\trop} \rvert} \lvert \calM_{g,n+1}^{\trop} \rvert \simeq \bigl\vert \mathbb R_{\geq 0} \mathop\times_{\calM_{g,n}^{\trop}} \calM_{g,n+1}^{\trop} \bigr\vert ,
\end{equation*}
 where the fiber product in the middle should be understood over $\ShpMon^{\mathrm{op}}$.  The map $\mathbb R_{\geq 0} \to \calM_{g,n}^{\trop}$ factors through a rational polyhedral cone $\sigma$ and map $\sigma \to \calM_{g,n}^{\trop}$ of rational cone stacks, corresponding to a tropical curve $\Gamma'$ over $\sigma$.  We have
\begin{equation*}
\mathbb R_{\geq 0} \mathop\times_{\calM_{g,n}^{\trop}} \calM_{g,n+1}^{\trop} \simeq \mathbb R_{\geq 0} \mathop\times_{\sigma} \sigma \mathop\times_{\calM_{g,n}^{\trop}} \calM_{g,n+1}^{\trop} \simeq \mathbb R_{\geq 0} \mathop\times_{\sigma} \Cone(\Gamma')
\end{equation*}
where the last identification is Theorem~\ref{thm_universalcurve}.  It is easily seen that the fiber over $1 \in |\mathbb R_{\geq 0}| \to |\sigma|$ of $|\Cone(\Gamma')| \to |\sigma|$ is the underlying topological space of $\Gamma$.
\end{proof}

{
\subsection{General monoids and extended tropical curves}\label{sec:com-mon}
There is nothing in the definition of a tropical curve that requires the monoid in which its edge lengths take values to be sharp, integral, and saturated.  We could just as easily allow the monoid to be arbitrary, even noncommutative, or even just a pointed set.  Allowing values in monoids with absorbing elements will allow us to treat the case of \emph{extended tropical curves}, as introduced in \cite{Caporaso_tropicalmoduli} and \cite{AbramovichCaporasoPayne_tropicalmoduli}. 

Let $\ComMon$ be the category of all commutative monoids, not necessarily sharp, integral, or saturated.  

\begin{definition}
Define $\calMbar^{\rm trop}_{g,n}$ to be the category fibered in groupoids over $\ComMon^{\rm op}$, whose fiber over a commutative monoid $P$ is groupoid of stable tropical curves in the sense of Definition~\ref{def_tropicalcurve}, but with edge length in the monoid $P$.
\end{definition}

Denote by $\Rbar_{\geq 0}$ the set $\R_{\geq 0}\sqcup\{\infty\}$ with the monoid structure that extends the additive monoid structure on $\R_{\geq 0}$ by setting $a+\infty=\infty$ for all $a\in\Rbar_{\geq 0}$.

An \emph{extended tropical curve} in the sense of \cite{Caporaso_tropicalmoduli} and \cite{AbramovichCaporasoPayne_tropicalmoduli} is a tropical curve with edge lengths in the monoid $\Rbar_{\geq 0}$. Observe that the set $\calMbar_{g,n}^{trop}\big(\Rbar_{\geq 0}\}\big)$ of $\Rbar_{\geq 0}$-valued points is nothing but the set-theoretic moduli space of extended tropical curves studied in \cite{AbramovichCaporasoPayne_tropicalmoduli}.

The inclusion of $\ShpMon$ in $\ComMon$ has a left adjoint that sends a commutative monoid to the initial sharp, integral, saturated monoid to which it admits a morphism.  As a left adjoint this functor preserves colimits, so its opposite functor preserves limits.  There is therefore a morphism of sites
\begin{equation*}
\Phi : \ShpMon^{\rm op} \to \ComMon^{\rm op}
\end{equation*}
in which both source and target are given the chaotic topologies, {and the pullback functor $\Phi^\ast$ sends a monoid to its associated sharp, integral, saturated monoid, as described above.}

For the following proposition, let $\Phi_!$ denote the left adjoint to $\Phi^\ast$; {the functor $\Phi_!$} is induced from the inclusion of sharp, integral, saturated monoids in all commutative monoids.  To be precise, if $F$ is a category fibered in groupoids over $\ShpMon^{\rm op}$ then, for any commutative monoid $M$, we have the following formula:
\begin{equation*}
\Phi_! F(M) = \varinjlim_{\substack{N \to M \\ N \in \ShpMon}} F(N)
\end{equation*}

{
The following proposition implies that working in the category of all commutative monoids does not change the moduli space of curves:  the same system of charts, by the same rational polyhedral cones, still glues to give $\calMbar^{\rm trop}_{g,n}$.
}

{\begin{proposition}
{There are natural equivalences 
\begin{equation*}
\Phi^\ast \calMbar^{\rm trop}_{g,n} \simeq \calM^{\rm trop}_{g,n} \qquad \textrm{ and } \qquad \Phi_! \calM^{\rm trop}_{g,n} \simeq \calMbar^{\rm trop}_{g,n}.\end{equation*}}
\end{proposition}
\begin{proof}
The first of these assertions is immediate from the definitions.  For the second, observe that the first assertion and the adjunction between $\Phi_!$ and $\Phi^\ast$ induce a morphism
\begin{equation*}
\rho : \Phi_! \calM^{\rm trop}_{g,n} \to \calMbar^{\rm trop}_{g,n}
\end{equation*}
that we argue is an isomorphism.  

Indeed, suppose that $M$ is a commutative monoid and $\overline \Gamma \in \calMbar^{\rm trop}_{g,n}(M)$ then let $R$ be the category of all tuples $(N, \Gamma, \varphi)$ where $N \in \ShpMon$, $\Gamma \in \calM^{\rm trop}_{g,n}(N)$, and $\varphi : \Gamma \to \overline \Gamma$ is a morphism in $\calMbar^{\rm trop}_{g,n}$.  The associated groupoid of the category $R$, i.e. the initial of all maps from R to
groupoids,  is the fiber of $\rho$ over $\overline \Gamma$.  We need to show that the associated groupoid of $R$ is trivial.

We do so by constructing a deformation retraction --- that is, an endofunctor with a natural transformation to the identity endofunctor --- onto a subcategory with a final object.  Upon passage to the associated groupoid, this natural transformation will become an isomorphism, because all morphisms in a groupoid are isomorphisms.

First, if $(N, \Gamma, \varphi)$ is an object of $R$, let $N'$ be the free commutative monoid generated by the edges of $\Gamma$, let $\Gamma'$ be the tropical curve with the same underlying graph as $\Gamma$, but labelled in the canonical way by $N'$, and let $\varphi' : \Gamma' \to \overline\Gamma$ be the edge contraction obtained by composing $\Gamma' \to \Gamma \to \overline\Gamma$.  This construction respects morphisms in $R$, and there is a canonical natural transformation from the functor $(N,\Gamma,\varphi) \mapsto (N',\Gamma',\varphi')$ to the identity functor on $R$.  Passing to associated groupoids, this functor turns into an equivalence of categories between the groupoid associated to $R$ onto the groupoid associated to the full subcategory $R'$ of $R$ consisting of triples $(N,\Gamma,\varphi)$ where $N$ is freely generated by the edges of $\Gamma$.

Now we observe that the category $R'$ has a final object, where the underlying graph of $\Gamma$ coincides with that of $\overline\Gamma$, and $N$ is freely generated by its edges.  The associated groupoid is therefore trivial, i.e. has a unique morphism from any
object to any other.
\end{proof}
}

In \cite{HuszarMarcusUlirsch_troplogclutch&glue}, the authors work with the moduli stacks $\calMbar_{g,n}^{trop}$  and show that there are natural clutching morphisms
\begin{equation*}
\overline{\delta}_{g_1,I_1}^{trop}\colon \calMbar_{g_1,I_1\cup\{\star\}}^{trop}\times \calMbar_{g_2,I_2\cup\{\bullet\}}^{trop}\longrightarrow \calMbar_{g,[n]}^{trop}
\end{equation*}
and
\begin{equation*}
\overline{\delta}_{irr}^{trop}\colon \calMbar_{g-1,[n]\cup\{\star,\bullet\}}^{trop}\longrightarrow \calMbar_{g,[n]}^{trop}
\end{equation*}
of categories fibered in groupoids over the category of pointed monoids, i.e. monoids with an absorbing element $\infty$ for which we have $a+\infty=\infty$ for all $a$ in the monoid,  %$\ShpMon_{\infty}^{\mathrm{op}}$\footnote{\rcolor{The subscript $\infty$ here denotes the restriction to the category of pointed monoids: there is an absorbing element, denoted $\infty$, such that $a+\infty = \infty$ for any element $a$ of the monoid. }}
that are given by assigning to the new edge the length $\infty$. We refer the reader to \cite{HuszarMarcusUlirsch_troplogclutch&glue} for details.

\begin{remark} In \cite{HuszarMarcusUlirsch_troplogclutch&glue} the authors also develop the notion of \emph{pointed logarithmic structures} (using sheaves of pointed monoids instead of sheaves of monoids). This notion is flexible enough to treat the classical clutching morphisms as morphisms of pointed logarithmic stacks and not as correspondences as in Section \ref{section_clutching} below. The main result of \cite{HuszarMarcusUlirsch_troplogclutch&glue} is that these morphisms naturally commute with the process of tropicalization.
\end{remark}
}

%%%%%%%%%%%%%%%%%%%%%%%%%%%%%%%%%%%%%%%%%%%%%%%%%%%%%%

%%%%%%%%%%%%%%%%%%%%%%%%%%%%%%%%%%%%%%%%%%%%%%%%%%%%%%

\bigskip

\part{Tropicalization and logarithmic geometry}
\bigskip

\section{Lifting cone stacks to the logarithmic category}
\label{sec:log-generalities}

%%%%%%%%%%%%%%%%%%%%%%%%%%%%%%%%%%%%%%%%%%%%%%%%%%%%%%

Our goal in this section is to extend tropical moduli problems to logarithmic moduli problems, which will be achieved via a functor $a^*$ (Definition~\ref{def:a*}) that lifts cone stacks to algebraic stacks with logarithmic structure.  In fact, $a^*$ defines an equivalence of $2$-categories between cone stacks and {\em Artin fans}, as in Theorem~\ref{thm_artinfans}.  The discussion in this section is not specific to the tropical moduli space of curves.  

The results of this section set up Sections~\ref{sec:log-curves} and Section~\ref{sec:tropicalization}, which \emph{are} specific to  the moduli space of tropical curves.  We will see that by associating to a logarithmic curve its dual tropical curve, we obtain a smooth morphism from the moduli space of logarithmic curves to the moduli space of tropical curves.  In Section~\ref{sec:tropicalization} we show that all of the tautological morphisms between the moduli spaces of logarithmic curves are compatible with their tropical analogues, as introduced in Section~\ref{section_families}.

We start by giving a brief review of logarithmic geometry and refer the reader to \cite{Kato_logstr} and \cite{Abramovichetal_log&moduli} for more details as well as a survey of its applications to compactifications of moduli spaces.
Throughout this section, we work over an algebraically closed field $k$, although many of the results presented here are valid over a more general base, with anodyne modifications (e.g., replacing finite type assumptions with finite presentation).  Unqualified, a relative term is applied relative to $k$:  for example, `locally of finite type' means `locally of finite type over $k$'.

\subsection{Logarithmic structures}

A reader who is not already familiar with logarithmic structures may find the discussion of logarithmic curves in Sections~\ref{sec:valuation} and~\ref{sec:log-curves-defs} useful as an example while reading this section.

Throughout this section all monoids are assumed integral and saturated.  Most treatments of logarithmic geometry restrict attention to monoids that are also finitely generated, but we find the added generality convenient.  For example, the monoid of positive elements in the value group of a non-discretely valued field is not finitely generated, but arises naturally as the characteristic monoid of a logarithmic structure in the example of Section~\ref{sec:valuation}.

A logarithmic structure on a scheme $X$ is an \'etale sheaf of monoids $M_X$ and a homomorphism $\varepsilon\mathrel{\mathop:}M_X \rightarrow \mathcal O_X$ of unital semigroups (the latter with its multiplicative semigroup\footnote{Since $\mathcal O_X$ is not integral, our conventions forbid us from calling this a homomorphism of monoids.} structure) such that every unit in $\mathcal O_X^\ast$ has a unique preimage in $M_X$.  We write $\ell : \mathcal O_X^\ast \rightarrow M_X$ for the unique section of $\varepsilon$ over $\mathcal O_X^\ast$.  The \emph{characteristic monoid} of the logarithmic structure is $\overline{M}_X = M_X / \ell(\mathcal O_X^\ast)$.  A morphism of logarithmic structures on $X$ is a homomorphism of sheaves of monoids $M_X \rightarrow M'_X$ fitting into a commutative triangle:
\begin{equation*} \xymatrix@R=10pt{
M_X \ar[dr] \ar[dd] \\
& \mathcal O_X \\
M'_X \ar[ur]
} \end{equation*}
If $X$ is a logarithmic scheme, we write $\underline X$ for its underlying scheme.

Suppose that $P$ is a monoid, that $X$ is a scheme, and that $P \rightarrow \Gamma(X, \mathcal O_X)$ is a homomorphism of monoids.  There is an initial logarithmic structure $M_X \rightarrow \mathcal O_X$ among those equipped with a factorization $P \rightarrow \Gamma(X, M_X) \rightarrow {\Gamma(X, \mathcal O_X)}$.  This is called the \emph{associated logarithmic structure}, and the map $P \rightarrow \Gamma(X, \mathcal O_X)$ is called a \emph{chart} for $M_X$. A logarithmic structure on a scheme that admits charts \'etale-locally is called \emph{quasicoherent}.  We will have no use for logarithmic structures that are not quasicoherent, so \textit{from now on, all logarithmic structures will be assumed quasicoherent.}

A logarithmic structure is called \emph{coherent} if it has local charts by finitely generated monoids.  A logarithmic scheme is said to be \emph{locally of finite type} if its underlying scheme is locally of finite type  and its logarithmic structure is coherent.

\begin{example} For example, if $D$ is a smooth Cartier divisor on $X$ with local equation $t$, then the logarithmic structure associated to the chart given by sending $n\in \mathbb{N}$ to $t^n$ is the usual divisorial logarithmic structure for $D$. This includes the special case of $\Spec R$ for $R$ a discrete valuation ring.%, recovering the setup of \S\ref{sec:valuation}.
\end{example}

\begin{example} 
Suppose that $D$ is a normal crossings divisor in a smooth scheme $X$.  This implies that \'etale-locally in $X$, we may write $D$ as the vanishing locus of a product $x_1 \cdots x_r$, and that the $x_i$ are uniquely determined up to scaling and permutation.  We obtain a chart $\mathbb N^r \to \Gamma(U, \mathcal O_X)$ for each sufficiently small \'etale map $U \to X$.  This serves as a chart for a logarithmic structure, $M_X$.

If $p$ is a point of $X$ then the rank of $\overline M_X$ is the number of branches of $D$ incident to $p$.  For example, if $X = \mathbb A^2$ and $D$ is the union of the axes, then $\overline M_X$ will have rank~$0$ away from $D$, rank~$1$ at all points of $D$ other than the origin, and rank~$2$ at the origin.
\end{example}

Recall that a monoid is called {\em fine} if it is finitely generated and integral.  We say that a logarithmic scheme is {\em fine and saturated} if there is an \'etale cover over which it has fine and saturated charts.

If $S$ is a logarithmic scheme that is locally of finite type then $S$ has a locally finite stratification into locally closed subschemes on which the characteristic monoid $\overline M_S$ is %locally constant.  
constant.
Indeed, to demonstrate this, we may assume that $S = \Spec A$ is affine and has a global chart $\varepsilon\colon P\to A$ by a finitely generated monoid $P$.  For any element $\alpha$ of $P$, we get a function $\varepsilon(\alpha) \in A$, which partitions $S$ into an open subset $D(\alpha)$, where $\varepsilon(\alpha)$ does not vanish, and a closed subset $V(\alpha)$, where $\varepsilon(\alpha)$ does vanish.  If $P$ is viewed as the set of integral points of a rational polyhedral cone then $D(\alpha) = D(\beta)$ whenever $\alpha$ and $\beta$ are contained in the interior of the same face of $P$.  This implies that the collection of closed subsets $V(\alpha)$, as $\alpha$ ranges through $P$, is finite, which implies that the $V(\alpha)$ determine the claimed finite stratification of $S$. 

Generalizing the concept of a chart, we consider a morphism of schemes $\pi : X \rightarrow S$ and a logarithmic structure $M_S$ on $S$.  We define $\pi^\ast M_S$ to be the initial logarithmic structure on $X$ equipped with a commutative square:
\begin{equation*} \xymatrix{
\pi^{-1} M_S \ar[r] \ar[d] &   \pi^{-1} \mathcal O_S \ar[d] \\
\pi^\ast M_S \ar[r] & \mathcal O_X
} \end{equation*}

A \emph{morphism of logarithmic schemes} $(X, M_X) \rightarrow (S, M_S)$ consists of a morphism of schemes $\pi : X \rightarrow S$ and a morphism of logarithmic structures $\pi^\ast M_S \rightarrow M_X$.  The morphism is called \emph{strict} if $\pi^\ast M_S \rightarrow M_X$ is an isomorphism.

If $S = \Spec K$ is the spectrum of an algebraically closed field then \'etale sheaves on $S$ are simply sets.  A logarithmic structure on $S$ is therefore nothing but a homomorphism $M_S \rightarrow K$ admitting a section $K^\ast \to M_S$.  The monoid $M_S$ then splits, non-canonically, as $\overline M_S \times K^\ast$.

\subsection{Morphisms between logarithmic schemes and cone stacks}
\label{sec:log-to-cs}

In this section, we will show how a moduli problem on rational polyhedral cones can be extended, canonically, to a moduli problem on logarithmic schemes.

Suppose that $\calC$ is a category fibered in groupoids over $\RPC$.  
We define an associated category fibered in groupoids $\mathcal A_\calC$ over the category of logarithmic schemes by setting
\begin{equation} \label{eqn:mon-to-log}
\mathcal A_\calC(S) = \calC \bigl( \Gamma(S, \overline M_S) \bigr)
\end{equation}
for a logarithmic scheme $S$, where we have extended $\calC$ to a functor defined on all monoids as in Section~\ref{sec:rpc-to-mon}. 
Note that the extension of $\calC$ to monoids that are not finitely generated was important  so that we could make sense of $\mathcal A_\calC(S)$ when $S$ is not of finite type. 

Unfortunately, $\mathcal A_\calC$ is not a stack in general.  It does, at least, turn out to be a stack when $\calC = \sigma$ is a cone, as the following lemma demonstrates.

\begin{lemma}[\cite{Olsson_LOG} Proposition 5.17]\label{lemma_Artincone}
Suppose that $\sigma$ is a rational polyhedral cone.  The stack $\mathcal A_\sigma$ is isomorphic to $\big[V_\sigma\big/T\big]$ where $V_\sigma = \Spec k[S_\sigma]$  is the toric variety associated to $\sigma$, with its standard logarithmic structure, and $T$ is its dense torus.  In particular, the stack $\mathcal A_\sigma$ is representable by an algebraic stack with a logarithmic structure.
\end{lemma}
\begin{proof}
Note that $V_\sigma$ represents the functor sending a logarithmic scheme $X$ to $\Hom\bigl(S_\sigma, \Gamma(X, M_X) \bigr)$ and that $T$ represents the functor sending $X$ to $\Hom\bigl(S_\sigma, \Gamma(X, \mathcal O_X^\ast) \bigr)$.
Therefore $\big[V_\sigma\big/T\big]$ represents the quotient functor 
\begin{equation*}
X\longmapsto\Hom\bigl(S_\sigma, \Gamma(X, M_X/\mathcal O_X^\ast) \bigr) = \mathcal A_\sigma(X) \ .
\end{equation*} \end{proof}

\begin{remark}
Even though $\big[V_\sigma\big/T\big]$ is an algebraic stack and not an algebraic space, when it is equipped with the logarithmic structure from the lemma, $\big[V_\sigma\big/T\big]$ represents a \emph{functor} (not a category fibered in groupoids) over logarithmic schemes (see \cite[Proposition 5.17]{Olsson_LOG}).  This is an example of a remarkable rigidification that occurs in logarithmic geometry.  A related phenomenon is that logarithmic blowups behave like injections:  a logarithmic blowup of a logarithmic scheme $X$ is defined as a \emph{subfunctor} of the functor represented by $X$ \cite[Definition~3.8]{Kato_LogMod}; see Lemma~\ref{lem:puncture} for an example of such a construction.
\end{remark}

\begin{remark}\label{remark_a*}
The assignment $\sigma \mapsto \mathcal A_\sigma$ determines a continuous, left exact functor from $\RPC$ to presheaves on the strict \'etale site on the category of logarithmic schemes.  By passing to the associated sheaf, it therefore determines a morphism of sites $a : \Et(\mathbf{LogSch}) \rightarrow \RPC$ where $a^\ast \sigma = \mathcal A_\sigma$.
\end{remark}

\begin{definition}\label{def:a*}
Inspired by the notation in Remark \ref{remark_a*}, we write $a^\ast \calC$ for the stackification of the category fibered in groupoids $\mathcal A_\calC$.
\end{definition}

\begin{corollary} \label{cor:a*}
If $\calC$ is a cone stack then $a^\ast \calC$ is representable by an algebraic stack with a logarithmic structure that is locally of finite presentation (i.e., locally of finite type and quasiseparated).  Moreover, if $\calC$ has a strict cover by finitely many cones then $a^\ast \calC$ is quasicompact.
\end{corollary}
\begin{proof}
Lemma \ref{lemma_Artincone} proves this when $\calC$ is a cone. Consider next the case where $\calC$ is a union of faces of a single cone.  If $\tau \rightarrow \sigma$ is a strict morphism (i.e., the embedding of a face), it follows immediately from the construction of $\mathcal A_\sigma$ in the lemma that $\mathcal A_\tau \subseteq \mathcal A_\sigma$ is an open embedding.  If $\calC$ has a strict embedding in a cone, then $\calC$ is a union of open subcones and therefore $a^\ast \calC$ is a union of open substacks of $\mathcal A_\sigma$, hence is algebraic.

More generally, suppose that $\calC$ is a cone space that injects (not necessarily strictly) into a cone $\sigma$. By definition, there is a covering family $\tau \rightarrow \calC$ of strict morphisms from cones.  If $\tau$ is one of those cones then the image of $\tau$ in $\sigma$ is a cone $\tau'$, necessarily contained in $\calC$, and the map $\tau \rightarrow \tau'$ is surjective.  Furthermore, $\tau \rightarrow \tau'$ is strict, since if $\omega \rightarrow \tau'$ is any morphism then $\tau \mathbin\times_{\tau'} \omega \simeq \tau \mathbin\times_{\calC} \omega \rightarrow \tau'$ is the base change of the strict morphism $\tau \rightarrow \calC$, hence is strict itself.  A strict surjective  morphism between cones is an isomorphism, so $\tau \rightarrow \tau'$ is an isomorphism.  That is $\calC$ has a cover by a union of strict \emph{sub}cones $\tau$.  But then the open substacks $a^\ast \tau$ cover $a^\ast \calC$ and $a^\ast \calC$ is therefore algebraic.

Now suppose that $\calC$ is a cone space.  Then $\calC$ is the quotient of a disjoint union of cones $\calC_0$ by a strict equivalence relation $\calC_1 \rightrightarrows \calC_0$ (which implies $\calC_1 \subset \calC_0 \times \calC_0$).  As $a^\ast$ preserves colimits, $a^\ast \calC$ is the colimit of $a^\ast \calC_1 \rightrightarrows a^\ast \calC_0$.  But $a^\ast \calC_1$ and $a^\ast \calC_0$ are representable by algebraic stacks, and the maps $a^\ast \calC_1 \rightarrow a^\ast \calC_0$ are local isomorphisms, so $a^\ast \calC$ is the quotient of an algebraic stack by an \'etale equivalence relation, hence is algebraic.

Finally, suppose $\calC$ is a cone stack.  We can present $\calC$ as the quotient of a disjoint union of cones $\calC_0$ by a strict groupoid $\calC_1 \rightrightarrows \calC_0$ where $\calC_1$ is a cone space.  As before, $a^\ast \calC_1$ and $a^\ast \calC_0$ are representable by algebraic stacks; moreover, the projections $\calC_1 \rightarrow \calC_0$ are strict morphisms of cone spaces, so the projections $a^\ast \calC_1 \rightarrow a^\ast \calC_0$ are representable by \'etale morphisms of algebraic spaces.  Therefore $\calC$ is the quotient of an algebraic stack by a groupoid whose projections are representable by algebraic spaces and \'etale; hence it is an algebraic stack.

For the claim about local finite presentation, note that $a^\ast \sigma$ is of finite presentation for any cone $\sigma$, so if $\calC$ is a cone stack then $a^\ast \calC$ has an \'etale cover by algebraic stacks that are locally of finite presentation, hence is itself locally of finite presentation.  Each of the $a^\ast \sigma$ is quasicompact, so if this cover is finite then $a^\ast \calC$ is quasicompact.
\end{proof}

\subsection{Artin fans}

The construction of $a^\ast$ in the last section serves as an interface between logarithmic geometry and tropical geometry.  This interface has been introduced elsewhere under the heading of \emph{Artin fans}.  We recall this terminology here, and indicate how it gives another perspective on cone stacks.

\begin{definition}
An \emph{Artin cone} is an algebraic stack with logarithmic structure that can be presented as the quotient $\calA_\sigma=\big[V_\sigma\big/T\big]$ of an affine toric variety $V_\sigma=\Spec k[S_\sigma]$ by its dense torus $T$.  An \emph{Artin fan} is a logarithmic algebraic stack that has a strict \'etale cover by a disjoint union of Artin cones.%
\end{definition}

In \cite{Ulirsch_nonArchArtin} and early drafts of \cite{AbramovichChenMarcusWise_boundedness} the term `Artin fan' was used more inclusively for logarithmic algebraic stacks that are logarithmically \'etale over the base field $k$.  We have elected to define it more restrictively here for the sake of Theorem~\ref{thm:artin-fans}.

\begin{lemma}\label{lemma_logetale}
Artin fans are logarithmically \'etale over the base.
\end{lemma}
\begin{proof}
Since Artin fans are modeled strict \'etale locally on Artin cones, it is sufficient to show that Artin cones are logarithmically \'etale over the base.  An Artin cone is certainly locally of finite presentation, since toric varieties and tori are.  Therefore it is sufficient to verify the formal criterion.

We consider a strict, square-zero extension $S \subseteq S'$, and a lifting problem:
\begin{equation*} \xymatrix{
S \ar[r] \ar[d] & \mathcal A_\sigma  \\
S' \ar@{-->}[ur] 
} \end{equation*}
We wish to show that there is a unique lift.  The map $S \rightarrow \mathcal A_\sigma$ comes from an element of $\Hom\bigl(S_\sigma, \Gamma(S, \overline M_S) \bigr)$.  But
\begin{equation*}
\Gamma(S', \overline M_{S'}) \rightarrow \Gamma(S, \overline M_S)
\end{equation*}
is an isomorphism because of $S'$ is a strict infinitesimal extension of $S$.  The unique lifting is therefore automatic.
\end{proof}

\begin{example}
There are algebraic stacks with logarithmic structures that are logarithmically \'etale over the base field $k$, but do not admit a strict \'etale cover by Artin cones.  One such example is the quotient of $\mathbb A^1$ by the $\mathbb G_m$-action $t . x = t^2 x$. 
\end{example}

\begin{theorem} \label{thm:artin-fans}
The functor $a^\ast$ defines an equivalence between the $2$-category of cone stacks and the $2$-category of Artin fans.
\end{theorem}
\begin{proof}
First we observe that $\Hom(\mathcal A_\sigma, \mathcal A_\tau) = \Hom_{\RPC}(\sigma, \tau)$.  Indeed, by definition of $\mathcal A_\tau$ (Equation~\ref{eqn:mon-to-log}),
\begin{equation*}
\Hom(\mathcal A_\sigma, \mathcal A_\tau) = \Hom\bigl(S_\tau, \Gamma(\mathcal A_\sigma, \overline M_{\mathcal A_\sigma}) \bigr) .
\end{equation*}

Writing $\mathcal A_\sigma$ as the quotient of the toric variety $\Spec k[S_\sigma]$ by its dense torus, we find that $\Gamma\big(\mathcal A_\sigma, \overline M_{\mathcal A_\sigma}\big)$ can be identified with the torus invariants in $\Gamma\big(\Spec k[S_\sigma], \overline M_{\Spec k[S_\sigma]}\big)$.  But the latter monoid %$\Gamma\big(\Spec k[S_\sigma], \overline M_{\Spec k[S_\sigma]}\big)$ 
may be identified with $S_\sigma$, as all sections are torus invariant.  Therefore
\begin{equation*}
\Hom(\mathcal A_\sigma, \mathcal A_\tau) = \Hom(S_\tau, S_\sigma) = \Hom(\sigma, \tau) .
\end{equation*} 
This proves that $a^\ast$ induces an equivalence of categories from rational polyhedral cones to Artin cones.

As a morphism $\sigma \rightarrow \tau$ is strict if and only if $\mathcal A_\sigma \rightarrow \mathcal A_\tau$ is strict and \'etale, and $\sigma \rightarrow \tau$ is surjective if and only if $\mathcal A_\sigma \rightarrow \mathcal A_\tau$ is surjective, we find that an algebraic stack with logarithmic structure has a strict \'etale cover by Artin cones if and only if $\calC$ has a strict cover by rational polyhedral cones.  That is, $a^\ast$ gives an equivalence of categories between cone stacks and Artin fans.
\end{proof}

\section{Tropicalizing the moduli space of logarithmic curves}
\label{sec:log-curves}

In Sections~\ref{sec:monoids} and~\ref{sec:rpc-to-mon}, we saw how to make sense of a family of tropical curves over an arbitrary sharp monoidal space.  In this section, we  define tropical curves over \emph{logarithmic schemes}. We show that the resulting moduli problem on logarithmic schemes is represented by an algebraic stack of finite presentation, logarithmically \'etale over $k$, and that it admits a tropicalization map from the moduli stack of logarithmic curves that is strict and smooth (Theorem \ref{thm:strictsmooth}).  

Technically, this material is not a direct application of Section~\ref{sec:rpc-to-mon}, because logarithmic structures sometimes require the \'etale topology and logarithmic schemes must therefore sometimes be seen as monoidal \emph{topoi}.  We do not engage with these technicalities directly --- the reader who is comfortable with them will not find it difficult to translate Section~\ref{sec:rpc-to-mon} into that language --- and our discussion  focuses only on logarithmic schemes, as opposed to monoidal topoi in general. 

%%The former 6.1
\subsection{Logarithmic curves over valuation rings}
\label{sec:valuation}

 To motivate the construction that  follows, we analyze a special case in which a logarithmic structure arises.  Let $K$ be a valued field with valuation ring $R$, and suppose that $X$ is a family of stable marked curves over $S = \Spec R$ with smooth generic fiber.  Let $\eta$ be the generic point of $S$ and let $X_\eta$ be the fiber of $X$ over the generic point.  

Let $G$ be the dual graph of the special fiber of $X$.  We can label each vertex by the genus of the normalization of the corresponding component.  An edge $e$ of the dual graph corresponds to a node of the special fiber, in an \'etale neighborhood of which, the family $X$ admits an \'etale $S$-morphism to
\begin{equation*}
\Spec R[x,y] / (xy - t_e) 
\end{equation*}
for some nonzero $t_e \in R$.  The element $t_e$ is uniquely determined by $e$ up to scaling by an element of $R^\ast$, so we call the image of $t_e$ in $(R \smallsetminus \{ 0 \}) / R^\ast$ the length of the edge $e$.  In fact, the element $t_e$ can be presented more canonically by viewing the orbit $R^\ast t_e$ as the set of generators for a principal ideal in $R$ that is dual to the deformation space of the node $e$.  For this reason, we refer to $t_e$ (up to scaling) as the \emph{smoothing parameter} for $e$.

As we will see in Section~\ref{sec:log-curves-defs}, the monoid $(R \smallsetminus \{ 0 \}) / R^\ast$ is the stalk $\Mbar_{S,s}$ of the characteristic monoid $\overline M_S$ of a logarithmic structure on $S$ at its closed point $s$.  Labelling each edge of the graph by the smoothing parameter $t_e \in (R \smallsetminus \{ 0 \}) / R^\ast = \overline M_{S,s}$, we have produced a tropical curve with a metric valued in $\Mbar_{S,s}$. 

Logarithmic structures give just enough additional information on the special fiber to carry out this construction without explicit reference to the smoothing family.  This is particularly convenient for applications to moduli problems, where a smoothing family may not be available.

%% The former 6.3

\subsection{Logarithmic curves and their tropicalizations}
\label{sec:log-curves-defs}

Suppose $S$ is a logarithmic scheme.  A \emph{loga\-rithmic curve} over $S$ is a logarithmically smooth morphism of logarithmic schemes $\pi : X \rightarrow S$ that is proper, has geometrically connected fibers, is integral, and is saturated.  Rather than define all of these terms, we recall a local structure theorem for logarithmic curves~\cite[Theorem 1.3]{Kato_logsmoothcurves} that also characterizes them:

\begin{theorem}[F.\ Kato] \label{thm:log-curve-local}
Suppose that $X$ is a logarithmic curve over a fine and saturated logarithmic scheme $S$.  Then every geometric point $x$ of $X$ has an \'etale neighborhood $V$ with a strict \'etale $S$-morphism  $\pi : V \rightarrow S$,  such that:
\begin{enumerate}[(i)]
\item $V = \Spec \mathcal O_S[u]$, with $M_V = \pi^\ast M_S$;
\item $V = \Spec \mathcal O_S[u]$, with $M_V = \pi^\ast M_S \oplus \N \upsilon$ with $\varepsilon(\upsilon) = u$;
\item $V = \Spec \mathcal O_S[x,y] / (xy - t)$ for some $t \in \mathcal O_S$, and
\begin{equation*}
M_V = \pi^\ast M_S \oplus \N \alpha \oplus \N \beta \big/ (\alpha + \beta = \delta)
\end{equation*}
for some $\delta \in M_S$, and with $\varepsilon(\alpha) = x$, $\varepsilon(\beta) = y$, and $\varepsilon(\delta) = t$.
\end{enumerate}
\end{theorem}

\begin{figure}[h] \label{fig:tropical-curve}
\begin{tikzpicture}
\draw[very thick] (0,0) .. controls (2,1) and (4,-1) .. (6,0);
\draw[very thick] (0.5,6) .. controls (2.5,7) and (4.5,5) .. (6.5,6);
\draw[very thick] (5.5,4.375) .. controls (4.875,3) and (4.875,3) .. (5.325,1.625);
\draw[very thick] (5.625,6.25) .. controls (5.125,5) and (5.125,5) .. (5.5,3.75);
\draw[very thick] (5.325,2.25) .. controls (4.75,1) and (4.75,1) .. (5.25,-0.5);

\draw[very thick] (1.875,6.75) .. controls (0,2.5) and (2,3.5) .. (1,-0.25);

\node at (3,-0.5) {$\N$};
\node at (6.5,-0.5) {$\N\oplus P$};
\node at (4.5,1) {$P$};
\node at (6.5,2) {$\N^2\oplus_\N P$};
\node at (4.625,3) {$P$};
\node at (6.625,4) {$\N^2\oplus_\N P$};
\node at (4.75,5) {$P$};
\node at (6.5,6.5) {$\N\oplus P$};
\node at (3.5,6.375) {$\N$};
\node at (0.5,3.25) {$0$};

\node at (0.875,-0.875) {$X$};
\node at (5.5,-1) {$X_s$};

\end{tikzpicture}
\caption{A logarithmically smooth degeneration of a smooth algebraic curve $X$ with marked sections over the spectrum of a discrete valuation ring $R$. Many different monoids $P$ are possible, but the choice consistent with Section~\ref{sec:valuation} is $P=(R - \{ 0 \}) / R^\ast$.}
\end{figure}
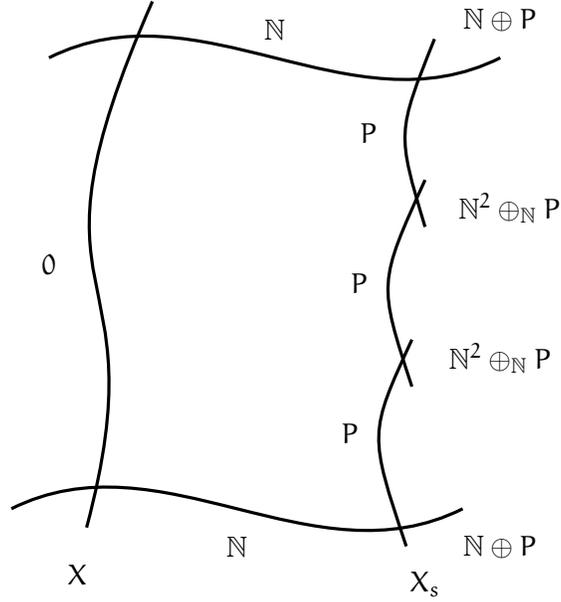

In particular, the stalk of $\overline M_X$ at the point $x$ in the statement of the theorem has one of the following forms:
\begin{enumerate}[(i)]
\item $\overline M_{X,x} \simeq \overline M_{S,\pi(x)}$ if $x$ is a smooth point of $\underline X$;
\item $\overline M_{X,x} \simeq \overline M_{S,\pi(x)} \oplus \mathbb N \upsilon$ if $x$ is a marked point of $\underline X$; and
\item $\overline M_{X,x} \simeq (\overline M_{S,\pi(x)} \oplus \mathbb N\alpha \oplus \mathbb N \beta) / (\alpha + \beta = \delta) \simeq \overline M_{S,\pi(x)} \oplus_{\mathbb N} \mathbb N^2$ if $x$ is a node of $\underline X$.
\end{enumerate}
In the last case, the element $\delta$ is known as the \emph{smoothing parameter} for the node $x$, as explained in Section~\ref{sec:valuation}.  

If $X$ is a logarithmic curve over a logarithmic point $S$, the observation above permits us to give the dual graph of $X$ the structure of a tropical curve, metrized by the characteristic monoid of $S$. 

\begin{definition}\label{def_tropicalization}
Let $S$ be a logarithmic scheme whose underlying scheme is the spectrum of an algebraically closed field and let $X$ be a logarithmic curve over $S$.  The \emph{dual tropical curve}  $\Gamma_X$ of $X$ is the tropical curve consisting of: 
\begin{enumerate}[(i)]
\item one vertex $v$ for each irreducible component $X_v$ of $X$, weighted by the genus of the normalization of $X_v$;
\item a flag $l_i$ incident to the vertex $v$ for each marked point $x_i$ on $X_v$, marked with the same label as $x_i$;
\item for each node $x_e$ of $X$, an edge incident to the two vertices  corresponding to components of $X$ adjacent to $x$, together with the edge length $d(e)=\delta_e$, the smoothing parameter of $x_e$.
\end{enumerate}
\end{definition}
We refer the reader to Figure \ref{figure_dualgraph} in the introduction for a picture of this construction.
\begin{remark}
The characteristic monoid at a node $x_e$ can alternatively be described as:
\begin{enumerate}[(i)]
\item[(iii$'$)] $\overline M_{X,x} \simeq \big\{ (\rho, \sigma) \in \overline M_{S,\pi(x)} \times \overline M_{S,\pi(x)} \: \big| \: \rho - \sigma \in \mathbb Z \delta \big\}$
\end{enumerate}
The advantage of this description is that it makes clear the generization maps associated to the two branches at the node:  if $x'$ is the generic point of a component incident to $x$ (so $\overline M_{X,x'} = \overline M_{S,s}$ and there is a specialization $x' \leadsto x$) then the corresponding map $\overline M_{X,x} \rightarrow \overline M_{X,x'} = \overline M_{S,s}$ is one of the two projections.  The identification between (iii) and (iii$'$) sends $\gamma + a \alpha + b \beta$ to $(\gamma + a \delta, \gamma + b \delta)$; {conversely, if $\rho -\sigma =  r\delta $, then  $(\rho,\sigma)$ is sent to $\sigma+ r\alpha$ if $r \geq 0$, and to $\rho - r\beta$ if $r \leq 0$}.  

This allows us to give a tropical interpretation to the sections of $\overline M_X$.  Recall that to specify an $\mathbb R_{\geq 0}$-valued affine linear function with integer slope along an interval is the same as to give a pair of nonnegative real numbers (the values of the function at the endpoints of the interval) whose difference is an integer multiple of the interval's length.  Thus sections of $\overline M_X$ near the nodes can be viewed as piecewise affine linear functions on the dual tropical curve of $X$ that have integer slopes along the edges, taking values in $\overline M_S$.

The same interpretation is also valid near the marked points.  Indeed, to give a $\mathbb R_{\geq 0}$-valued affine linear function with integer slope along an infinite ray, we should specify a nonnegative value at the endpoint of the ray, and a nonnegative slope.  Extrapolating to $\overline M_S$-valued functions, we find that affine linear functions along a ray correspond to elements of $\overline M_S \times \mathbb N$, which is the local structure of $\overline M_X$ near a marked point.%

It follows immediately from these observations that sections of $\overline M_X^{\rm gp}$ may be identified with piecewise linear functions on the dual graph of $X$, taking values in $\overline M_S^{\rm gp}$ and having integer slopes along the edges.  We view the sections of $\overline M_X$ as the submonoid of functions taking values $\geq 0$.

\end{remark}

Now we fix a logarithmic scheme $S$ whose underlying scheme is the spectrum of an algebraically closed field.  Assume that the characteristic monoid $\overline M_S$ is finitely generated and define $\sigma$ to be the rational polyhedral cone such that $S_\sigma = \overline M_S$.  The construction indicated above defines a functor
\begin{equation*}
\trop_{g,n,S}\colon\mathcal M_{g,n}^{\log}(S) \longrightarrow \mathcal M_{g,n}^{\trop}(\sigma) \ .
\end{equation*}

We now extend this construction to families over logarithmic schemes:
\begin{definition} \label{def:tropical-curve-log-base}
Let $S$ be a logarithmic scheme that is locally of finite type.  A \emph{tropical curve} over $S$ consists of
\begin{enumerate}[(i)]
\item \label{item:tropical-curve-1} a tropical curve $\Gamma_s$ with edge lengths in $\Mbar_s$ (see Section~\ref{section_modulifunctor}) for each geometric point $s$ of $S$;
\item \label{item:tropical-curve-2} a weighted edge contraction  $\Gamma_s \to \Gamma_t$ for each geometric specialization (see Appendix~\ref{sec:etale-specialization}) $t \leadsto s$ in $S$; such that
\item \label{item:tropical-curve-4} if $t \leadsto s$ is a geometric specialization then the map $\Gamma_s \to \Gamma_t$ is the contraction of weighted marked graphs in which $\Gamma_t$ is metrized by the composition
\begin{equation*}
E\big(\GG(\Gamma_s)\big)\xlongrightarrow{d_{\Gamma_s}}\Mbar_{S,s}\longrightarrow\Mbar_{S,t},
\end{equation*} 
with the usual convention that we contract edges of length zero.
\end{enumerate}
We write $\widetilde{\mathcal M}_{g,n}^{\trop}$ for the fibered category over logarithmic schemes of finite type whose fiber over $S$ is the category of families of genus~$g$ tropical curves over $S$ with $n$ marked points.
\end{definition}

\begin{remark}
We have defined $\widetilde{\mathcal M}_{g,n}^{\trop}$ only for logarithmic schemes locally of finite type because the general definition is more technical.  However, we will see soon that $\widetilde{\mathcal M}_{g,n}^{\trop}$ extends canonically to all logarithmic schemes when we recognize it as the extension of $\mathcal M_{g,n}^{\trop}$ to the category of logarithmic schemes.
\end{remark}

\begin{lemma}
$\widetilde{\mathcal M}_{g,n}^{\trop}$ is a stack in the strict \'etale topology on the category of logarithmic schemes that are locally of finite type.
\end{lemma}
\begin{proof}
We sketch how to construct an object of $\widetilde{\mathcal M}_{g,n}^{\trop}$ from a descent datum and leave the remaining verifications to the reader.

Let $U_i \rightarrow S$ be a strict \'etale covering family of a scheme $S$ and let $\Gamma_\bullet$ be a descent datum for $\widetilde{\mathcal M}_{g,n}^{\trop}$ over $U_\bullet$.  For each geometric point $s$ of $S$, choose a lift of $s$ to some $U_i$ and define $\Gamma_s = {\Gamma_i}_s$.  This is well-defined up to isomorphism, as a second lift $s \rightarrow U_j$ would induce a map $s \rightarrow U_{ij}$ and therefore an isomorphism ${\Gamma_i}_s \simeq {\Gamma_{ij}}_s \simeq {\Gamma_j}_s$.  The cocycle condition guarantees that the isomorphisms
\begin{gather*}
{\Gamma_i}_s \simeq {\Gamma_{ij}}_s \simeq {\Gamma_j}_s \simeq {\Gamma_{jk}}_s \simeq {\Gamma_k}_s \\
{\Gamma_i}_s \simeq {\Gamma_{ik}}_s \simeq {\Gamma_k}_s
\end{gather*}
both factor as ${\Gamma_i}_s \simeq {\Gamma_{ijk}}_s \simeq {\Gamma_k}_s$, hence coincide.  A similar argument constructs the generization maps associated to each geometric specialization $t \leadsto s$.
\end{proof}

Using this language, Definition \ref{def_tropicalization} actually defines a \emph{tropicalization morphism}
\begin{equation*}
\trop_{g,n}\colon\mathcal M_{g,n}^{\log} \longrightarrow \widetilde{\mathcal M}_{g,n}^{\trop}
\end{equation*}
of stacks over logarithmic schemes $S$ locally  of finite type:
consider an $S$-point of $\mathcal M_{g,n}^{\log}$, which is to say, a logarithmic curve of genus $g$ with $n$ marked points over $S$.  In Definition \ref{def_tropicalization} we have already constructed a tropical curve $\Gamma_s=\Gamma_{X_s}$ at each geometric point $s$ of $S$.  If $t \leadsto s$ is a specialization of geometric points then we can verify easily that the dual graph of $X_t$ is obtained from the dual graph of $X_s$ by collapsing those edges whose lengths map to $0$ in $\overline M_{S,t}$ because an element $\delta \in M_{S,s}$ maps to zero in $\overline M_{S,t}$ if and only if $\varepsilon(\delta)$ maps to a unit in $\mathcal O_{S,t}$, and this is equivalent to the smoothing of the corresponding node.  

\begin{remark}
Consider the situation where $S$ is a logarithmic scheme and $\overline M_S$ is a \emph{constant} sheaf of monoids.  If $\Gamma$ is a tropical curve over $S$, and if $t \leadsto s$ is a specialization of geometric points of $S$, then condition~\ref{item:tropical-curve-4} in Definition \ref{def:tropical-curve-log-base} above guarantees that $\Gamma_s \rightarrow \Gamma_t$ is an isomorphism.

If $S$ is merely assumed to be locally of finite type, then $S$ has a locally finite stratification into locally closed subsets on which $\overline M_S$ is locally constant.  The previous paragraph implies that a family of weighted, marked, metrized graphs $\Gamma$ over $S$ satisfying conditions~\ref{item:tropical-curve-1} through~\ref{item:tropical-curve-4} of the definition of a tropical curve, is locally constant on each stratum.
\end{remark}

%\begin{remark}
%It is possible to state Definition~\ref{def:tropical-curve-log-base} more succinctly, with a bit of technical abstraction.  Let $\mathbf{Mon}$ be the category of monoids and let $\mathcal M_{g,n}^{\trop\!+}$ be the extension of $\mathcal M_{g,n}^{\trop\!+}$ to a fibered category over $\mathbf{Mon}^\circ$ whose fiber over a monoid $M$ is the category of weighted, marked graphs, metrized by $M$.  For any logarithmic scheme $S$, let $\mathbf{Points}(S)$ be the category of geometric points of the underlying sheme of $S$, with morphisms being specializations.  Then the first three conditions of Definition~\ref{def:tropical-curve-log-base} describe a section of $\widetilde{\mathcal M}_{g,n}^{\trop\!+}$ over $\mathbf{Points}(S)$.  In other words, they give a dashed arrow
%\begin{equation*} \xymatrix{
%& \widetilde{\mathcal M}_{g,n}^{\trop\!+} \ar[d] \\
%\mathbf{Points}(S) \ar[r] \ar@{-->}[ur] & \mathbf{Mon}^\circ
%} \end{equation*}
%making the diagram commute \emph{on the nose}.  The lower horizontal arrow sends a geometric point $s$ of $S$ to $\overline M_{S,s}$.  Condition~\eqref{item:tropical-curve-4} is equivalent to requiring the dashed arrow to be cartesian.
%\end{remark}

\subsection{Algebraicity of the stack of tropical curves over logarithmic schemes}

In this section, we  prove that $\widetilde{\mathcal M}_{g,n}^{\trop}$ is isomorphic to $a^\ast \mathcal M_{g,n}^{\trop}$ over logarithmic schemes of finite type.  By Corollary~\ref{cor:a*} and Lemma~\ref{lemma_logetale}, this  implies the first half of Theorem \ref{thm_tropicalization} from the introduction, namely that $\widetilde{\mathcal M}_{g,n}^{\trop}$ is representable by an algebraic stack with a logarithmic structure that is logarithmically \'etale over $k$.  

To begin, we construct a natural family of maps
\begin{equation*}
\Phi : a^\ast \mathcal M_{g,n}^{\trop}(S) \longrightarrow \widetilde{\mathcal M}_{g,n}^{\trop}(S)
\end{equation*}
for every logarithmic scheme $S$ locally of finite type.  This means that, for every $S$, and every morphism $f : S \rightarrow a^\ast \mathcal M_{g,n}^{\trop}$, we want to build an associated family of tropical curves $\Phi(f)$.

Since $\widetilde{\mathcal M}_{g,n}^{\trop}$ is a stack, we can use the universal property of the associated stack $a^\ast \mathcal M_{g,n}^{\trop}$ (note that restriction from the strict \'etale site on the category of all logarithmic schemes to the strict \'etale site on the category of logarithmic schemes that are locally of finite type over $k$ preserves stackification) to specify $\Phi$ by giving maps
\begin{equation} \label{eqn:to-sheafify}
\mathcal M_{g,n}^{\trop}\bigl(\Gamma(S, \overline M_S)\bigr) \longrightarrow \widetilde{\mathcal M}_{g,n}^{\trop}(S)
\end{equation}
for each logarithmic scheme $S$ locally of finite type over $k$, in a manner that is natural in $S$.  An object of the domain is a tropical curve $\Gamma$, metrized by $\Gamma(S, \overline M_S)$.  The maps
\begin{equation*}
\Gamma(S, \overline M_S) \longrightarrow \overline M_{S,s}
\end{equation*}
induce tropical curves $\Gamma_s$, metrized by $\overline M_{S,s}$, for each geometric point $s$ of $S$, and the generization maps
\begin{equation*}
\overline M_{S,s} \longrightarrow \overline M_{S,t}
\end{equation*}
associated to geometric specializations $t \leadsto s$ induce maps $\Gamma_s \rightarrow \Gamma_t$.  This gives us the required map~\eqref{eqn:to-sheafify}.

\begin{lemma} \label{lem:isom-ext}
Suppose that $S$ is a logarithmic scheme locally of finite type, that $\Gamma$ and $\Gamma'$ are tropical curves over $S$ and $s$ is a geometric point of $S$.  If $\phi_s : \Gamma_s \xrightarrow{\sim} \Gamma'_s$ is an isomorphism then there is an \'etale neighborhood $U$ of $s$ and a unique extension of $\phi$ to $U$.
\end{lemma}
\begin{proof}
Let $Z$ be the stratum of $S$ containing $s$.  We can find an \'etale neighborhood $p : U \rightarrow S$ of $s$ such that $U$ is of finite type and
\begin{enumerate}[(a)]
\item $\overline M_S$ is constant on $p^{-1} Z$;
\item $p^{-1} Z$ is connected; 
\item $p^{-1} Z$ is the unique closed stratum of $U$.
\end{enumerate}
Then $\Gamma_t = \Gamma_s$ and $\Gamma'_t = \Gamma'_s$ for all points $t \in p^{-1} Z$ (since $p^{-1} Z$ is connected), so $\phi$ extends to $p^{-1} Z$.  But the closure of every stratum of $U$ contains a point of $p^{-1} Z$.
Hence if $W$ is a stratum of $U$ then there is some $w \in W$ and some $z \in p^{-1}Z$ and a specialization $w \leadsto z$.  Since $\Gamma_w$ and $\Gamma'_w$ are, respectively, the contractions of the edges in $\Gamma_z$ and $\Gamma'_z$ whose lengths generize to $0$ in $\overline M_w$, there is a \emph{unique} isomorphism $\phi_w : \Gamma_w \simeq \Gamma'_w$ making the diagram 
\begin{equation*} \xymatrix{
\Gamma_z \ar[r] \ar[d]_{\phi_z} & \Gamma_w \ar[d]^{\phi_w} \\
\Gamma'_z \ar[r] & \Gamma'_w
} \end{equation*}
commute. As before, this extends to an isomorphism over the whole stratum $W$ containing $w$.  The uniqueness guarantees that this extension does not depend on the choice of $w$ in $W$.
\end{proof}

\begin{lemma}
The morphism $\Phi$ constructed above is an isomorphism for all logarithmic schemes $S$ locally of finite type.
\end{lemma}
\begin{proof}
We fix $\Gamma \in \widetilde{\mathcal M}_{g,n}^{\trop}(S)$ and argue that, up to unique isomorphism, there is a unique lift of $\Gamma$ to $a^\ast \mathcal M_{g,n}^{\trop}(S)$.  Since both $a^\ast \mathcal M_{g,n}^{\trop}$ and $\widetilde{\mathcal M}_{g,n}^{\trop}$ are stacks, this is a local assertion on $S$.  We can therefore localize freely in $S$ and so can  assume that there is a geometric point $s$ and a chart of $\overline M_S$ by $\overline M_{S,s}$.  Using the map
\begin{equation*}
\overline M_{S,s} \rightarrow \Gamma(S, \overline M_S)
\end{equation*}
we extend $\Gamma_s$ to $\Gamma' \in \mathcal M_{g,n}^{\trop}\bigl(\Gamma(S, \overline M_S)\bigr)$.  Note that if $\Gamma$ is to lift to any object of $\mathcal M_{g,n}^{\trop}\bigl(\Gamma(S, \overline M_S)\bigr)$ it must be to $\Gamma'$, as the composition
\begin{equation*}
\mathcal M_{g,n}^{\trop}\bigl(\Gamma(S, \overline M_S)\bigr) \rightarrow \widetilde{\mathcal M}_{g,n}^{\trop}(S) \rightarrow \widetilde{\mathcal M}_{g,n}^{\trop}(s) = \mathcal M_{g,n}^{\trop}\bigl( \overline M_{S,s}\bigr)
\end{equation*}
is an isomorphism.  Let $\Gamma''$ be the image of $\Gamma'$ in $\widetilde{\mathcal M}^{\trop}_{g,n}(S)$.  Then by Lemma~\ref{lem:isom-ext}, there is a unique extension of the isomorphism $\Gamma''_s \simeq \Gamma_s$ to an \'etale neighborhood of $s$, which is precisely what is needed.
\end{proof}

\begin{corollary} \label{cor:rep}
The stack $\widetilde{\mathcal M}_{g,n}^{\trop}$ is representable by an algebraic stack of finite presentation (i.e., locally of finite presentation, quasicompact, and quasiseparated) over $k$ with a logarithmic structure that is logarithmically \'etale over $k$.  
\end{corollary}
\begin{proof}
Since $\widetilde{\mathcal M}_{g,n}^{\trop} \simeq a^\ast \mathcal M_{g,n}^{\trop}$, it is algebraic and locally of finite presentation.  We only need to verify that it is quasicompact and quasiseparated.  For the former, we note that $\mathcal M_{g,n}^{\trop}$ can be covered by finitely many cones.  For quasiseparation, we must show that the equalizer of a pair of map $f, g : Z \rightarrow \widetilde{\mathcal M}_{g,n}^{\trop}$ is quasicompact if $Z$ is quasicompact.  This is a local assertion on $Z$, so we can assume that there are cones $\sigma$ and $\tau$ and maps $\sigma \rightarrow \mathcal M_{g,n}^{\trop}$ and $\tau \rightarrow \mathcal M_{g,n}^{\trop}$ such that $f$ and $g$ factor through $\mathcal A_\sigma \rightarrow \widetilde{\mathcal M}_{g,n}^{\trop}$ and $\mathcal A_\tau \rightarrow \widetilde{\mathcal M}_{g,n}^{\trop}$, respectively.  Therefore the equalizer is the pullback of $\displaystyle \mathcal A_\sigma \mathop\times_{\widetilde{\mathcal M}_{g,n}^{\trop}} \mathcal A_\tau = a^\ast \bigl( \sigma \mathop\times_{\mathcal M_{g,n}^{\trop}} \tau \bigr)$ along the map $Z \rightarrow \mathcal A_\sigma \times \mathcal A_\tau$.  But by $\displaystyle \sigma \mathop\times_{\mathcal M_{g,n}^{\trop}} \tau$ has a finite cover by cones by Lemma~\ref{lemma_atlas=strict}, so the map $\displaystyle \mathcal A_\sigma \mathop\times_{\widetilde{\mathcal M}_{g,n}^{\trop}} \mathcal A_\tau \rightarrow \mathcal A_\sigma \times \mathcal A_\tau$ is quasicompact, as required.  
\end{proof}

\subsection{The tropicalization map for the moduli space of curves}
\label{sec:log-to-trop}

\begin{corollary} \label{cor:ext}
The morphism $\trop_{g,n}\colon \mathcal M_{g,n}^{\log} \rightarrow a^\ast \mathcal M_{g,n}^{\trop}$ extends to all logarithmic schemes. 
\end{corollary}
\begin{proof}
It is sufficient to describe the map 
\begin{equation*}
\mathcal M_{g,n}^{\log}(S) \longrightarrow a^\ast \mathcal M_{g,n}^{\trop}(S)
\end{equation*}
locally on $S$, since $a^\ast \mathcal M_{g,n}^{\trop}$ is a stack.  We can therefore assume that the underlying scheme of $S$ is affine and that the logarithmic structure has a global chart.  In that case, $S$ is the cofiltered limit of affine morphisms of logarithmic schemes of finite type $S_i$.  As $\mathcal M_{g,n}^{\log}$ is  locally of finite presentation, we obtain
\begin{equation*}
\mathcal M_{g,n}^{\log}(S) \xleftarrow{\sim} \varinjlim \mathcal M_{g,n}^{\log}(S_i) \rightarrow \varinjlim a^\ast \mathcal M_{g,n}^{\trop}(S_i) \xrightarrow{\sim}a^\ast\mathcal M_{g,n}^{\trop}(S) .
\end{equation*}
We omit the verification of the naturality of the extension.
\end{proof}

\begin{theorem} \label{thm:strictsmooth}
The tropicalization map $\trop_{g,n}\colon\mathcal M_{g,n}^{\log} \rightarrow a^\ast \mathcal M_{g,n}^{\trop}$ is strict, smooth, and surjective.
\end{theorem}
\begin{proof}
Given a stable tropical curve $\Gamma$ (of genus $g$ with $n$ marked legs) and edge lengths in a sharp monoid $P$, we can always find a logarithmic curve over $S=\big(\Spec k,k^\ast\oplus P\big)$ whose dual tropical curve is $\Gamma$. This shows that $\trop_{g,n}$ is surjective.

As $a^\ast \mathcal M_{g,n}^{\trop}$ is logarithmically \'etale over the base,  $\mathcal M_{g,n}^{\log}$ is logarithmically smooth over $a^\ast \mathcal M_{g,n}^{\trop}$ if and only if it is logarithmically smooth over the base.  Since $\mathcal M_{g,n}^{\log}$ is locally of finite presentation, it is sufficient to invoke the infinitesimal lifting criterion, and obstructions to infinitesimal deformations of a logarithmic curve $X$ are classified by $H^2(X, T_X^{\log})$, which vanishes because $X$ is a curve and $T_X^{\log}$ is quasicoherent.
As soon as we show that $\trop_{g,n}$ is strict, we will also know that it is smooth (in the usual, non-logarithmic sense), since strict logarithmically smooth morphisms are smooth \cite[Proposition 3.8]{Kato_logstr}.

To check that $\mathcal M_{g,n}^{\log}$ is strict over $a^\ast \mathcal M_{g,n}^{\trop}$, we must show that if $S$ is a scheme with a morphism of logarithmic structures $M'_S \to M_S$, then every diagram of solid arrows, as depicted in~\eqref{eqn:strict-lift}, has a unique lift:
\begin{equation} \label{eqn:strict-lift} \vcenter{\xymatrix{
(S, M_S) \ar[r] \ar[d] & \mathcal M_{g,n}^{\log} \ar[d] \\
(S, M'_S) \ar[r] \ar@{-->}[ur] & a^\ast \mathcal M_{g,n}^{\trop}
}} \end{equation}
We therefore assume that we have a logarithmic curve $(X, M_X)$ over $(S, M_S)$, a tropical curve $\Gamma'$ over $(S, M'_S)$ such that the family of tropical curves associated to $(X, M_X)$ is the family $\Gamma$ induced from $\Gamma'$ by the morphism of logarithmic structures $M'_S \rightarrow M_S$.  We wish to show that there is a unique logarithmic curve $(X, M'_X)$ over $(S, M'_S)$ inducing both $(X, M_X)$ and $\Gamma'$.  

We will construct the logarithmic structure $M'_X$ \'etale-locally on $X$.  Suppose that $x$ is a geometric point of $X$ and choose an \'etale neighborhood $U$ of $x$ such that one of the situations enumerated in Theorem~\ref{thm:log-curve-local} applies to $M_X$ and such that $\Gamma(U, \overline M_X) \rightarrow \overline M_{X,x}$ is a bijection.  Case by case, we define:
\begin{enumerate}[(i)]
\item $\overline N = \overline M'_{S,\pi(x)}$ if $x$ is a smooth point of $X$;
\item $\overline N = \overline M'_{S,\pi(x)} \oplus \mathbb N$ if $x$ is a marked point of $X$;
\item $\overline N = (\overline M'_{S,\pi(x)} \oplus \mathbb N\alpha' \oplus \mathbb N\beta') / (\alpha' + \beta' = \delta')$ if $x$ is a node of $X$ and $\delta'$ is the length of the corresponding edge in $\Gamma'$.
\end{enumerate}
In all cases, there is a canonical map from $\overline N$ to $\overline M_X$ over $U$, by virtue of the local structure of $\overline M_X$.  This is immediate in the first two cases, but requires some justification in the third:  if $\delta$ denotes the length of the edge corresponding to $x$ in $\Gamma$ then $\overline M_{X,x} \simeq (\overline M_{S,\pi(x)} \oplus \overline{\mathbb N} \alpha \oplus \overline{\mathbb N}\beta) / (\alpha + \beta = \delta)$ and the map $\overline M'_{S,\pi(x)} \rightarrow \overline M_{S,\pi(x)}$ carries $\delta'$ to $\delta$.  The elements $\alpha$ and $\beta$ are associated with branches of the node $x$ and therefore we choose that $\alpha'$ and $\beta'$ should be associated with the same branches, there is a unique map $\overline N \rightarrow \overline M_{X,x}$ sending $\alpha'$ to $\alpha$ and $\beta'$ to $\beta$.

Since $\Gamma(U, \overline M_X) \rightarrow \overline M_{X,x}$ is a bijection, we have a map $\overline N \rightarrow \overline M_U$, where we view $\overline N$ as a constant sheaf on $U$ and we write $\overline M_U$ for the restriction of $\overline M_X$ to $U$.  Let $N = M_U \mathbin\times_{\overline M_U} \overline N$.  Then we have a monoid homomorphism $N \rightarrow M_U \rightarrow \mathcal O_U$ and we take $M'_X$ to be the associated logarithmic structure.

To complete the construction of $M'_X$, we only need to see that these constructions are compatible.  But if we restrict $M'_U$ to the complement of the marked points and nodes, we get $\pi^\ast M'_S$, as required.  Finally, it is immediate from Theorem~\ref{thm:log-curve-local} that $(X, M'_X)$ is a logarithmic curve, and it is immediate from the construction that $(X, M'_X)$ induces both $(X, M_X)$ and $\Gamma'$ via the natural maps.
\end{proof}

\subsection{Non-Archimedean tropicalization -- revisited}\label{section_nonArchtrop}

Let $\calX$ be an algebraic stack that is locally of finite type over a trivially valued field $k$. Recall from \cite[Section 5]{Ulirsch_nonArchArtin} that Thuillier's generic fiber functor $(.)^\beth$ (as defined in \cite[Section 1]{Thuillier_toroidal}) extends to a functor that sends $\calX$ to a (strict) non-Archimedean stack $\calX^\beth$. The underlying topological space of $\calX^\beth$ may be constructed as follows:

A \emph{point} of $\calX^\beth$ is a morphism $\Spec R\rightarrow \calX$ from a valuation ring $R$ (of rank one) extending the base field $k$. One may think of $R$ as being, e.g., the ring $k[[t]]$ of formal power series, the ring $k[[t^\Q]]$ of Puiseux series, or the ring $k[[t^\R]]$ of Hahn series. Two points $\Spec R\rightarrow \calX$ and $\Spec R'\rightarrow \calX$ are said to be equivalent if there is a valuation ring $R''$ extending both $R$ and $R'$ that makes the diagram
\begin{equation*}\begin{CD}
\Spec R'' @>>> \Spec R' \\
@VVV @VVV\\
\Spec R @>>>  \calX 
\end{CD}\end{equation*}
commute in the sense of $2$-categories. The underlying topological space $\big\vert \calX^\beth\big\vert$ is the set of equivalence classes of points of $\calX^\beth$. The association $\calX\mapsto \big\vert \calX^\beth\big\vert$ is functorial and, whenever there is a smooth and surjective morphism $U\rightarrow \calX$ that is locally of finite type, the induced map $\big\vert U^\beth\big\vert\rightarrow \big\vert\calX^\beth\big\vert$ is a topological quotient map.

We  conclude this section with the proof of Corollary \ref{cor_trop=anal} from the introduction.

\begin{proof}[Proof of Corollary \ref{cor_trop=anal}]
Let $\Spec R\rightarrow a^\ast\calM_{g,n}^{trop}$ be a morphism from a valuation ring (endowed with its divisorial logarithmic structure as in Section \ref{sec:valuation} above) and suppose  that the morphism is logarithmic. This map corresponds to a family of tropical curves over $\Spec R$. The tropical curve in the generic fiber is trivial, and the edge lengths of the tropical curve $\Gamma$ in the special fiber are in the multiplicative monoid $R/R^\ast$. 

Given a non-Archimedean extension $R'$ of $R$, the induced map $\Spec R'\rightarrow a^\ast \calM_{g,n}^{trop}$ is also logarithmic and gives rise to a tropical curve $\Gamma'$ that agrees with $\Gamma$ on the level of underlying graphs, but its edge lengths formally lie in the multiplicative monoid $R'/(R')^\ast$. Since $R'$ extends the valuation on $R$, the valuations of the edge lengths on $\Gamma$ and $\Gamma'$ are equal as elements in $\R_{\geq 0}$. So we obtain a well-defined element $[\Gamma]$ in the set-theoretic moduli space $M_{g,n}^{trop}$ of stable $n$-marked tropical curves of genus $g$. 

Conversely, given an element $[\Gamma]$ in $M_{g,n}^{trop}$ we may always find a family of tropical curves over a valuation ring $R$ whose special fiber is $\Gamma$. Finally, suppose we are given two families of tropical curves $\Gamma$ and $\Gamma'$ over valuation rings $R'$ and $R$, whose classes $[\Gamma]$ and $[\Gamma']$ in $M_{g,n}^{trop}$ are equal. Let $R''$ be the valuation ring of a common non-Archimedean extension of the quotient fields of both $R$ and $R'$.  %$R''=k[[t^\R]]$ and notice that there are canonical embeddings $R,R'\hookrightarrow R''$ making $R''$ into a non-Archimedean extension of both $R$ and $R'$. 
Since $[\Gamma]=[\Gamma']$ in $M_{g,n}^{trop}$, the base changes of $\Gamma$ and $\Gamma'$ to $R''$ (with edge lengths in $R''/(R'')^\ast$) are isomorphic and thus $\Gamma$ and $\Gamma'$ represent the same point in $a^\ast \calM_{g,n}^{trop}$.

This shows that there is a natural injection from  $M_{g,n}^{trop}$ into a the open and dense subset of $a^\ast \calM_{g,n}^{trop}$, whose points can be represented by a logarithmic morphism $\Spec R\rightarrow a^\ast \calM_{g,n}^{trop}$. Let us now show that the diagram  
\begin{equation*}\begin{CD}
\calM_{g,n}^{an}@>\trop_{g,n}^{\ACP}>>M_{g,n}^{trop}\\
@V\subseteq VV @VV\subseteq V\\
\calMbar_{g,n}^\beth @>\trop_{g,n}^\beth>>  \big(a^\ast \calM_{g,n}^{trop}\big)^\beth
\end{CD}\end{equation*}
commutes, i.e., that the restriction of the analytic map  $\trop_{g,n}^\beth$ to $\calM_{g,n}^{an}$ is nothing but the non-Archimedean tropicalization map from \cite{AbramovichCaporasoPayne_tropicalmoduli}. A point $x$ in $\calM_{g,n}^{an}$ is represented by a morphism $\Spec K\rightarrow \calM_{g,n}$ from a non-Archimedean extension $K$ of $k$ and, since  $\calMbar_{g,n}$ is proper, we may extend this to a logarithmic map $\Spec R\rightarrow\calMbar_{g,n}$ from the valuation ring $R$ of a finite extension of $K$. It is a consequence of the description of  $\trop_{g,n}$ in Section \ref{sec:log-curves-defs}, i.e., of Theorem \ref{thm_tropicalization}, that the image of this point in $(a^\ast\calM_{g,n}^{trop}\big)^\beth$ under $\trop_{g,n}^\beth$ is precisely the tropicalization of $x$ in the sense of \cite{AbramovichCaporasoPayne_tropicalmoduli}. 

Finally, since $\trop_{g,n}$ is smooth and surjective, it induces the quotient topology on $\big(a^\ast \calM_{g,n}^{trop}\big)^\beth$. Moreover, as the non-Archimedean tropicalization map $\trop_{g,n}^{\ACP}$ can be identified with a deformation retraction onto the non-Archimedean skeleton of $\calM_{g,n}^{an}$ by \cite[Theorem 1.2.1]{AbramovichCaporasoPayne_tropicalmoduli}, it is also a topological quotient map and therefore the injection between $M_{g,n}^{trop}\hookrightarrow\big(a^\ast \calM_{g,n}^{trop}\big)^\beth$ is continuous. 
\end{proof}

\section{Tautological morphisms in logarithmic geometry}
\label{sec:tropicalization}

\subsection{Logarithmic clutching}\label{section_clutching}

In \cite{Knudsen_projectivityII}, Knudsen introduced clutching maps whose images define the boundary stratification in the moduli spaces of curves:  given $g= g_1+g_2$ and $[n] = I_1 \cup I_2$  a partition of the index set in two disjoint subsets with $2g_i-2+|I_i|\geq 0$, we have 
\begin{gather*}
\delta_{g_1,I_1}: \calMbar_{g_1, I_1 \cup \{\star\}} \times \calMbar_{g_2, I_2 \cup \{\bullet\}} \rightarrow \calMbar_{g, n} \\
\delta_{irr}:\calMbar_{g, [n] \cup \{\star,\bullet \}} \rightarrow \calMbar_{g + 1,n} \ .
\end{gather*}
These morphisms do not extend to morphisms of logarithmic schemes, since the interior of the domain maps into the boundary of the codomain, and there are no morphisms from a scheme with trivial logarithmic structure to one with an everywhere nontrivial logarithmic structure.  In the logarithmic category, we  instead have correspondences:
\begin{equation} \xymatrix@!C=25pt{
& Q \ar[dr] \ar[dl] \\
\mathcal M^{\log}_{g_1, I_1 \cup \{\star\}} \times \mathcal M^{\log}_{g_2, I_2 \cup \{\bullet\}} & & \mathcal M^{\log}_{g,n}
} \qquad \xymatrix@!C=25pt{
& R \ar[dr] \ar[dl] \\
\mathcal M^{\log}_{g, [n] \cup \{\star,\bullet \}} & & \mathcal M^{\log}_{g+1,n}
}\label{logclutch}\end{equation}
At issue here is that, because logarithmic geometry synthesizes algebraic and tropical geometry, a logarithmic clutching map must simultaneously induce an algebraic clutching map and a tropical one.  We have seen in Section~\ref{section_clutching} that tropical clutching requires the specification of the length of the newly minted edge, so logarithmic clutching  requires an extra tropical parameter.  This parameter is manifested in an extra generator in the characteristic monoids $\overline M_Q$ and $\overline M_R$.  

However, in order for the morphisms $Q \to \calM_{g,n}^{\log}$ and $R \to \calM_{g+1,n}^{\log}$ to be well-defined, the logarithmic structures $M_Q$ and $M_R$ cannot be trivial extensions of $\calM_{g_1,I_1 \cup \{ \star \}}^{\log} \times \calM_{g_2,I_2\cup\{\bullet\}}^{\log}$ and $\calM_{g,[n]\cup\{\star,\bullet\}}^{\log}$, respectively.  
Every element of the characteristic monoid of a logarithmic scheme induces a line bundle and cosection on that scheme.  In the case of $Q$ and $R$, the extra parameter is the smoothing parameter of the new node, and therefore the associated line bundle is the conormal bundle $\mathbb L_\star \otimes \mathbb L_\bullet$ (here $\mathbb L_\star$ and $\mathbb L_\bullet$ denote the cotangent line bundles at the marked points) to its image. 

We define $Q$ and $R$ by pulling back the logarithmic structure from $\mathcal M^{\log}_{g,n}$ and $\mathcal M^{\log}_{g+1,n}$, respectively, via the classical (non-logarithmic) clutching maps $\delta_{g_1, I_1}$ and $\delta_{irr}$.
%$\overline{\mathcal M}_{g,n+1} \times \overline{\mathcal M}_{h,m+1} \rightarrow \overline{\mathcal M}_{g+h,n+m}$ and $\overline{\mathcal M}_{g,n+2} \rightarrow \overline{\mathcal M}_{g+1,n}$. 

\begin{remark}
At present, we do not have a good modular description of $Q$ or $R$.  In principle, it should be possible to describe them as moduli spaces of logarithmic curves with a fixed tropicalization, but it is unclear how to describe the scheme structure on such a space in a modular way.
\end{remark}

%Since the image of a clutching morphism has one more node than does the domain, the characteristic monoid of each of $Q$ and $R$ has one more generator than does $\calM_{g_1,I_1 \cup \{ \star \}}^{\log} \times \calM_{g_2,I_2\cup\{\bullet\}}^{\log}$ or $\calM_{g,[n]\cup\{\star,\bullet\}}^{\log}$.  Every element of the characteristic monoid of a logarithmic scheme induces a line bundle and cosection on that scheme.  In the case of $Q$ and $R$, the extra parameter is the smoothing parameter of the new node, and therefore the associated line bundle is the normal bundle $\mathbb L_\star \otimes \mathbb L_\bullet$ (where $\mathbb L_\star$ and $\mathbb L_\bullet$ are the cotangent lines) to its image. \renzo{Isn't smoothing bundle $\mathbb L^\ast_\star \otimes \mathbb L^\ast_\bullet$?}

By construction, the maps $Q \rightarrow \mathcal M^{\log}_{g_1, I_1 \cup \{\star\}} \times \mathcal M^{\log}_{g_2, I_2 \cup \{\bullet\}}$ and $R \rightarrow \mathcal M_{g,[n]\cup \{\star,\bullet\}}^{\log}$, which we  need to construct, are isomorphisms on the underlying algebraic stacks, but they are not strict.  We construct just the map $Q \rightarrow \mathcal M^{\log}_{g_1, I_1 \cup \{\star\}}$, since the remaining constructions are analogous.

A logarithmic morphism $S \rightarrow Q$ composes with $Q \rightarrow \mathcal M^{\log}_{g,n}$ to give a logarithmic curve $X$ over $S$.  Since the underlying algebraic stack of $Q$ is $\calMbar_{g_1, I_1 \cup \{\star\}} \times \calMbar_{g_2, I_2 \cup \{\bullet\}}$, the underlying curve $\underline X$ of $X$ comes with a splitting into a curve $\underline Y$ of genus $g_1$ (corresponding to the composition $\underline S \rightarrow \underline Q \rightarrow \overline {\mathcal M}_{g_1, I_1 \cup \{\star\}}$) and a curve of genus $g_2$.

We have a closed embedding $\underline Y \subseteq \underline X$ and by pullback, $\underline Y$ inherits a logarithmic structure from $X$.  However, this does not make $\underline Y$ into a logarithmic curve over $S$ because it is not logarithmically smooth at the preimage of the node.  To fix this, we modify the logarithmic structure.

Let $y \in \underline Y$ be the preimage of the node of $\underline X$ and let $s$ be its image in $S$.  By the local structure of a logarithmic scheme $\overline M_{X,y}$ is generated by $\overline M_{S,s}$ and two elements $\alpha$ and $\beta$ whose sum is an element of $\overline M_{S,s}$.  Each of $\alpha$ and $\beta$ corresponds to a $\mathcal O_{X,y}^\ast$-coset of $M_{X,y}$ and then generates an ideal in $\mathcal O_{X,y}$.  We choose the labelling of $\alpha$ and $\beta$ so that $\alpha$ is the one whose image under $M_{X,y} \rightarrow \mathcal O_{X,y} \rightarrow \mathcal O_{Y,y}$ generates the ideal of $y$.  Finally, we define $\overline M_{Y,y}$ to be the submonoid of $\overline M_{X,y}$ generated by $\overline M_{S,s}$ and $\alpha$.  The preimage of $\overline M_{Y,y}$ in $M_{X,y} \oplus_{\mathcal O_{X,y}^\ast} \mathcal O_{Y,y}^\ast$ determines a logarithmic structure $M_{Y,y}$ on $\underline Y$ near $y$.  We define $Y$ to be $\underline Y$, equipped with the logarithmic structure $M_{Y,y}$ near $y$ and with the pullback of $M_X$ away from $y$. 

\begin{lemma}
$Y$ is a logarithmic curve over $S$.
\end{lemma}
\begin{proof}
We know that $Y$ is connected and proper, so we only need to verify that it has the right local structure.  This is immediate away from $y$, since $X$ is a logarithmic curve over $S$.  At $y$, the characteristic monoid $\overline M_{Y,y}$ is the submonoid of $\overline M_{S,s} \oplus \mathbb N \alpha \oplus \mathbb N \beta / (\alpha + \beta = \delta)$ generated by $\overline M_{S,s}$ and $\alpha$.  But this is simply $\overline M_{S,s} \oplus \mathbb N \alpha$, so $y$ is indeed a marked point of a logarithmic curve.
\end{proof}

By the universal property of $\mathcal M^{\log}_{g_1, I_1 \cup \{\star\}}$, the logarithmic curve $Y$ induces a morphism $S \rightarrow \mathcal M^{\log}_{g_1, I_1 \cup \{\star\}}$.  Our construction is evidently functorial in logarithmic schemes over $Q$, so it induces the required morphism $Q \rightarrow \mathcal M^{\log}_{g_1, I_1 \cup \{\star\}}$.
\begin{definition}
We define the \emph{logarithmic clutching morphisms} to be the morphisms
\begin{gather*}
\delta^{log}_{g_1,I_1}: Q\longrightarrow \calM^{log}_{g, n} \\
\delta^{log}_{irr}: R \longrightarrow \calM^{log}_{g + 1,n} \ .
\end{gather*}
introduced in diagram \eqref{logclutch}.
\end{definition}

The next theorem shows that the logarithmic clutching morphisms commute with the tropicalization morphism introduced in Corollary \ref{cor:ext} and Theorem \ref{thm:strictsmooth}.
\begin{theorem}\label{thm_clutching}
\begin{enumerate}[(i)]
\item Let $\delta^{log}_{g_1,I_1}$ be the logarithmic clutching morphism, and $\delta^{\trop}_{g_1,I_1}$ the tropical clutching morphism for the same discrete data. There exists a tropicalization morphism
\begin{equation*}
\trop\colon Q \longrightarrow  a^\ast \big(\calM^{\trop}_{g_1, I_1 \cup \{\star\}} \times \calM^{\trop}_{g_2, I_2 \cup \{\bullet\}}  \times \mathbb R_{\geq 0}\big)
\end{equation*}
such that diagram~\eqref{eqn:Q-trop},  commutes:
\begin{equation}\label{eqn:Q-trop}\vcenter{
\xymatrix@C=4em{
\mathcal M^{\log}_{g_1, I_1 \cup \{ \star \}} \times \mathcal M^{\log}_{g_2, I_2 \cup \{ \bullet \}} \ar[d]_{\trop} & \ar[l] Q \ar[r]^{\delta^{log}_{g_1,I_1}} \ar[d]_{\trop} & \calM^{log}_{g, n} \ar[d]^{\trop}\\
a^\ast \bigl( \mathcal M^{\trop}_{g_1, I_1 \cup \{ \star \}} \times \mathcal M^{\trop}_{g_2,I_2 \cup \{\bullet\}} \bigr) & \ar[l] a^\ast \big(\calM^{\trop}_{g_1, I_1 \cup \{\star\}} \times \calM^{\trop}_{g_2, I_2 \cup \{\bullet\}}  \times \mathbb R_{\geq 0}\big) \ar[r]^-{a^\ast\delta^{\trop}_{g_1,I_1}} & a^\ast \calM^{\trop}_{g, n} 
}}
\end{equation}
\item Let $\delta_{irr}^{log}$ and $\delta_{irr}^{\trop}$ be the logarithmic and tropical self-gluing maps. There is a tropicalization morphism
\begin{equation*}
\trop\colon R\longrightarrow a^\ast\big(\calM_{g,[n]\cup\{\star,\bullet\}}^{\trop}\times \R_{\geq 0}\big)
\end{equation*}
that makes the analogous diagram commutative:
\begin{equation*} 
\xymatrix@C=4em{ 
\mathcal M_{g,[n]\cup\{\star,\bullet\}}^{\log} \ar[d]_{\trop} & \ar[l] R \ar[rr]^{\delta^{log}_{irr}} \ar[d]_{\trop} & & \calM^{log}_{g+1, n} \ar[d]^{\trop}\\
a^\ast \bigl( \mathcal M_{g,[n]\cup\{\star,\bullet\}}^{\trop} \bigr) & \ar[l] a^\ast \big(\calM^{\trop}_{g, [n] \cup \{\star,\bullet\}}   \times \mathbb R_{\geq 0}\big) \ar[rr]^{\hspace{1cm}a^\ast\delta^{\trop}_{irr}}& & a^\ast \calM^{\trop}_{g+1, n} 
}
\end{equation*}
\end{enumerate}
\label{thm:clutch}
\end{theorem}
\begin{proof}
We only consider $Q$ and leave the analogous proof for $R$ to the avid reader. First we construct the map
\begin{equation*}
\trop_{g,n}\colon Q \rightarrow a^\ast \mathcal M^{\trop}_{g_1, I_1 \cup \{\star\}}\times a^\ast \mathcal M^{\trop}_{g_2, I_2 \cup \{\bullet\}} \times a^\ast \mathbb R_{\geq 0} = a^\ast (\mathcal M^{\trop}_{g_1, I_1 \cup \{\star\}} \times \mathcal M^{\trop}_{g_2, I_2 \cup \{\bullet\}} \times \mathbb R_{\geq 0}) \ .
\end{equation*}
On the first two factors, we  use the composition
\begin{equation*}
Q \rightarrow \mathcal M^{\log}_{g_1, I_1 \cup \{\star\}} \times \mathcal M^{\log}_{g_2, I_2 \cup \{\bullet\}} \rightarrow a^\ast \mathcal M^{\trop}_{g_1, I_1 \cup \{\star\}} \times a^\ast\mathcal M^{\trop}_{g_2, I_2 \cup \{\bullet\}} .
\end{equation*}
For the final factor, note that by definition, a morphism from $Q$ into $a^\ast \mathbb R_{\geq 0}$ corresponds to a global section of $\overline M_Q$.  We choose the smoothing parameter for the node $y$ analyzed above.  Note that this is a well-defined global section of $\overline M_Q$, as the node $y$ is well-defined globally in $Q$. %This might fail to be the case over the image of $\underline Q \rightarrow \overline{\mathcal M}_{g,n}$. 

Now we observe that diagram~\eqref{eqn:Q-trop} commutes.
%\begin{equation*} \xymatrix{
%\mathcal M^{\log}_{g_1, I_1 \cup \{\star\}} \times \mathcal M^{\log}_{g_2, I_2 \cup \{\bullet\}} \ar[d] & Q \ar[l] \ar[r] \ar[d] & \mathcal M^{\log}_{g,n} \ar[d] \\
%a^\ast \mathcal M^{\trop}_{g_1, I_1 \cup \{\star\}} \times a^\ast \mathcal M^{\trop}_{g_2, I_2 \cup \{\bullet\}} & a^\ast \mathcal M^{\trop}_{g_1, I_1 \cup \{\star\}} \times a^\ast \mathcal M^{\trop}_{g_2, I_2 \cup \{\bullet\}} \times a^\ast \mathbb R_{\geq 0} \ar[l] \ar[r] & a^\ast \mathcal M^{\trop}_{g,n}
%} \end{equation*}
This holds by definition for the square on the left. To prove the commutativity of the square on the right, it is sufficient to consider geometric points.  If $s$ is a geometric point of $Q$ corresponding to a logarithmic curve $X$ by the upper horizontal arrow, then the image of $s$ in $a^\ast M^{\trop}_{g,n}$ is the dual graph of $X$.  Going around the other way, $s$ maps to the pair of dual graphs $\Gamma'$ and $\Gamma''$ obtained by deleting the edge corresponding to the distinguished node of $X$ and replacing it in each factor with a leg.  The last factor of the center vertical map records the length $\delta$ of the edge that was deleted.  The horizontal arrow on the lower right then assembles a graph $\Gamma$ by deleting the legs of $\Gamma'$ and $\Gamma''$ that were just added and replaces them with an edge of length $\delta$.  This is clearly the dual graph of $X$, so the diagram commutes.
\end{proof}

\subsection{Puncturing logarithmic curves}
\label{sec:puncturing}

In the classical theory of $\Mbar_{g,n}$, there is a canonical way to assign an $(n+1)$-pointed stable curve to the choice of a point on an $n$-pointed curve. In this section we show an analogous construction for logarithmic curves. We refer to this process as \emph{puncturing}, because that is the effect it has on the \emph{Kato--Nakayama space} of the logarithmic curve (as constructed in \cite{KatoNakayama_rounding}).

\begin{lemma} \label{lem:puncture}
Suppose that $\pi : X \rightarrow S$ is a logarithmic curve over $S$ and that $x : S \rightarrow X$ is a section (in the category of logarithmic schemes).  There exists a universal logarithmic modification $Y \rightarrow X$ such that $x$ factors through a section $y$ of $Y$ and the image of $y$ does not meet any of the nodes or marked points of $Y$.
\end{lemma}
\begin{proof}
We construct $Y$ as a subfunctor of $X$.  We give the construction locally near each geometric point of $X$ and then sheafify the result.  Over the complement of the image of $x$, we take $Y$ to coincide with $X$.  Near the image of $x$, we have the usual three possibilities, the latter two of which are illustrated in Figures~\ref{fig:through-mark} and~\ref{fig:through-node}:
\begin{enumerate}[(i)]
\item If $x(s)$ is a smooth point of $X_s$ then $x(t)$ is also in the smooth locus of $X_t$ for all $t$ in an \'etale neighborhood of $s$.  We take $Y = X$ in this neighborhood.
\item If $x(s)$ is a marked point of $X_s$ then $x(t)$ is either the same marked point or a smooth point for all $t$ in an \'etale neighborhood $U$ of $s$.  We have a canonical identification $\overline M_{X,x(s)} \simeq \overline M_{S,s} \oplus \mathbb N \upsilon$, and the section $\upsilon$ extends to an \'etale neighborhood of $x(s)$ inside $\pi^{-1} U$.  On this neighborhood, we take $Y$ to be the subfunctor of $X$ where the sections $\pi^\ast s^\ast \upsilon$ and $\upsilon$ are comparable (i.e., either $\pi^\ast s^\ast \upsilon - \upsilon \in \overline M_Y$ or $\upsilon - \pi^\ast s^\ast \upsilon \in \overline M_Y$).  We note that away from this marked point, $\upsilon$ restricts to $0$, so that $\pi^\ast s^\ast \upsilon$ and $\upsilon$ are already comparable in $X$, and therefore $Y$ agrees with $X$ away from the marked point.
\item If $x(s)$ is a node of $X_s$ then there is an \'etale neighborhood $U$ of $s$ in $S$ such that $x(t)$ is the same node of $X_t$ or is in the smooth locus of $X_t$.  We have an identification $\overline M_{X,x(s)} \simeq (\overline M_{S,s} \oplus \mathbb N\alpha \oplus \mathbb N\beta) / (\alpha + \beta = \delta)$ for some $\delta \in \overline M_{S,s}$, and this identification is unique up to exchange of $\alpha$ and $\beta$.  The sections $\alpha$ and $\beta$ extend to an \'etale neighborhood of $x(s)$ in $\pi^{-1} U$.  On this neighborhood, we take $Y$ to be the subfunctor of $X$ where $\alpha$ is comparable to $\pi^\ast s^\ast \alpha$ and $\beta$ is comparable to $\pi^\ast s^\ast \beta$.%
\footnote{In fact, $\alpha \leq \pi^\ast s^\ast \alpha$ if and only if $\beta \geq \pi^\ast s^\ast \beta$, and $\alpha \geq \pi^\ast s^\ast \alpha$ if and only if $\beta \leq \pi^\ast s^\ast \beta$, so that each of these conditions is equivalent to the other.  We include both so that the construction is manifestly unchanged by exchange of $\alpha$ and $\beta$.}  %
Now, note that if $x'$ is a smooth point of the same \'etale neighborhood then one of $\alpha$ or $\beta$ restricts to $0$ and the other restricts to $\delta$.  But the elements $\pi^\ast s^\ast \alpha$ and $\pi^\ast s^\ast \beta$ restrict to $s^\ast \alpha$ and $s^\ast \beta$ under the identification $\overline M_{X,x'} \simeq \overline M_{S,\pi(x')}$, and as $s^\ast \alpha + s^\ast \beta = s^\ast \delta = \delta$, we determine that the restrictions of $\pi^\ast s^\ast \alpha$ and $\pi^\ast s^\ast \beta$ are already comparable to $\alpha$ and $\beta$, respectively, at $x'$.  That is $Y$ agrees with $X$ away from the node.
\end{enumerate}

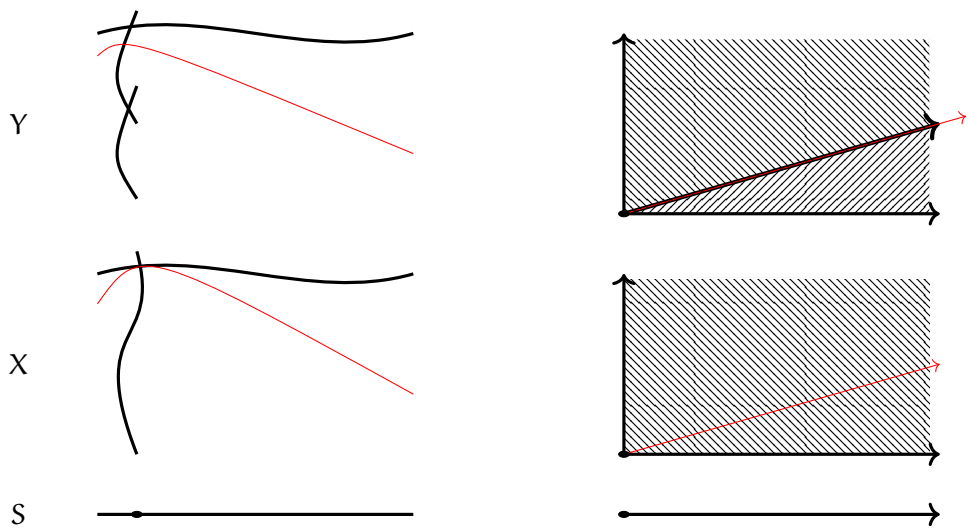
\begin{figure}
\begin{center}
\begin{tikzpicture}[xscale=.7, yscale=.4]
\draw[very thick, name path=mark2] (0,14) .. controls (2,15) and (4,13) .. (6,14);
\draw[very thick, name path=exceptional] (.75,14.75) .. controls (.25, 12.5) .. (.75,11);
\draw[very thick, name path=exceptional] (.75,12.25) .. controls (.25, 9.875) .. (.75,8.5);
%\draw[red] plot [smooth, tension=1] coordinates { (0,12.75) (3,12.5) (6,10) };
\draw[red] (0,13.25) .. controls (.5,14) .. (6,10);

\node at (-1.5,11) {$Y$};

\draw[very thick, name path=mark] (0,6) .. controls (2,7) and (4,5) .. (6,6);
\draw[very thick, name path=fiber] (.75,6.75) .. controls (1.25,3.5) and (-.25,4.5) .. (.75,0);

\node at (-1.5,3) {$X$};

\draw[very thick] (0,-2) -- (6,-2);
\draw[fill=black] (.75,-2) circle (.1);
%\node at (.75,-2.5) {$s$};
\node at (-1.5,-2) {$S$};

\path [name intersections={of=mark and fiber,by=X}];

\draw[red] (0,5) .. controls ($(X)+(0,.75)$)  .. (6,2);
%\node at ($(X)+(.6,.6)$) {$x(s)$};

\begin{scope}[shift={(10,0)}]
\draw[very thick,->] (0,8) -- (0,14);
\draw[very thick,->] (0,8) -- (6,8);
\draw[ultra thick,->] (0,8) -- (6,11);
\draw[->,red] (0,8) -- (6.5,11.25);

\fill[pattern=north west lines] (0,8) -- (5.8,10.9) -- (5.8,13.8) -- (0,13.8);
\fill[pattern=north east lines] (0,8) -- (5.8,10.9) -- (5.8,8);

\draw[very thick,->] (0,0) -- (6,0);
\draw[very thick,->] (0,0) -- (0,6);

\fill[pattern=north west lines] (0,0) -- (5.8,0) -- (5.8,5.8) -- (0,5.8);

\draw[red,->] (0,0) -- (6,3);

\draw[very thick,->] (0,-2) -- (6,-2);

\draw[fill=black] (0,0) circle (.1);
\draw[fill=black] (0,-2) circle (.1);
\draw[fill=black] (0,8) circle (.1);
\end{scope}
\end{tikzpicture}
\end{center}
\caption{On the left is a logarithmic curve $X$ over $S$ with a section passing through a marked point, along with the logarithmic modification $Y$ of $X$ where the section passes through the smooth locus.  On the right are the tropicalizations.}
\label{fig:through-mark}
\end{figure}

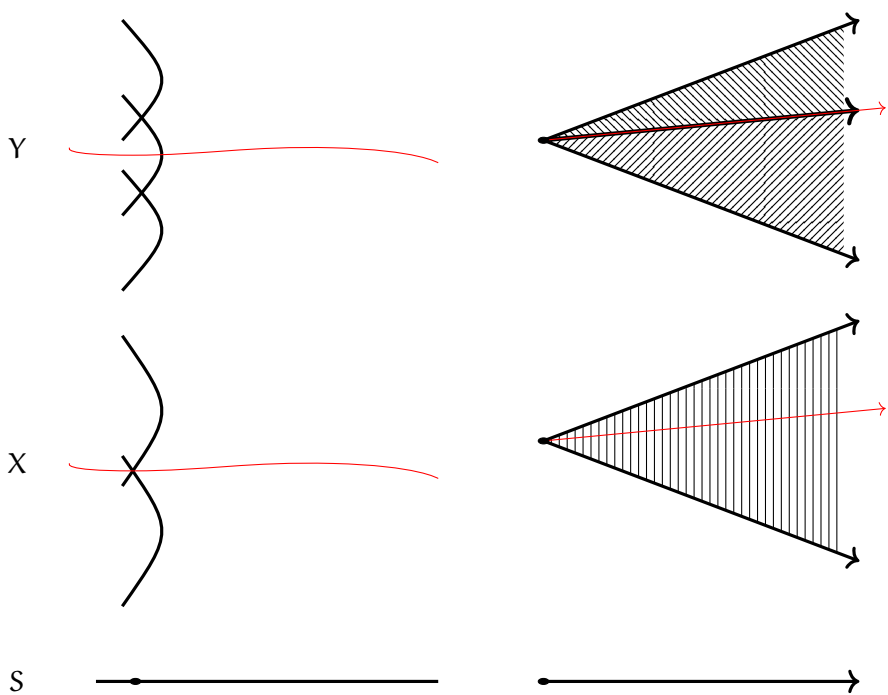
\begin{figure}
\begin{center}
\begin{tikzpicture}[xscale=.7, yscale=.4]
\draw[very thick,name path=A] (0,5.5) .. controls (1,3) .. (0,.5);
\draw[very thick,name path=B] (0,1.5) .. controls (1,-1) .. (0,-3.5);
\path [name intersections={of=A and B,by=X}];
\draw[red] plot [smooth,tension=1] coordinates { (-1,1.25) (X) (4,1.25) (6,.75) };
\node at (-2,1.25) {$X$};

\draw[very thick] (0,16) .. controls (1,14) .. (0,12);
\draw[very thick] (0,13.5) .. controls (1,11.5) .. (0,9.5);
\draw[very thick] (0,11) .. controls (1,9) .. (0,7);
% \path [name intersections={of=A and B,by=X}];
\draw[red] plot [smooth,tension=1] coordinates { (-1,11.75) ($(X)+(0,10.5)$) (4,11.75) (6,11.25) };
\node at (-2, 11.75) {$Y$};

\node at (-2,-6) {$S$};
\draw[very thick] (-.5,-6) -- (6,-6);
\draw[fill=black] (.25,-6) circle (.1);
%\node at (.25,-6.5) {$s$};

\begin{scope}[shift={(8,2)}]
\draw[very thick,->] (0,-8) -- (6,-8);
\draw[fill=black] (0,-8) circle (.1);

\draw[very thick,->] (0,0) -- (6,4);
\draw[very thick,->] (0,0) -- (6,-4);

\draw[very thick,->] (0,10) -- (6,14);
\draw[very thick,->] (0,10) -- (6,6);
\draw[ultra thick,->] (0,10) -- (6,11);

\fill[pattern=vertical lines] (0,0) -- (5.7,3.8) -- (5.7,-3.8);
\fill[pattern=north west lines] (0,10) -- (5.7,13.8) -- (5.7,11);
\fill[pattern=north east lines] (0,10) -- (5.7,6.2) -- (5.7,11);

\draw[red,->] (0,0) -- ($(6,1) + 1/12*(6,1)$);
\draw[red,->] (0,10) -- ($(6,11) + 1/12*(6,1)$);

\draw[fill=black] (0,0) circle (.1);
\draw[fill=black] (0,10) circle (.1);
\end{scope}
\end{tikzpicture}
\end{center}
\caption{The figure on the left shows a logarithmic curve $X$ over $S$ with a section passing through a node in red.  Above $X$ is the logarithmic modification $Y$ in which the unique lift of the section factors through the smooth locus.  On the right are the tropicalizations.}
 \label{fig:through-node}
\end{figure}

Since these modifications are all locally pulled back from modifications of families of toric curves, it is immediate that $Y$ is representable by a family of logarithmic curves over $S$.  It remains to lift $x$ to a section of $Y$ and verify it passes through the smooth locus.

Each of the local considerations above involves the imposition of order between local sections $\pi^\ast x^\ast \gamma$ and $\gamma$ of $\overline M_{X}$.  As $x^\ast \pi^\ast x^\ast \gamma = x^\ast \gamma$ and $x^\ast \gamma$ are certainly comparable in $\overline M_S$, the section $x : S \to X$ factors uniquely through a section $y : S \to Y$.  In fact, $y$ factors through the locus where $\pi^\ast x^\ast \gamma$ and $\gamma$ are equal; as $\gamma$ generates $\overline M_X^{\rm gp}$, and therefore also $\overline M_Y^{\rm gp}$, over $\overline M_S^{\rm gp}$ in this neighborhood, we conclude that $\overline M_Y^{\rm gp} = \pi^\ast \overline M_S^{\rm gp}$ near the image of $Y$.  That is, $y$ factors through the smooth locus of $Y$ over $S$.
\end{proof}

Using the lemma, we may puncture $Y$:  let $j : Y_0 \rightarrow Y$ be the inclusion of the complement of the section $y$, and let  $M_Y = \mathcal O_Y \mathbin\times_{j_\ast \mathcal O_{Y_0}} \mathcal O_{Y_0}^\ast$ be the logarithmic structure associated to the divisor $y(S)$ in the total space of $Y$.

\begin{remark}
The section $y$ is not a morphism of logarithmic schemes from $S$ to $Y$ when $Y$ is equipped with the logarithmic structure described above.  It can only be promoted to one if we replace the logarithmic structure of $S$ with $y^\ast M_Y$.  Tropically, this corresponds to the choice of a distance along the leg of the dual tropical curve that corresponds to $y$.
\end{remark}

\begin{remark}
The whole process can be repeated if we have more than one section of $X$, with one crucial difference that has the effect of modifying the base.  Suppose that $x_1$ and $x_2$ are two sections of $X$ over $S$.  We apply Lemma~\ref{lem:puncture} to obtain $Y \to X$ and a lift of $x_1$ to a logarithmic section $y_1$ of $Y$ over $S$ that does not meet any nodes or marked points of $Y$.  However, $x_2$ does not automatically lift to $Y$.  Indeed, $Y$ was constructed (locally) as the universal object in which $\pi^\ast x_1^\ast \alpha$ and $\alpha$ are comparable.  Thus $x_2 : S \to X$ will factor through $Y$ if, and only if, $x_2^\ast \pi^\ast x_1^\ast \alpha = x_1^\ast \alpha$ and $x_2^\ast \alpha$ are comparable.

In general, $x_1^\ast \alpha$ and $x_2^\ast \alpha$ need not be comparable in $S$, but there is a universal logarihtmic modification $\tilde S \to S$ in which they are.  Writing $\tilde Y$ for the base change of $Y$ to $\tilde S$, we can then apply Lemma~\ref{lem:puncture} to the section $y_2$ induced from $x_2$ and find a universal $Z \to \tilde Y$ with sections $z_1$ and $z_2$ lifting both $x_1$ and $x_2$.

One could approach this construction another way.  According to Lemma~\ref{lem:puncture}, there are modifications $Y_1$ and $Y_2$ of $X$ such that the section $x_i$ of $X$ lifts to a section $y_i$ of $Y_i$ that lies in its strict locus, relative to $S$.  One may attempt to find a common modification $Z = Y_1 \mathop\times_X Y_2$ of $X$, but $Y_{12}$ is not necessarily flat over $S$ and neither of the $y_i$ necessarily lifts to $Y_{12}$.  However, there is a modification $\tilde S \to S$ (in fact, a universal one) such that the base change $Z$ of $Y_{12}$ is flat over $\tilde S$ and the sections $y_1$ and $y_2$ lift to well-defined sections of $Y_{12}$ over $\tilde S$.  This is an example of (toroidal) flattening~\cite{RaynaudGruson,AbramovichKaru,Molcho}.
\end{remark}

\subsection{Tropicalizing the universal curve}

Let $\mathcal X_{g,n}^{\trop}$ denote the universal curve over $\mathcal M_{g,n}^{\trop}$ and let $\mathcal X_{g,n}^{\log}$ denote the universal curve over $\mathcal M_{g,n}^{\log}$.  In this section, we construct a commutative square:
\begin{equation*} \xymatrix{
\mathcal X_{g,n}^{\log} \ar[r]^-{\trop} \ar[d] & a^\ast \mathcal X_{g,n}^{\trop} \ar[d] \\
\mathcal M_{g,n}^{\log} \ar[r]^-{\trop_{g,n}} & a^\ast \mathcal M_{g,n}^{\trop}
}
\end{equation*}

In other words, for every logarithmic scheme $S$, and every logarithmic curve $X$ over $S$, we must give an $S$-morphism
\begin{equation*}
X \rightarrow S \mathbin\times_{a^\ast \mathcal M_{g,n}^{\trop}} a^\ast \mathcal X_{g,n}^{\trop} .
\end{equation*}
These are stacks, so it is sufficient to make the construction locally in $S$.  We can therefore assume that the map $S \rightarrow a^\ast \mathcal M_{g,n}^{\trop}$ factors through $a^\ast \sigma$ for some strict map $\sigma \rightarrow \mathcal M_{g,n}^{\trop}$.  With this reduction, Theorem~\ref{thm_universalcurve} informs us that we require the upper horizontal arrow of a commutative diagram
\begin{equation*} \xymatrix{
X \ar[r]^-f \ar[d] & a^\ast \Cone(\Gamma) \ar[d] \\
S \ar[r] & a^\ast \sigma
} \end{equation*}
where $\Gamma$ is the tropical curve associated to $\sigma \rightarrow \mathcal M_{g,n}^{\trop}$.  The map $f$ can be constructed \'etale-locally  in $X$, and the local construction is immediate from the definition of $\Cone(\Gamma)$ in Section~\ref{section_families} and the local description of logarithmic curves in and following Theorem~\ref{thm:log-curve-local}.

\subsection{Stabilization and forgetting markings}
\label{sec:forgetting}

Let $S$ be a logarithmic scheme.  Suppose that $2g - 2 + n > 0$ and that $X$ is a family of logarithmic curves over $S$ that has $n$ markings and fibers of genus $g$, but is not necessarily stable.  Then the underlying family of curves $\underline X$ over $\underline S$ has a stabilization $\tau : \underline X \rightarrow \underline Y$ over $\underline S$.  We give $\underline Y$ a logarithmic structure by defining $M_Y = \tau_\ast M_X$.

\begin{lemma} \label{lem:stab}
Let $X$ be a logarithmic curve over $S$, and let $\tau : \underline X \rightarrow \underline Y$ be a morphism that contracts an unstable, connected subcurve of $\underline X$ to a point of $\underline Y$.  Then $M_Y = \tau_\ast M_X$ is a logarithmic structure on $\underline Y$ and it makes $Y = (\underline Y, M_Y)$ into a logarithmic curve over $S$.  The tropicalization of $Y$ is the stabilization of the tropicalization of $X$. 
\end{lemma}
\begin{proof}
First we verify that $M_Y$ is a logarithmic structure on $Y$.  We note that $\tau_\ast \mathcal O_X = \mathcal O_Y$ since $\tau$ is proper with integral fibers and reduced codomain, and that therefore $\tau_\ast \mathcal O_X^\ast = \mathcal O_Y^\ast$.  The map 
\begin{equation*}
M_Y = \tau_\ast M_X \rightarrow \tau_\ast \mathcal O_X = \mathcal O_Y
\end{equation*}
makes $M_Y$ into a prelogarithmic structure.  To show it is a logarithmic structure, we need to verify that $\varepsilon^{-1} \mathcal O_Y^\ast \rightarrow \mathcal O_Y^\ast$ is a bijection.  We have a commutative diagram:
\begin{equation*} \xymatrix{
\varepsilon^{-1} \mathcal O_Y^\ast \ar[r] \ar[d] & \varepsilon^{-1} \tau_\ast \mathcal O_X^\ast \ar[d] \ar[r] &  \tau_\ast \varepsilon^{-1} \mathcal O_X^\ast \ar[dl] \\
\mathcal O_Y^\ast \ar[r] & \tau_\ast \mathcal O_X^\ast
} \end{equation*}
Thus the horizontal arrows on the left side of the diagram are isomorphisms.  The horizontal arrow on the right side of the diagram is an isomorphism because $\tau_\ast$ is left exact.  The diagonal arrow is an isomorphism because $M_X$ is a logarithmic structure, so we deduce that the vertical arrow on the left is an isomorphism, as required.

Now we check that $Y$ is a logarithmic curve.  Since $Y$ is proper over $S$ with connected geometric fibers, we only need to analyze the local structure at a geometric point $y$ of $Y$ lying above a geometric point $s$ of $S$.  If the fiber is a point there is nothing to show.  The other possibilities are:\footnote{It is sufficient, by induction, to assume that the $\tau^{-1} y$ is an irreducible rational curve, which simplifies the subsequent argument somewhat.  However, we feel that the argument, while more complicated, is instructive concerning the relationship between logarithmic geometry and tropical geometry.  We are confident that a reader who prefers the inductive argument can make the necessary substitutions and attendant simplifications.}
\begin{enumerate}[(i)]
\item $y$ is a smooth point of $\underline Y$ and the fiber is a tree of rational curves with no marked points;
\item $y$ is a marked point of $\underline Y$, in which case the fiber is a tree of rational curves with one marked point;
\item $y$ is a node and the fiber is a tree of rational curves with two marked points.
\end{enumerate}
In each case, we can compute using the exact sequence
\begin{equation*}
0 \rightarrow \mathcal O_X^\ast \rightarrow M_X \rightarrow \overline M_X \rightarrow 0 .
\end{equation*}
By \cite[Lemma~B.4]{AbramovichMarcusWise_comparison}, the formation of $\tau_\ast M_X$ commutes with base change.  Restricting to the fiber over $y$ and pushing forward via $\tau$, we get another exact sequence:
\begin{equation*}
0 \rightarrow \mathcal O_{Y,y}^\ast \rightarrow M_{Y,y} \rightarrow \tau_\ast \overline M_{X_y} \rightarrow \Pic(\tau^{-1} y) 
\end{equation*}
We can identify $\Pic(\tau^{-1} y)$ with $\mathbb Z^V$, where $V$ is the set of vertices of the dual graph of $\tau^{-1} y$, since line bundles on trees of rational curves are determined up to isomorphism by their multidegrees.

Let $s$ be the image of $y$ in $S$ and let $\Gamma$ be the dual graph of $\tau^{-1} y$.  By this, we mean that $\Gamma$ has 
\begin{enumerate}[(a)]
\item one vertex for each component of $\tau^{-1} y$;
\item one vertex for each node of $X_s$ joining $\tau^{-1} y$ to the rest of $X_s$ (note that if $\tau^{-1} y$ is joined to a component $Z$ of $X_s$ by two different nodes then we add \emph{two} nodes to $\Gamma$);\footnote{In actual fact, we should have a dual graph with legs of finite length in this case.  Rather than go through the effort necessary to define such a thing, we add vertices at the ends of these finite length edges.}
\item one leg for each marked point of $\tau^{-1} y$;
\item two vertices are connected by an edge if the corresponding components of $X_s$ are connected by a node.
\end{enumerate}
This graph inherits a metric valued in $\overline M_{S,s}$ from the dual graph of $X_s$. 

We can identify $\tau_\ast \overline M_{X_y}$ with the piecewise linear functions on $\Gamma$, valued in $\overline M_{S,s}$, that are linear with integer slopes along the edges.  The map
\begin{equation} \label{eqn:boundary}
\tau_\ast \overline M_{X_y} \rightarrow \Pic(\tau^{-1} y) \simeq \mathbb Z^V
\end{equation}
sends a piecewise linear function to the sum of the outgoing slopes at each vertex.  The kernel therefore consists of the piecewise linear functions that are \emph{linear} on the interior of the dual graph of $\tau^{-1} y$ where, by convention, linearity at a vertex means that the sum of the outgoing slopes is zero, and the interior consists of those vertices that correspond to components of $\tau^{-1} y$.

\begin{enumerate}[(i)]
\item If $y$ is a smooth, unmarked point of $\underline Y$ then $\tau^{-1} y$ has only one external vertex and no legs, so it supports no non-constant linear functions.  Therefore the kernel of~\eqref{eqn:boundary} is $\overline M_{S,s}$.  We deduce that $\overline M_{Y,y} = \overline M_{S,s}$, and therefore that $M_{Y,y}$ is the pullback of $M_{S,s}$ in this case.

\item If $y$ is a marked point, then $\tau^{-1} y$ has one external vertex and one leg.  This means that a piecewise linear function on $\tau^{-1} y$ that is linear on the interior is determined by its value on the external vertex and the outgoing slope on the leg (which must be a non-negative integer).  Thus $\overline M_{Y,y} = \overline M_{S,s} \oplus \mathbb N$.

\item If $y$ is a node then $\tau^{-1} y$ has two external vertices and no legs, so an element of $\overline M_{Y,y}$ is determined uniquely by the two values at the external vertices.  These values are arbitrary, provided their difference is a multiple of the distance, measured along the dual graph of $\tau^{-1} y$, between the two internal vertices.  This allows us to identify the fiber of 
\begin{equation*}
\overline M_{Y,y} = \big\{ (\rho, \sigma) \in \overline M_{S,s} \times \overline M_{S,s} \: \big| \: \rho - \sigma \in \mathbb Z \delta \big\}
\end{equation*}
where $\delta$ is the length of the path connecting the two external vertices of $\tau^{-1} y$.  This is precisely the form (iii$'$) of Section~\ref{sec:log-curves-defs}.
\end{enumerate}

In each of the three cases, we get the appropriate local structure of a logarithmic curve (see Section~\ref{sec:log-curves-defs} above), as required.
\end{proof}

The following theorem is a consequence of Lemma \ref{lem:puncture} and Lemma \ref{lem:stab}. 

\begin{theorem} \label{thm:twodiag}
Denote by $\mathcal X_{g,n}^{\log}$ the universal curve over $\mathcal M_{g,n}^{\log}$. We have the following commutative diagram:
\begin{equation*}
\xymatrix{
\mathcal X_{g,n}^{\log} \ar[r]^-{\sim} \ar[d]_{\trop} & \mathcal M_{g,n+1}^{\log} \ar[r] \ar[d]^{\trop_{g,n+1}} & \mathcal M_{g,n}^{\log} \ar[d]^{\trop} \\
a^\ast \mathcal X_{g,n}^{\trop} \ar[r]^-{\sim} & a^\ast \mathcal M_{g,n+1}^{\trop_{g,n}} \ar[r] & a^\ast \mathcal M_{g,n}^{\trop}
}
\end{equation*}
\end{theorem}
We observe that Theorem \ref{thm:twodiag} contains several statements: the commutativity of the left square shows that in the logarithmic category the universal curve over $\calM_{g,n}^{log}$ may be identified with the forgetful  morphism $\pi^{log}_{g,n+1}$, and that this identification commutes with tropicalization; the commutativity of the right square asserts that the forgetful morphisms commute with tropicalization.
\begin{proof}
The universal curve $\mathcal X_{g,n}^{\log}$ over $\mathcal M_{g,n}^{\log}$ is an example of  a logarithmic curve with a section. The construction following Lemma \ref{lem:puncture} gives a morphism of logarithmic moduli problems:
\begin{equation*}
\mathcal X_{g,n}^{\log} \rightarrow \mathcal M_{g,n+1}^{\log}.
\end{equation*}
The first part of Lemma \ref{lem:stab}  defines the morphism $\pi^{log}_{g,n+1}$, by forgetting the last marked point and stabilizing the result; this allows to view $\calM^{log}_{g,n+1}$ as a curve over $\calM^{log}_{g,n}$ with a section (namely the image in the stabilization of the marked point that is forgotten) and this constructs the inverse morphism 
\begin{equation*}
\mathcal M_{g,n+1}^{\log} \rightarrow \mathcal X_{g,n}^{\log}.
\end{equation*}
The last sentence in Lemma \ref{lem:stab} implies the commutativity of the right square.

For the commutativity of the diagram on the left, it is sufficient to consider geometric points.  We analyze the logarithmic curve $Y$ obtained by puncturing a logarithmic curve $X$ over a geometric point $S$ at a section $x$:
\begin{enumerate}[(i)]
\item If $x$ lies in the smooth locus of the underlying curve of $X$ then puncturing simply replaces the logarithmic structure  of $X$ with one having relative rank~$1$ at $x$.  The tropicalization of $Y$ is therefore obtained from $X$, by attaching a leg to the vertex whose corresponding component of $X$ contains the image of $x$.
\item If $x$ lies at a marked point of the underlying curve of $X$ then the canonical section $\alpha$ of $\overline M_{X,x}$ pulls back under $\overline M_{X,x} \rightarrow \overline M_{S}$ to a section $d$ of $\overline M_S$.  Modifying $X$ so that $\alpha$ is comparable to $d$ subdivides the leg corresponding to $\alpha$ at a distance $d$ from the vertex to which it is attached.  Then we add a leg to the new vertex, as in the previous case.
\item Finally, if $x$ lies at a node of the underlying curve of $X$ then $\overline M_{X,x}$ contains two generators $\alpha_1$ and $\alpha_2$ that sum to an element $d \in \overline M_S$.  These map, under $\overline M_{X,x} \rightarrow \overline M_S$ to elements $d_1$ and $d_2$ of $\overline M_S$ such that $d_1 + d_2 = d$.  Modifying $X$ to make $d_i$ comparable to $\alpha_i$ has the effect of subdividing the edge into two edges, of lengths $d_1$ and $d_2$.  We then add a leg to the new vertex, as in the first case.
\end{enumerate}
In all cases, the tropical curve we get is the same as the one constructed in the proof of Theorem~\ref{thm_universalcurve}, and the proof is complete.
\end{proof}

%%%%%%%%%%%%%%%%%%%%%%%%%%%%%%%%%%%%%%%%%%%%%%%%%%%%%%%%%%%%%%%%%%%%%%%%%%%
\appendix

\section{\'Etale specialization}
\label{sec:etale-specialization}

Let $S$ be a scheme. By definition a \emph{geometric point} of $S$ is a point of the small \'etale site.  Naively, a geometric point is a morphism from the spectrum of an algebraically closed field to $S$.  However, we wish to identify two geometric points $\Spec K \rightarrow S$ and $\Spec L \rightarrow S$ when the latter is induced from a field extension $K \subseteq L$.  This can be done rather efficiently by observing that a field extension $K \subseteq L$ induces an equivalence of \'etale sites $\et(\Spec K) \rightarrow \et(\Spec L)$ by pullback (both are equivalent to the category of sets).  This observation justifies our definition of a geometric point $S$ as a functor $s^\ast : \et(S) \rightarrow \mathbf{Sets}$ that is induced from a morphism $s : \Spec K \rightarrow S$, where $K$ is the spectrum of an algebraically closed field.

The geometric points of $S$ form a category in which the morphisms are natural transformations.  By convention, the direction of these transformations is opposite the usual direction for functors $\et(S) \rightarrow \mathbf{Sets}$.  That is, a \emph{specialization} from a geometric point $t$ of $S$ to a geometric point $s$ of $S$ is a natural transformation $s^\ast \rightarrow t^\ast$.  A specialization of $t$ to $s$ is denoted $t \leadsto s$.

It is worth noting that if $U$ is \'etale over $S$ and $s$ is a geometric point of $S$ then $s^\ast U$ is simply the fiber of $U$ over $s$.  If $t \leadsto s$ and $\varphi : s^\ast \rightarrow t^\ast$ is the corresponding natural transformation then $\varphi_U : s^\ast U \rightarrow t^\ast U$ takes points of the fiber over $s$ to points of the fiber over $t$.  Now suppose that $U$ is an \'etale neighborhood of $s$, so that we have a distinguished $s' \in s^\ast U$.  Then $\varphi_U(s')$ is a distinguished element of $t^\ast U$.  In other words, a specialization from $t$ to $s$ is a choice of point in $t^\ast U$ for every \'etale neighborhood $U$ of $s$.

Passing to the limit over all \'etale neighborhoods of $s$, we obtain the strict henselization $\mathcal O_{S,s}$ of $S$ at $s$.  As we have just seen, a specialization $t \leadsto s$ gives a factorization of $t \rightarrow S$ through each \'etale neighborhood of $s$ in $S$, so it gives a factorization through $\Spec \mathcal O_{S,s}$.  Conversely, it is straightforward to reverse the reasoning and see that each such map gives an \'etale specialization.

Every \'etale specialization $t \leadsto s$ induces a Zariski specialization.  Indeed, if $t$ lifts to every \'etale neighborhood of $s$ then the image of $t$ in $S$ is contained in every open neighborhood of $s$, since \'etale maps are open.  However, a geometric point can specialize to another in more than one way.  

\begin{example}
For example, let $S$ be an irreducible nodal curve.  In the Zariski topology there is just one specialization 
from the generic point of $S$ to the node.
However, let $s$ be a geometric point of $S$ situated at the node and let $t$ be a geometric point situated at the generic point.  If $U$ is any sufficiently small \'etale neighborhood of $s$ then $U$ is reducible, with two points above $t$.  Therefore there are two specializations $t \leadsto s$.
\end{example}

%%%%%%%%%%%%%%%%%%%%%%%%%%%%%%%%%%%%%%%%%%%%%%%%%%%%%%

%%%%%%%%%%%%%%%%%%%%%%%%%%%%%%%%%%%%%%%%%%%%%%%%%%%%%%

%%%%%%%%%%%%%%%%%%%%%%%%%%%%%%%%%%%%%%%%%%%%%%%%%%%%%%

\bibliographystyle{amsalpha}
\bibliography{biblio}{}

\newcommand{\etalchar}[1]{$^{#1}$}
\providecommand{\bysame}{\leavevmode\hbox to3em{\hrulefill}\thinspace}
\providecommand{\MR}{\relax\ifhmode\unskip\space\fi MR }
% \MRhref is called by the amsart/book/proc definition of \MR.
\providecommand{\MRhref}[2]{%
  \href{http://www.ams.org/mathscinet-getitem?mr=#1}{#2}
}
\providecommand{\href}[2]{#2}
\begin{thebibliography}{KKMSD73}

\bibitem[AB15]{AminiBaker_metrizedcurvecomplexes}
Omid Amini and Matthew Baker, \emph{Linear series on metrized complexes of
  algebraic curves}, Math. Ann. \textbf{362} (2015), no.~1-2, 55--106.

\bibitem[ACG{\etalchar{+}}13]{Abramovichetal_log&moduli}
Dan Abramovich, Qile Chen, Danny Gillam, Yuhao Huang, Martin Olsson, Matthew
  Satriano, and Shenghao Sun, \emph{Logarithmic geometry and moduli}, Handbook
  of moduli. {V}ol. {I}, Adv. Lect. Math. (ALM), vol.~24, Int. Press,
  Somerville, MA, 2013, pp.~1--61.

\bibitem[ACM{\etalchar{+}}16a]{Abramovichetal_logsurvey}
Dan Abramovich, Qile Chen, Steffen Marcus, Martin Ulirsch, and Jonathan Wise,
  \emph{Skeletons and fans of logarithmic structures}, Non-Archimedean and
  Tropical Geometry, Simons Symposia, Springer International Publishing, 2016,
  pp.~287--336.

\bibitem[ACM{\etalchar{+}}16b]{AbramovichChenMarcusUlirschWise}
\bysame, \emph{Skeletons and fans of logarithmic structures}, Nonarchimedean
  and Tropical Geometry (Cham) (Matthew Baker and Sam Payne, eds.), Springer
  International Publishing, 2016, pp.~287--336.

\bibitem[ACMW17]{AbramovichChenMarcusWise_boundedness}
Dan Abramovich, Qile Chen, Steffen Marcus, and Jonathan Wise, \emph{Boundedness
  of the space of stable logarithmic maps}, J. Eur. Math. Soc. (JEMS)
  \textbf{19} (2017), no.~9, 2783--2809.

\bibitem[ACP15]{AbramovichCaporasoPayne_tropicalmoduli}
Dan Abramovich, Lucia Caporaso, and Sam Payne, \emph{The tropicalization of the
  moduli space of curves}, Ann. Sci. \'Ec. Norm. Sup\'er. (4) \textbf{48}
  (2015), no.~4, 765--809.

\bibitem[ACV03]{AbramovichCortiVistoli}
Dan Abramovich, Alessio Corti, and Angelo Vistoli, \emph{Twisted bundles and
  admissible covers}, Comm. Algebra \textbf{31} (2003), no.~8, 3547--3618,
  Special issue in honor of Steven L. Kleiman.

\bibitem[AK00]{AbramovichKaru}
D.~Abramovich and K.~Karu, \emph{Weak semistable reduction in characteristic
  0}, Invent. Math. \textbf{139} (2000), no.~2, 241--273. \MR{1738451}

\bibitem[AMW14]{AbramovichMarcusWise_comparison}
Dan Abramovich, Steffen Marcus, and Jonathan Wise, \emph{Comparison theorems
  for {G}romov-{W}itten invariants of smooth pairs and of degenerations}, Ann.
  Inst. Fourier (Grenoble) \textbf{64} (2014), no.~4, 1611--1667.

\bibitem[AW18]{AbramovichWise_invariance}
Dan Abramovich and Jonathan Wise, \emph{Birational invariance in logarithmic
  {G}romov-{W}itten theory}, Compos. Math. \textbf{154} (2018), no.~3,
  595--620.

\bibitem[Ber90]{Berkovich_book}
Vladimir~G. Berkovich, \emph{Spectral theory and analytic geometry over
  non-{A}rchimedean fields}, Mathematical Surveys and Monographs, vol.~33,
  American Mathematical Society, Providence, RI, 1990.

\bibitem[Ber93]{Berkovich_etalecoho}
\bysame, \emph{\'{E}tale cohomology for non-{A}rchimedean analytic spaces},
  Inst. Hautes \'Etudes Sci. Publ. Math. (1993), no.~78, 5--161 (1994).

\bibitem[BGNX12]{BehrendGinotNoohiXu_stringtopology}
Kai Behrend, Gr{\'e}gory Ginot, Behrang Noohi, and Ping Xu, \emph{String
  topology for stacks}, Ast\'erisque (2012), no.~343, xiv+169.

\bibitem[BMV11]{BrannettiMeloViviani_tropicalTorelli}
Silvia Brannetti, Margarida Melo, and Filippo Viviani, \emph{On the tropical
  {T}orelli map}, Adv. Math. \textbf{226} (2011), no.~3, 2546--2586.

\bibitem[Cap13]{Caporaso_tropicalmoduli}
Lucia Caporaso, \emph{Algebraic and tropical curves: comparing their moduli
  spaces}, Handbook of moduli. {V}ol. {I}, Adv. Lect. Math. (ALM), vol.~24,
  Int. Press, Somerville, MA, 2013, pp.~119--160.

\bibitem[Cap16]{Caporaso_tropicalmoduliII}
\bysame, \emph{Tropical methods in the moduli theory of algebraic curves},
  arXiv:1606.00323 [math] (2016).

\bibitem[CFGP19]{ChanFaberGalatiusPayne}
Melody Chan, Carel Faber, Soren Galatius, and Sam Payne, \emph{The
  {S}\_n-equivariant top weight {Euler} characteristic of {M}\_\{g,n\}},
  arXiv:1904.06367 [math] (2019).

\bibitem[CGP18]{ChanGalatiusPayne_tropicalmoduliII}
Melody Chan, Soren Galatius, and Sam Payne, \emph{Tropical curves, graph
  homology, and top weight cohomology of {M}\_g}, arXiv:1805.10186 [math]
  (2018).

\bibitem[CGP19]{ChanGalatiusPayne_Mgn}
\bysame, \emph{Topology of moduli spaces of tropical curves with marked
  points}, arXiv:1903.07187 [math] (2019).

\bibitem[Cha12]{Chan_tropicalTorelli}
Melody Chan, \emph{Combinatorics of the tropical {T}orelli map}, Algebra Number
  Theory \textbf{6} (2012), no.~6, 1133--1169.

\bibitem[Cha15]{Chan_topologyM02}
\bysame, \emph{Topology of the tropical moduli spaces \${M}\_\{2,n\}\$},
  arXiv:1507.03878 [math] (2015).

\bibitem[Cha16]{Chan_lectures}
\bysame, \emph{Lectures on tropical curves and their moduli spaces},
  arXiv:1606.02778 [math] (2016), Proceedings of the School on Moduli of
  Curves, Guanajuato, Lecture Notes of the Unione Matematica Italiana,
  Springer-UMI, to appear.

\bibitem[CHMR16]{CavalieriHampeMarkwigRanganathan_tropicalHassett}
Renzo Cavalieri, Simon Hampe, Hannah Markwig, and Dhruv Ranganathan,
  \emph{Moduli spaces of rational weighted stable curves and tropical
  geometry}, Forum Math. Sigma \textbf{4} (2016), e9, 35.

\bibitem[CMR16]{CavalieriMarkwigRanganathan_tropicalHurwitz}
Renzo Cavalieri, Hannah Markwig, and Dhruv Ranganathan, \emph{Tropicalizing the
  space of admissible covers}, Math. Ann. \textbf{364} (2016), no.~3-4,
  1275--1313.

\bibitem[CMV13]{ChanMeloViviani_tropicalmoduli}
Melody Chan, Margarida Melo, and Filippo Viviani, \emph{Tropical
  {T}eichm\"uller and {S}iegel spaces}, Algebraic and combinatorial aspects of
  tropical geometry, Contemp. Math., vol. 589, Amer. Math. Soc., Providence,
  RI, 2013, pp.~45--85.

\bibitem[CT19]{ConradTemkin_descent}
Brian Conrad and Michael Temkin, \emph{Descent for non-archimedean analytic
  spaces}, arXiv:1912.06230 [math] (2019).

\bibitem[CV10]{CaporasoViviani_tropicalTorelli}
Lucia Caporaso and Filippo Viviani, \emph{Torelli theorem for graphs and
  tropical curves}, Duke Math. J. \textbf{153} (2010), no.~1, 129--171.

\bibitem[DM69]{DeligneMumford_moduliofcurves}
P.~Deligne and D.~Mumford, \emph{The irreducibility of the space of curves of
  given genus}, Inst. Hautes \'Etudes Sci. Publ. Math. (1969), no.~36, 75--109.

\bibitem[FH13]{FrancoisHampe_universalfamilies}
Georges Francois and Simon Hampe, \emph{Universal families of rational tropical
  curves}, Canad. J. Math. \textbf{65} (2013), no.~1, 120--148.

\bibitem[FRTU19]{FosterRanganathanTalpoUlirsch_logPic}
Tyler Foster, Dhruv Ranganathan, Mattia Talpo, and Martin Ulirsch,
  \emph{Logarithmic {P}icard groups, chip firing, and the combinatorial rank},
  Math. Z. \textbf{291} (2019), no.~1-2, 313--327.

\bibitem[GG16]{GiansiracusaGiansiracusa_tropicalschemes}
Jeffrey Giansiracusa and Noah Giansiracusa, \emph{Equations of tropical
  varieties}, Duke Math. J. \textbf{165} (2016), no.~18, 3379--3433.

\bibitem[GKM09]{GathmannKerberMarkwig_tropicalfans}
Andreas Gathmann, Michael Kerber, and Hannah Markwig, \emph{Tropical fans and
  the moduli spaces of tropical curves}, Compos. Math. \textbf{145} (2009),
  no.~1, 173--195.

\bibitem[GM08]{GathmannMarkwig_tropicalKontsevich}
Andreas Gathmann and Hannah Markwig, \emph{Kontsevich's formula and the {WDVV}
  equations in tropical geometry}, Adv. Math. \textbf{217} (2008), no.~2,
  537--560.

\bibitem[Gro11]{Gross_book}
Mark Gross, \emph{Tropical geometry and mirror symmetry}, CBMS Regional
  Conference Series in Mathematics, vol. 114, Published for the Conference
  Board of the Mathematical Sciences, Washington, DC; by the American
  Mathematical Society, Providence, RI, 2011.

\bibitem[Gro18]{Gross_toroidalintersectiontheory}
Andreas Gross, \emph{Intersection theory on tropicalizations of toroidal
  embeddings}, Proc. Lond. Math. Soc. (3) \textbf{116} (2018), no.~6,
  1365--1405.

\bibitem[GS15]{GeraschenkoSatriano_toricstacksI}
Anton Geraschenko and Matthew Satriano, \emph{Toric stacks {I}: {T}he theory of
  stacky fans}, Trans. Amer. Math. Soc. \textbf{367} (2015), no.~2, 1033--1071.

\bibitem[HMU19]{HuszarMarcusUlirsch_troplogclutch&glue}
Alana Huszar, Steffen Marcus, and Martin Ulirsch, \emph{Clutching and gluing in
  tropical and logarithmic geometry}, J. Pure Appl. Algebra \textbf{223}
  (2019), no.~5, 2036--2061.

\bibitem[Kat89]{Kato_logstr}
Kazuya Kato, \emph{Logarithmic structures of {F}ontaine-{I}llusie}, Algebraic
  analysis, geometry, and number theory ({B}altimore, {MD}, 1988), Johns
  Hopkins Univ. Press, Baltimore, MD, 1989, pp.~191--224.

\bibitem[Kat94]{Kato_toricsing}
\bysame, \emph{Toric singularities}, Amer. J. Math. \textbf{116} (1994), no.~5,
  1073--1099.

\bibitem[Kat99]{Kato_LogMod}
Fumiharu Kato, \emph{Exactness, integrality, and log modifications}, July 1999,
  \href{https://arxiv.org/abs/math/9907124}{\texttt{arXiv:9907124}}.

\bibitem[Kat00]{Kato_logsmoothcurves}
\bysame, \emph{Log smooth deformation and moduli of log smooth curves},
  Internat. J. Math. \textbf{11} (2000), no.~2, 215--232.

\bibitem[KKMSD73]{KKMSD_toroidal}
G.~Kempf, Finn~Faye Knudsen, D.~Mumford, and B.~Saint-Donat, \emph{Toroidal
  embeddings. {I}}, Lecture Notes in Mathematics, Vol. 339, Springer-Verlag,
  Berlin, 1973.

\bibitem[KN99]{KatoNakayama_rounding}
Kazuya Kato and Chikara Nakayama, \emph{Log {B}etti cohomology, log \'etale
  cohomology, and log de {R}ham cohomology of log schemes over {${\bf C}$}},
  Kodai Math. J. \textbf{22} (1999), no.~2, 161--186.

\bibitem[Knu83]{Knudsen_projectivityII}
Finn~F. Knudsen, \emph{The projectivity of the moduli space of stable curves.
  {II}. {T}he stacks {$M_{g,n}$}}, Math. Scand. \textbf{52} (1983), no.~2,
  161--199.

\bibitem[Lor15]{Lorscheid_tropicalschemes}
Oliver Lorscheid, \emph{Scheme theoretic tropicalization}, arXiv:1508.07949
  [math] (2015).

\bibitem[Met03]{Metzler_top&smoothstacks}
David Metzler, \emph{Topological and {Smooth} {Stacks}}, arXiv:math/0306176
  (2003).

\bibitem[Mik06]{Mikhalkin_ICM}
Grigory Mikhalkin, \emph{Tropical geometry and its applications}, International
  {C}ongress of {M}athematicians. {V}ol. {II}, Eur. Math. Soc., Z\"urich, 2006,
  pp.~827--852.

\bibitem[Mik07]{Mikhalkin_Gokova}
\bysame, \emph{Moduli spaces of rational tropical curves}, Proceedings of
  {G}\"okova {G}eometry-{T}opology {C}onference 2006, G\"okova
  Geometry/Topology Conference (GGT), G\"okova, 2007, pp.~39--51.

\bibitem[Mol]{Molcho}
Samouil Molcho, \emph{Universal stacky semistable resolution}.

\bibitem[MR14]{MacLaganRincon_tropicalschemes}
Diane Maclagan and Felipe Rinc{\'o}n, \emph{Tropical schemes, tropical cycles,
  and valuated matroids}, arXiv:1401.4654 [math] (2014).

\bibitem[MR18]{MacLaganRincon_tropicalideals}
Diane Maclagan and Felipe Rinc\'{o}n, \emph{Tropical ideals}, Compos. Math.
  \textbf{154} (2018), no.~3, 640--670.

\bibitem[MW18]{MolchoWise}
Samouil Molcho and Jonathan Wise, \emph{The logarithmic {Picard} group and its
  tropicalization}, arXiv:1807.11364 [math] (2018).

\bibitem[Noo05]{Noohi_topstacksI}
Behrang Noohi, \emph{Foundations of {Topological} {Stacks} {I}},
  arXiv:math/0503247 (2005).

\bibitem[NS06]{NishinouSiebert_correspondence}
Takeo Nishinou and Bernd Siebert, \emph{Toric degenerations of toric varieties
  and tropical curves}, Duke Math. J. \textbf{135} (2006), no.~1, 1--51.

\bibitem[Ols03]{Olsson_LOG}
Martin~C. Olsson, \emph{Logarithmic geometry and algebraic stacks}, Ann. Sci.
  \'Ecole Norm. Sup. (4) \textbf{36} (2003), no.~5, 747--791.

\bibitem[PY16]{PortaYu_higherGAGA}
Mauro Porta and Tony~Yue Yu, \emph{Higher analytic stacks and {GAGA} theorems},
  Adv. Math. \textbf{302} (2016), 351--409.

\bibitem[Ran17]{Ranganathan_ratcurtorvar&nonArch}
Dhruv Ranganathan, \emph{Skeletons of stable maps {I}: rational curves in toric
  varieties}, J. Lond. Math. Soc. (2) \textbf{95} (2017), no.~3, 804--832.

\bibitem[RG71]{RaynaudGruson}
Michel Raynaud and Laurent Gruson, \emph{Crit\`eres de platitude et de
  projectivit\'{e}. {T}echniques de ``platification'' d'un module}, Invent.
  Math. \textbf{13} (1971), 1--89. \MR{308104}

\bibitem[RGaS99]{RG}
J.~C. Rosales and P.~A. Garc\'\i~a S\'anchez, \emph{Finitely generated
  commutative monoids}, Nova Science Publishers, Inc., Commack, NY, 1999.

\bibitem[RSPW17a]{RanganathanSantosParkerWiseI}
Dhruv Ranganathan, Keli Santos-Parker, and Jonathan Wise, \emph{Moduli of
  stable maps in genus one and logarithmic geometry {I}}, arXiv:1708.02359
  [math] (2017).

\bibitem[RSPW17b]{RanganathanSantosParkerWiseII}
\bysame, \emph{Moduli of stable maps in genus one and logarithmic geometry
  {II}}, arXiv:1709.00490 [math] (2017).

\bibitem[Ser77]{Serre_trees}
Jean-Pierre Serre, \emph{Arbres, amalgames, {${\rm SL}_{2}$}}, Soci\'et\'e
  Math\'ematique de France, Paris, 1977, Avec un sommaire anglais,
  R{\'e}dig{\'e} avec la collaboration de Hyman Bass, Ast{\'e}risque, No. 46.

\bibitem[Sim96]{Simpson_nstacks}
Carlos Simpson, \emph{Algebraic (geometric) $n$-stacks}, arXiv:alg-geom/9609014
  (1996).

\bibitem[{Sta}16]{stacks-project}
The {Stacks Project Authors}, \emph{\itshape stacks project},
  \url{http://stacks.math.columbia.edu}, 2016.

\bibitem[Thu07]{Thuillier_toroidal}
Amaury Thuillier, \emph{G\'eom\'etrie toro\"\i dale et g\'eom\'etrie analytique
  non archim\'edienne. {A}pplication au type d'homotopie de certains sch\'emas
  formels}, Manuscripta Math. \textbf{123} (2007), no.~4, 381--451.

\bibitem[TV08a]{ToenVaquie_algebrisation}
Bertrand To{\"e}n and Michel Vaqui{\'e}, \emph{Alg\'ebrisation des vari\'et\'es
  analytiques complexes et cat\'egories d\'eriv\'ees}, Math. Ann. \textbf{342}
  (2008), no.~4, 789--831.

\bibitem[TV08b]{ToenVezzosi_HAGII}
Bertrand To{\"e}n and Gabriele Vezzosi, \emph{Homotopical algebraic geometry.
  {II}. {G}eometric stacks and applications}, Mem. Amer. Math. Soc.
  \textbf{193} (2008), no.~902, x+224.

\bibitem[Uli15]{Ulirsch_tropicalHassett}
Martin Ulirsch, \emph{Tropical geometry of moduli spaces of weighted stable
  curves}, J. Lond. Math. Soc. (2) \textbf{92} (2015), no.~2, 427--450.

\bibitem[Uli17a]{Ulirsch_functroplogsch}
\bysame, \emph{Functorial tropicalization of logarithmic schemes: the case of
  constant coefficients}, Proc. Lond. Math. Soc. (3) \textbf{114} (2017),
  no.~6, 1081--1113.

\bibitem[Uli17b]{Ulirsch_trop=quot}
\bysame, \emph{Tropicalization is a non-{A}rchimedean analytic stack quotient},
  Math. Res. Lett. \textbf{24} (2017), no.~4, 1205--1237.

\bibitem[Uli19]{Ulirsch_nonArchArtin}
\bysame, \emph{Non-{A}rchimedean geometry of {A}rtin fans}, Adv. Math.
  \textbf{345} (2019), 346--381.

\bibitem[Vis89]{Vistoli_intersectiontheorystacks}
Angelo Vistoli, \emph{Intersection theory on algebraic stacks and on their
  moduli spaces}, Invent. Math. \textbf{97} (1989), no.~3, 613--670.

\bibitem[Viv13]{Viviani_tropvscompTorelli}
Filippo Viviani, \emph{Tropicalizing vs. compactifying the {T}orelli morphism},
  Tropical and non-{A}rchimedean geometry, Contemp. Math., vol. 605, Amer.
  Math. Soc., Providence, RI, 2013, pp.~181--210.

\bibitem[Yu15]{Yu_tropstablemaps}
Tony~Yue Yu, \emph{Tropicalization of the moduli space of stable maps}, Math.
  Z. \textbf{281} (2015), no.~3-4, 1035--1059.

\bibitem[Yu18]{Yu_Gromovcompactness}
\bysame, \emph{Gromov compactness in non-archimedean analytic geometry}, J.
  Reine Angew. Math. \textbf{741} (2018), 179--210.

\end{thebibliography}

\end{document}